%% file: matching-complex.tex
\documentclass[final,nomarks]{dmtcs-episciences}
\usepackage[utf8]{inputenc}
\usepackage{subfigure}
\usepackage[round]{natbib}

\usepackage[title]{appendix}

\usepackage{amsmath, amsthm, amssymb,mathtools,enumerate}
\usepackage[mathscr]{eucal}

\usepackage{color,soul}
\usepackage{hyperref}

\usepackage{tabularray}
\usepackage{xfrac}

\usepackage{tikz}
\usepackage{tikz-cd}
\usetikzlibrary{graphs}
\usetikzlibrary{graphs.standard}
\usetikzlibrary{arrows.meta}
\usepackage{calc}
\usepackage{float}
\usetikzlibrary{calc}
\tikzstyle{vertex}=[circle, draw, fill=black, inner sep=0pt, minimum size=6pt]
\tikzstyle{smlv}=[circle, draw, fill=black, inner sep=0pt, minimum size=2pt]
\tikzstyle{grvert}=[circle, draw, gray, fill=gray,  inner sep=0pt, minimum size=6pt]
\newcommand{\vertex}{\node[vertex]}
\newcommand{\smlv}{\node[smlv]}

\DeclarePairedDelimiter{\ceil}{\lceil}{\rceil}
\DeclarePairedDelimiter{\floor}{\lfloor}{\rfloor}

\newtheorem{theorem}{Theorem}
\newtheorem{proposition}[theorem]{Proposition}
\newtheorem{obs}[theorem]{Observation}
\newtheorem{conjecture}[theorem]{Conjecture}
\newtheorem{problem}[theorem]{Problem}

\theoremstyle{definition}
\newtheorem{definition}[theorem]{Definition}

\theoremstyle{remark}
\newtheorem{remark}[theorem]{Remark}
\newtheorem{example}[theorem]{Example}

\def\U{\mathscr{U}}
\def\M{\mathcal{M}}
\def\K{\mathcal{K}}
\def\V{\mathscr{V}}
\def\Z{\integers}

\numberwithin{theorem}{section}
\numberwithin{equation}{section}
\theoremstyle{plain}

\author{Anupam Mondal \and Sajal Mukherjee \and Kuldeep Saha}
\title[Matching complexes via discrete Morse theory]{Topology of matching complexes of complete graphs via discrete Morse theory}

\affiliation{Institute for Advancing Intelligence (IAI), TCG CREST, Kolkata, India\\
	Academy of Scientific \& Innovative Research (AcSIR), Ghaziabad, India}

\keywords{discrete Morse theory, complete graph, matching, abstract simplicial complex, gradient vector field, Morse homology}

\begin{document}
	\publicationdata{vol. 26:3}{2024}{13}{10.46298/dmtcs.12887}{2024-01-18; 2024-01-18; 2024-06-18; 2024-09-11}{2024-09-18}
	
\maketitle
	
	\begin{abstract}
	\hfill

		Bouc (1992) first studied the topological properties of $M_n$, the matching complex of the complete graph of order $n$, in connection with Brown complexes and Quillen complexes. Bj\"{o}rner et~al.\ (1994) showed that $M_n$ is homotopically $(\nu_n-1)$-connected, where $\nu_n=\lfloor{\frac{n+1}{3}}\rfloor-1$, and conjectured that this connectivity bound is sharp. Shareshian and Wachs (2007) settled the conjecture by inductively showing that the $\nu_n$-dimensional homology group of $M_n$ is nontrivial, with Bouc's calculation of $H_1(M_7)$ serving as the pivotal base step. In general, the topology of $M_n$ is not very well-understood, even for a small $n$. In the present article, we look into the topology of $M_n$, and $M_7$ in particular, in the light of discrete Morse theory as developed by Forman (1998). We first construct a gradient vector field on $M_n$ (for $n \ge 5$) that doesn't admit any critical simplices of dimension up to $\nu_n-1$, except one unavoidable $0$-simplex, which also leads to the aforementioned $(\nu_n-1)$-connectedness of $M_n$ in a purely combinatorial way. However, for an efficient homology computation by discrete Morse theoretic techniques, we are required to work with a gradient vector field that admits a low number of critical simplices, and also allows an efficient enumeration of gradient paths. An optimal gradient vector field is one with the least number of critical simplices, but the problem of finding an optimal gradient vector field, in general, is an \textsf{NP-hard} problem (even for $2$-dimensional complexes). We improve the gradient vector field constructed on $M_7$ in particular to a much more efficient (near-optimal) one, and then with the help of this improved gradient vector field, compute the homology groups of $M_7$ in an efficient and algorithmic manner. We also augment this near-optimal gradient vector field to one that we conjecture to be optimal.
	\end{abstract}
	
	{\let\thefootnote\relax\footnotetext{2020 \textit{Mathematics Subject Classification (MSC2020)}. 57Q70 (primary), 05C70, 05E45.}}

	\section{Introduction}
	The collection of all \emph{matchings} or independent edge sets (i.e., sets of edges without common endvertices) in a graph constitutes an abstract simplicial complex, called the \emph{matching complex} of the graph. The matching complex of a graph $G$ may also be realized as the \emph{independence complex} of the \emph{line graph} of $G$. In particular, we denote the matching complex of the complete graph of order $n$ (see \cite{blvz,bouc,shareshian,wachs}) by $M_n$. Topological properties of $M_n$ were first studied by \cite{bouc}, in connection with \emph{Brown complexes} (\cite{brown74,brown75}) and \emph{Quillen complexes} (\cite{quillen}). \cite{blvz} proved the following result regarding the homotopical connectivity of $M_n$.
	\begin{theorem}\label{homcon}
		For all $n$, the matching complex $M_n$ is homotopically $(\nu_n-1)$-connected, where $\nu_n=\floor{\frac{n+1}{3}}-1$.
	\end{theorem}
	They also conjectured that the connectivity bounds of Theorem~\ref{homcon} are sharp. \cite{shareshian} settled the conjecture by inductively showing that the $\nu_n$-dimensional homology group of $M_n$ is nontrivial (see also \cite{wachs}). The crucial base step was Bouc's ``hand" calculation of the first homology group of $M_7$. Here we remark that the topology of $M_n$ is not very well-understood even for a small $n$. For example, the first nontrivial reduced homology group of $M_{14}$ is not known, but \cite{jonsson} showed that, rather surprisingly, $M_{14}$ is the only matching complex of a complete graph, for which the $\nu_n$-dimensional homology group has torsion other than $3$-torsion (see also \cite{shareshian}). Also, very little is known about the higher dimensional homology groups of $M_n$ in general.
	
	\cite{forman} developed discrete Morse theory as a combinatorial analogue of (smooth) Morse theory (see also \cite{chari,forman02,knudson,kozlov,scoville}). Over the years, the theory has turned out to be immensely useful in diverse fields of theoretical and applied mathematics, and also in computer science. The central notion of this theory is that of a \emph{discrete Morse function} defined on a finite (abstract) simplicial complex (or a (regular) CW complex). It helps us understand the topology of the complex through an efficient cell decomposition (i.e., one with fewer cells than in the original decomposition) of the complex. In practice, however, instead of such functions, we usually consider an equivalent and more useful notion of (discrete) \emph{gradient vector fields} on the complex. The homotopy type of the complex is determined by only the simplices (or cells) that are \emph{critical} with respect to an assigned gradient vector field. Discrete Morse theoretic techniques also help us compute the homology groups, Betti numbers, etc.\ of a complex in a computationally efficient way, provided one manages to construct a sufficiently ``good" gradient vector field on it, in the first place.
	
	In Section~\ref{sec-cons} of this article, we first prove the following by explicitly constructing a gradient vector field on $M_n$.
	\begin{theorem}\label{nolowcrit}
		There is a gradient vector field on $M_n$ (for $n\ge 5$) with respect to which there are no critical simplices of dimension up to $\nu_n -1$, except one $0$-simplex.
	\end{theorem}
	Theorem~\ref{homcon} follows as a natural consequence of the above. Existence of a gradient vector field with the same property was previously shown by \cite{shareshian} (Section~9), with the help of the fact that the $\nu_n$-skeleton of $M_n$ is \emph{vertex decomposable}, as established by \cite{athan}. However, the task of homology computation via discrete Morse theory relies on an efficient enumeration of all critical simplices and all possible \emph{gradient paths} of a specific type. Thus, in order to compute the homology groups efficiently, we are required to work with a gradient vector field that admits a low number of critical simplices of each dimension to begin with. An \emph{optimal gradient vector field} is one with the least number of critical simplices, but the problem of finding an optimal gradient vector field, in general, is an \textsf{NP-hard} problem (even for $2$-dimensional complexes) (\cite{egecioglu,joswig,lewiner,lewiner2}). In this context, we mention that \cite{adi}, \cite{benedetti2}, \cite{benedetti1} described a scheme for searching optimal discrete gradient vector fields with a random heuristic, which turned out to be successful, even in some cases with a large input size.
	
	An efficient gradient vector field reduces the task of computing homology to the computation of (discrete) \emph{Morse homology} groups, which are homology groups of a relatively simpler chain complex. Bouc's computation of the first homology group of $M_7$, which, as mentioned before, is the pivotal base step to determine the connectivity of $M_n$ in general, relies on some ingenious, but somewhat ad hoc tricks. In Section~\ref{hom-m7}, we present an efficient and algorithmic computation of the homology groups of $M_7$. We extend our previously constructed gradient vector field on $M_7$ to a \emph{near-optimal} one, and use it to compute the Morse homology groups of $M_7$ in Subsection~\ref{morse-hom}, and obtain the following.
	\begin{theorem}\label{m7-hom}	
		The nontrivial discrete Morse homology groups of $M_7$ are the following:
		\[H_0(M_7) = \Z, H_1(M_7) = \Z_3, \text{ and } H_2(M_7) = \Z^{20}.\]		
	\end{theorem}
	In Subsection~\ref{opt-cons}, we augment the near-optimal gradient vector field on $M_7$ even further and get the following.
	\begin{theorem} \label{augmnt}
		There is a gradient vector field on $M_7$ with respect to which there are 22 critical $2$-simplices, two critical $1$-simplices, and one critical $0$-simplex.
	\end{theorem}
	It follows from Theorem~\ref{m7-hom} that with respect to any gradient vector field on $M_7$, there is at least one critical $0$-simplex, and there is at least one critical $1$-simplex. Also, there are at least 21 critical $2$-simplices (as $H_1(M_7)$ has a torsion). This naturally raises the question whether these lower bounds are sharp. However, we believe these bounds are not attainable, and thus a gradient vector field on $M_7$ that satisfies the requirements of Theorem~\ref{augmnt} is indeed an optimal one. We pose this as a conjecture (Conjecture~\ref{conj}) in the Conclusion section.
	
	\section{Preliminaries}
	\subsection{Basics of combinatorics and graph theory}\label{cgt}
	An \emph{abstract simplicial complex} (or simply, a complex) is a (finite, nonempty) collection, say $\K$, of finite sets with the property that if $\sigma \in \K$ and $\tau \subseteq \sigma$, then $\tau \in \K$. We note that the empty set is always in $\K$. If $\sigma \in \K$, then $\sigma$ is called a \emph{simplex} or a \emph{face} of $\K$. If the simplex $\sigma$ is a set of cardinality $d+1$, then the dimension of $\sigma$ is $d$, and we call $\sigma$ a $d$-dimensional simplex (or simply, a $d$-simplex). We denote a $d$-simplex $\sigma$ by $\sigma^{(d)}$ whenever necessary. The dimension of a complex $\K$, denoted by $\dim(\K)$, is the largest dimension of its faces. The \emph{vertex set} of a complex $\K$ is defined as $V(\K)=\cup_{\sigma \in \K} \sigma$ (i.e., the collection of all elements in all the faces of $\K$). The elements of $V(\K)$ are called the vertices of the complex $\K$.
	
	Any complex has a unique (up to a homeomorphism) \emph{geometric realization}. However, we don't differentiate between a complex and its geometric realization while discussing its topological properties; it should be understood from the context.
	
	The \emph{$f$-vector} of a complex $\K$ is the integer sequence $(f_0,\ldots,f_{\dim(\K)})$, where $f_i$ is the number of $i$-dimensional faces of $\K$. The \emph{Euler characteristic} of $\K$, denoted by $\chi(\K)$, is given by 
	\[\chi(\K) = \sum_{i=0}^{\dim(\K)} (-1)^i f_i.\]
	
	A (simple, finite, undirected) \emph{graph} $G$ is an ordered pair of (disjoint) finite sets $(V(G), E(G))$, where $E(G) \subseteq \{e \subseteq V(G) : |e|=2\}$. The sets $V(G)$ and $E(G)$ are called the \emph{vertex set} and the \emph{edge set} of the graph $G$, respectively. We call an element of $V(G)$ a \emph{vertex} of $G$ and an element of $E(G)$ an \emph{edge} in $G$. If $e = \{x,y\}$ is an edge, then $x$ and $y$ are \emph{endvertices} of the edge $e$. Also, we denote the edge $e = \{x,y\}$ by $xy$ (or, $yx$) for the sake of brevity. We note that a graph can be viewed as an abstract simplicial complex of dimension 1 or less (with the vertices and the edges being considered as $0$-simplices and $1$-simplices, respectively), and vice versa.
	
	A \emph{complete graph} with $n$ ($\ge 1$) vertices is a graph $G$ with $|V(G)|=n$ and $E(G) = \{e \subseteq V(G) : |e|=2\}$. This graph is unique (up to a graph isomorphism), and we denote it by $K_n$.
	
	A \emph{matching} $\alpha$ in a graph $G$ is a set of edges with the property that if the edges $e_1$ and $e_2$ are in $\alpha$, then $e_1 \cap e_2 = \emptyset$ (i.e., $e_1$ and $e_2$ don't share an endvertex). If the vertex $x$ is an endvertex of an edge in a matching $\alpha$, then we say $\alpha$ \emph{covers} or \emph{matches} $x$, and otherwise, $x$ is \emph{uncovered} or \emph{unmatched} by $\alpha$. Moreover, we say the matching $\alpha$ matches $x$ with another vertex $y$ if the edge $xy$ is in $\alpha$. A \emph{perfect matching} is one that covers all the vertices of the graph.
	
	We note that the collection of all matchings in a graph $G$ is an abstract simplicial complex, and we call it the \emph{matching complex} of $G$. In particular, for all $n$, we denote the matching complex of $K_n$ by $M_n$. We may verify that for all $n$, $\dim(M_n) = \floor{\frac{n}{2}} -1$, and if $(f_0,\ldots,f_{\floor{\frac{n}{2}}-1})$ is the $f$-vector of $M_n$, then for all $i \in \{0,\ldots,\floor{\frac{n}{2}}-1\}$,
	\begin{equation}\label{fv}
		f_i= \frac{(i+1)!}{2^{i+1}} {n\choose {2 (i+1)}}{{2(i+1)}\choose {i+1}}.
	\end{equation}
	
	\begin{example}
		Let $V(K_n)=\{1,\ldots,n\}$. We describe (the topology of) $M_n$, for all $n$ up to 8.
		\begin{enumerate}
			\item $M_1 = \{\emptyset\}$.
			\item $M_2 = \{\emptyset, \{12\}\}$, which is a single point.
			\item $M_3 = \{\emptyset, \{12\}, \{13\}, \{23\}\}$, which is the space with three distinct points.
			\item The maximal simplices of $M_4$ are $\{12,34\}$, $\{13,24\}$, and $\{14,23\}$. Thus, $M_4$ is a space with three mutually disjoint $1$-simplices, which is homotopy equivalent to the space with three distinct points.
			\item As $\dim(M_5)=1$, the complex $M_5$ can be viewed as a (connected) graph with the vertex set $\{ij: i,j\in [5], i<j\}$ and the edge set $\{\{ij,k\ell\} : i,j,k,\ell \in [5], i<j, k<\ell, \{i,j\} \cap \{k,\ell\} = \emptyset\}$. This is the well-known Petersen graph as shown in Fig.~\ref{petersen}.
		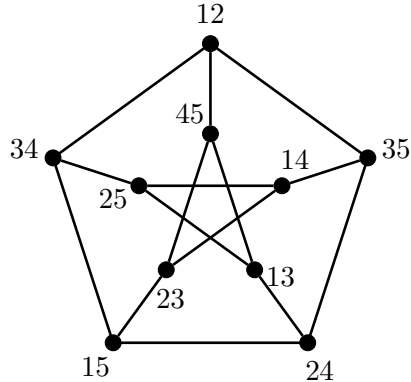
\begin{figure}[!ht]
			\centering
			\begin{tikzpicture}
				\foreach \a in {1,...,5}
				{\node[vertex] (u\a) at ({\a*72+18}:2.2){};}
				
				\node at (18:2.6) {35};
				\node at (90:2.6) {12};
				\node at (162:2.6) {34};
				\node at (234:2.6) {15};
				\node at (304:2.6) {24};
				
				\foreach \a in {1,...,5}
				{\node[vertex] (v\a) at ({\a*72+18}:1){};}
				
				\node at (30:1.3) {14};
				\node at (102:1.3) {45};
				\node at (174:1.3) {25};
				\node at (246:1.3) {23};
				\node at (316:1.3) {13};
				
				\foreach \a in {18, 90, ..., 306}
				{\draw[line width=1pt] (\a:2.2) -- (\a+72:2.2);}
				
				\path [line width=1pt]
				(v1) edge (v3)
				(v3) edge (v5)
				(v5) edge (v2)
				(v2) edge (v4)
				(v4) edge (v1)
				;
				
				\foreach \a in {1,...,5}
				{\draw[line width=1pt] (u\a) -- (v\a);}
			\end{tikzpicture}
			\caption{The matching complex $M_5$ is the Petersen graph.}\label{petersen}
		\end{figure}
	
		Any connected graph $G$ is homotopy equivalent to a wedge of $k$ circles, where $k$ is the cyclomatic number (also known as the circuit rank or cycle rank) of the graph, which is the number $|E(G)|-|V(G)|+1$. Thus, $M_5$ is homotopy equivalent to the wedge of $15-10+1=6$ circles.
		
		Here we also recall that the cyclomatic number of a connected graph is the same as the first Betti number of the graph when considered as a simplicial complex.
		\item Although $\dim(M_6)=2$, we observe that each $1$-simplex in $M_6$ is contained in exactly one $2$-simplex. By the notion of (elementary) \emph{collapses} in topology \cite[Chapter~1]{cohen}, the complex $M_6$ deformation-retracts to a complex of dimension 1. Thus, $M_6$ is also homotopy equivalent to a (connected) graph, and subsequently, to a wedge of $k$ circles. Here we determine $k$ from the Euler characteristic (which is a homotopy invariant) of $M_6$. If $(f_0,f_1,f_2)$ is the $f$-vector of $M_6$, then from Equation~(\ref{fv}),
		$f_0 =15$, $f_1 =45$, ${f_2= 15}$, and thus $1-k=\chi(M_6) = f_0-f_1+f_2 = -15$. Therefore, $M_6$ is homotopy equivalent to the wedge of 16 circles.
		
		\item $M_7$ is a $2$-dimensional complex, whose first homology group is not torsion-free (\cite{bouc}). The zeroth, first, and second homology groups of $M_7$ are $\mathbb{Z}$, $\mathbb{Z}_3$, and $\mathbb{Z}^{20}$, respectively (see Subsection~\ref{morse-hom}).
		
		\item $M_8$ is homotopy equivalent to the wedge of 132 spheres of dimension 2 (follows from \cite{blvz}; see Example~\ref{m8top} below for a discrete Morse theoretic approach).
		\end{enumerate}
	\end{example}

	\subsection{Simplicial homology}
	First, we need to introduce the notion of an orientation of a simplex. An orientation of a simplex is given by an ordering of its vertices, with two orderings defining the same orientation if and only if they differ by an even permutation. We denote an oriented $k$-simplex consisting of the vertices $x_0$, $x_1,\ldots,$ $x_k$, with the orientation given by the increasing ordering of the indices, by $[x_0$, $x_1,\ldots,$ $x_k]$. We usually choose and fix an ordering of the vertices of the complex to begin with, and assign each simplex the orientation corresponding to the induced ordering of its vertices. In other words, if $\sigma=\{x_0,x_1,\ldots,x_k\}$ is a $k$-simplex of a complex $\K$, and $x_0<x_1<\cdots<x_k$ with respect to the chosen order on $V(\K)$, then in order to avoid notational complicacy, we denote the oriented $k$-simplex $[x_0,x_1,\ldots,x_k]$ by $\sigma$ as well.
	
	A $k$-chain in a complex $\K$ is a \emph{finite} formal sum $\sum c_i\sigma_i$, where each $c_i \in \Z$ and each $\sigma_i$ is an oriented $k$-simplex, with the notion that an oriented simplex is equal to the negative of the simplex with the opposite orientation (e.g., $[x_0$, $x_1$, $x_2,\ldots,$ $x_k]=-[x_1$, $x_0$, $x_2,\ldots,$ $x_k]$).
	
	We denote the free abelian group generated by all $k$-simplices of a complex $\K$, i.e., the group of $k$-chains, by $C_k(\K)$. We now define a homomorphism $\partial_k: C_k(\K) \to C_{k-1}(\K)$ called the \emph{boundary operator}. If $\sigma = [x_0,x_1,\ldots,x_k]$, considered as a basis element of $C_k(\K)$, then  
	\[\partial_k(\sigma) \coloneqq \sum_{i=0}^k (-1)^i[x_0,\ldots,\widehat{x_i},\ldots,x_k],\]
	where $[x_0,\ldots,\widehat{x_i},\ldots,x_k]$ is the oriented $(k-1)$-simplex obtained from $\sigma$ after deleting $x_i$ (and with the induced orientation). We then extend $\partial_k$ linearly to all $k$-chains. We also define $\partial_0$ to be the zero map. In $C_k(\K$), the elements of the subgroup $\ker(\partial_k)$ are called \emph{cycles} (more specifically, $k$-cycles), and the elements of the subgroup $\operatorname{im}(\partial_{k+1})$ are called \emph{boundaries} (more specifically, $k$-boundaries).
	
	We may verify that for all $k \ge 1$, and for any $k$-chain $\tau$, we have  $\partial_{k-1}\circ\partial_{k}(\tau)=0$. In other words, $(C_*(\K), \partial_*)$ is a \emph{chain complex}. The $k$-th homology group of $\K$, denoted by $H_k(\K)$, is given by  
	\[H_k(\K) \coloneqq \sfrac{\ker(\partial_k)}{\operatorname{im}(\partial_{k+1})}.\]
	The $k$-th \emph{Betti number} of $\K$ is the free rank (i.e., rank of the torsion-free part) of $H_k(\K)$.
	
	\bigskip
	We refer to the book \emph{Elements of Algebraic Topology} by \cite{munkres} for background on algebraic and combinatorial topology, and the book \emph{Graph Theory} by \cite{diestel} for background on graph theory.

	\subsection{Basics of discrete Morse theory}\label{dmt}
	First, we introduce the notion of a \emph{discrete vector field} and a (discrete) \emph{gradient vector field} on an abstract simplicial complex following \cite{forman,forman02}. 
	\begin{definition}[Discrete vector field]
		A discrete vector field $\V$ on an abstract simplicial complex $\K$ is a collection of ordered pairs of simplices of the form $(\alpha,\beta)$ such that 
		\begin{enumerate}[(i)]
			\item $\alpha \subsetneq \beta$,
			\item dimension of the simplex $\beta$ is 1 more than that of $\alpha$,
			\item each face of $\K$ is in \emph{at most} one pair of $\V$.
		\end{enumerate}
	\end{definition}
	If the simplex $\alpha^{(p)}$ is paired off with the simplex $\beta^{(p+1)}$ in a discrete vector field (i.e., the pair of simplices $(\alpha^{(p)}, \beta^{(p+1)})$ is an element of the discrete vector field), then we denote it by $\alpha \rightarrowtail \beta$ (or $\beta \leftarrowtail \alpha$).

	Given a discrete vector field $\V$ on a simplicial complex $\K$, a $\V$-path is a sequence of simplices 
	\[\alpha_0^{(d)}, \beta_0^{(d+1)}, \alpha_1^{(d)}, \beta_1^{(d+1)}, \ldots, \alpha_k^{(d)}, \beta_k^{(d+1)}, \alpha_{k+1}^{(d)}\]
	such that for each $i \in \{0,\ldots,k\}, (\alpha_i,\beta_i) \in \V$ and $\beta_i \supsetneq \alpha_{i+1} \ne \alpha_i$. We represent such a path diagrammatically as below
	\begin{center}
		\begin{tikzcd}
			\alpha_0 \ar[r,tail] & \beta_0 \ar[r] & \alpha_1 \ar[r,tail] & \beta_1 \ar[r] & \cdots \ar[r] & \alpha_k \ar[r,tail] & \beta_k \ar[r] & \alpha_{k+1}
		\end{tikzcd}
	\end{center}
	(in the diagram above, $\beta^{(d+1)} \to \alpha^{(d)}$ implies $\beta \supsetneq \alpha$). We say such a path is a \emph{nontrivial closed path} if $k \ge 0$ and $\alpha_{k+1} = \alpha_0$.
	\begin{definition}[Gradient vector field]
		A gradient vector field on a simplicial complex $\K$ is a discrete vector field $\V$ on $\K$ which does not admit nontrivial closed $\V$-paths.
	\end{definition}
	For a gradient vector field $\V$, when it is clear from the context, we sometimes call a $\V$-path a \emph{gradient path}.

	Let $\V$ be a gradient vector field on a simplicial complex $\K$. We call a nonempty simplex $\alpha$ a \emph{critical simplex} (with respect to $\V$) if one of the following holds:
	\begin{enumerate}[(i)]
		\item $\alpha$ does not appear in any pair of $\V$, or
		\item $\alpha$ is a $0$-simplex and $(\emptyset,\alpha) \in \V$.
	\end{enumerate}
	We recall that a \emph{CW complex} is a topological space built recursively by gluing cells (which are homeomorphic copies of balls) of increasing dimension. The fundamental theorem of discrete Morse theory is as below.
	\begin{theorem}[{\cite{forman,forman02}}]\label{funda}
		If $\K$ is a simplicial complex and $\V$ is a gradient vector field on $\K$, then $\K$ is homotopy equivalent to a CW complex with exactly one cell of dimension $p$ for each critical simplex (with respect to $\V$) of dimension $p$.
	\end{theorem}
	The following is an important corollary of Theorem~\ref{funda}.
	\begin{theorem}[{\cite{forman02}}]\label{wedge}
		If $\K$ is a simplicial complex and $\V$ is a gradient vector field on $\K$ such that the only critical simplices are one $0$-simplex and $k$ simplices of dimension $d$, then $\K$ is homotopy equivalent to the wedge of $k$ spheres of dimension $d$.
	\end{theorem}
	
	Theorem~\ref{funda} implies that the topological information pertaining to $\K$ would be concise and easier to compute if the number of critical simplices of each dimension, with respect to $\V$, is as low as possible. If $m_i$ is the number of $i$-dimensional critical simplices of $\K$, and $b_i$ is the $i$-th Betti number of $\K$, then we have the following inequalities.
	\begin{theorem}[Morse inequalities, {\cite{forman,forman02}}]\label{morse-ineq}
		\hfill
		\begin{description}
			\item[The weak Morse inequalities:] If $d$ is the dimension of $\K$, then
			\begin{enumerate}[(i)]
				\item for each $i \in \{0,1,\ldots,d\}$, $m_i \ge b_i$,
				\item $m_0-m_1+\cdots+(-1)^d m_d = b_0-b_1+\cdots+(-1)^d b_d$.
			\end{enumerate}
			\item[The strong Morse inequalities:]
			For each $i \ge 0$,
			\[m_i-m_{i-1}+\cdots+(-1)^im_0 \ge b_i-b_{i-1}+\cdots+(-1)^ib_0.\]		
		\end{description}
	\end{theorem}
	We call a gradient vector field a \emph{perfect gradient vector field} if  $m_i = b_i$ for all $i$. Since $b_i$ is the free rank of the $i$-th homology group of $\K$, it follows that no perfect gradient vector field exists on $\K$ if a homology group of $\K$ has torsion. Moreover, a perfect gradient vector field on a complex may not exist even when all the homology groups are torsion-free, e.g., the dunce hat (\cite{ayala,whitehead,zeeman}). This motivates us to call a gradient vector field an \emph{optimal gradient vector field} if the number of critical simplices is the least possible (in comparison with all other gradient vector fields on the same complex). Here we note that the problem of finding an optimal gradient vector field (equivalently, finding a gradient vector field of the highest cardinality) on a given complex is not a computationally easy problem in general; in fact it was shown to be an \textsf{NP-hard} problem (\cite{joswig,lewiner}), even for $2$-dimensional complexes (\cite{egecioglu,lewiner2}).  A linear algorithm to find optimal gradient vector fields on (discrete) $2$-manifolds was provided by \cite{lewiner2}. Although this algorithm can be extended to CW complexes of dimension (up to) two without the manifold property (an example of such a complex, in relation to this article, is the matching complex $M_7$), the resulting gradient vector field may be arbitrarily far from the optimum.
	
	 The following is a useful result to augment a given gradient vector field on a complex (i.e., to reduce the number of critical simplices) by ``cancelling" a pair of critical simplices.
	\begin{theorem}[Cancellation of a critical pair, {\cite{forman,forman02}}]\label{cancel}
		Suppose $\V$ is a gradient vector field on a complex $\K$ such that $\alpha^{(d)}$ and $\beta^{(d+1)}$ are critical. If there is a unique $\V$-path from a $d$-simplex contained in $\beta$, say $\alpha^{(d)}_0$, to $\alpha$, then there is a gradient vector field $\V'$ on $\K$ such that the  critical simplices with respect to $\V'$ remain the same, except that $\alpha$ and $\beta$ are no longer critical. Moreover, $\V'$ is same as $\V$ except along the unique $\V$-path from $\alpha_0$ to $\alpha$.
	\end{theorem}
	A sketch of a proof is as follows. Let the unique $\V$-path from $\alpha_0$ to $\alpha$ be
	\[\alpha_0^{(d)}, \beta_0^{(d+1)}, \alpha_1^{(d)}, \beta_1^{(d+1)}, \ldots, \alpha_k^{(d)}, \beta_k^{(d+1)}, \alpha_{k+1}^{(d)}=\alpha.\]
	Thus, we have the following diagram.
	\begin{center}
		\begin{tikzcd}
			\beta \ar[r] & \alpha_0 \ar[r,tail] & \beta_0 \ar[r] & \alpha_1 \ar[r,tail] & \beta_1 \ar[r] & \cdots \ar[r] & \alpha_k \ar[r,tail] & \beta_k \ar[r] & \alpha_{k+1}=\alpha
		\end{tikzcd}
	\end{center}
	We get $\V'$ from $\V$ by reversing the arrows (with $\to$ becoming $\leftarrowtail$ and $\rightarrowtail$ becoming $\leftarrow$) in the diagram above.
	\begin{center}
		\begin{tikzcd}
			\beta & \alpha_0 \ar[l,tail] & \beta_0 \ar[l] & \alpha_1 \ar[l,tail] & \beta_1 \ar[l] & \cdots \ar[l,tail] & \alpha_k \ar[l,tail] & \beta_k \ar[l] & \alpha_{k+1}=\alpha \ar[l,tail]
		\end{tikzcd}
	\end{center}
	In other words, 
	\[\V'=\left(\V \setminus \left\{(\alpha_i,\beta_i): i \in \{0,1,\ldots,k\}\right\}\right) \sqcup \left\{(\alpha_{i+1},\beta_i): i \in \{0,1,\ldots,k\}\right\} \sqcup \{(\alpha_0,\beta)\}\]
	($\sqcup$ denotes the union of disjoint sets). The uniqueness of the $\V$-path from $\alpha_0$ to $\alpha$ guarantees that $\V'$ is also a gradient vector field on $\K$. Moreover, it implies that $\alpha$ and $\beta$ are not critical with respect to $\V'$, while the criticality of all other simplices remains unchanged.
	
	The following allows us to apply the technique above to cancel several pairs of critical simplices simultaneously.
	\begin{theorem}[{\cite{hersh}}]\label{multi-cancel}
		Let $\V$ be a gradient vector field on a complex $\K$ such that for $i \in \{1,\ldots,r\}$ there is a unique $\V$-path $\gamma_i$ from a $(d_i-1)$-simplex contained in the critical $d_i$-simplex $\beta_i$ to the critical $(d_i-1)$-simplex $\alpha_i$. If there is no non-identity permutation $\pi$ of $r$ elements such that there is a $\V$-path from a $(d_i-1)$-simplex contained in $\beta_i$ to $\alpha_{\pi(i)}$ for all $i \in \{1,\ldots,r\}$, then reversing all the $\V$-paths $\gamma_i$ (to cancel the critical pair $\alpha_i$ and $\beta_i$) would still produce a gradient vector field on $\K$.
	\end{theorem}
	
	\subsection{Discrete Morse homology}
	Let $\K$ be an abstract simplicial complex and $\V$ be a given gradient vector field on $\K$. Let us fix an ordering on $V(\K)$, which induces an orientation on the simplices.
	
	First, we need to introduce the notion of the \emph{incidence number} between two oriented simplices of consecutive dimensions. Let $\beta=[x_0,x_1,\ldots,x_k]$ be a $k$-simplex. If $\alpha=[x_0,\ldots,\widehat{x_i},\ldots,x_k]$ is a $(k-1)$-simplex (contained in $\beta$), then the incidence number between $\beta$ and $\alpha$ is $(-1)^i$, and we denote it by $\langle\beta,\alpha\rangle$. Otherwise, if $\alpha$ is a $(k-1)$-simplex such that $\alpha \nsubseteq \beta$, then we define $\langle\beta,\alpha\rangle$ to be $0$.
	
	Now, let 
	\[\gamma :  \alpha_0^{(d)}, \beta_0^{(d+1)}, \alpha_1^{(d)}, \beta_1^{(d+1)}, \ldots, \alpha_{k-1}^{(d)}, \beta_{k-1}^{(d+1)}, \alpha_{k}^{(d)}\]
	be a $\V$-path. The \emph{multiplicity} of $\gamma$ (\cite{forman,gallais}), denoted by $m(\gamma)$, is given by
	\begin{equation}\label{mult}
		m(\gamma) \coloneqq 
		\prod_{i=0}^{k-1} (-1)\langle\beta_i,\alpha_i\rangle\langle\beta_i,\alpha_{i+1}\rangle =
		(-1)^{k}\prod_{i=0}^{k-1} \langle\beta_i,\alpha_i\rangle\langle\beta_i,\alpha_{i+1}\rangle.
	\end{equation}
	We observe that, for all $i$, both $\langle\beta_i,\alpha_i\rangle$, $\langle\beta_i,\alpha_{i+1}\rangle$ $\in \{-1,+1\}$, and thus $m(\gamma) \in \{-1,+1\}$. 
	
	Let $\Gamma(\sigma', \sigma)$ be the set of $\V$-paths starting at $\sigma'$ and ending at $\sigma$. Let us denote the free abelian group generated by all critical (with respect to $\V$) $k$-simplices of $\K$ by $\tilde{C}_k(\K)$. For a (oriented) critical $k$-simplex $\tau$, first we define a boundary operator $\tilde{\partial}_k$ on $\tau$ as below, and then extend it linearly to $\tilde{\partial}_k : \tilde{C}_k(\K) \to \tilde{C}_{k-1}(\K)$.
	\begin{equation}\label{boundary}
		\tilde{\partial}_k(\tau)  \coloneqq \sum_{\sigma : \text{ critical}} n(\tau^{(k)}, \sigma^{(k-1)})\cdot \sigma^{(k-1)},
	\end{equation}
	where
	\begin{equation}\label{bound-coeff}
		n(\tau, \sigma) \coloneqq 
		\sum_{\sigma' \subsetneq \tau} \langle\tau^{(k)},\sigma'^{(k-1)}\rangle \sum_{\gamma \in \Gamma(\sigma', \sigma)}m(\gamma).
	\end{equation}
	It follows that, for all $k \ge 0$, we have $\tilde{\partial}_k \circ \tilde{\partial}_{k+1}=0$ (\cite{forman,gallais}). Thus, $(\tilde{C}_*(\K), \tilde{\partial}_*)$ is a chain complex, and we call the homology of this chain complex the (discrete) \emph{Morse homology} of $\K$ (with respect to the chosen gradient vector field $\V$).
	\begin{theorem}[{\cite{forman}}]\label{funda-hom}
		With respect to any given gradient vector field $\V$ on $\K$, the chain complex $(\tilde{C}_*(\K), \tilde{\partial}_*)$ is homotopy equivalent to the simplicial chain complex $(C_*(\K), \partial_*)$. Consequently, the discrete Morse homology, which is independent of the chosen gradient vector field, is isomorphic to the simplicial homology, i.e., 
		\[H_k(\tilde{C}_*(\K), \tilde{\partial}_*)=H_k(\K), \text{ for all }k \ge 0.\]
	\end{theorem} 
	
	\section{Construction of a gradient vector field on $M_n$}\label{sec-cons}
	In this section, we construct a gradient vector field on $M_n$ for $n \ge 5$, which doesn't admit any critical simplices of dimension up to (and including) $\nu_n -1$, except one unavoidable critical $0$-simplex. For the rest of the article, we implicitly assume $n \ge 5$ whenever we talk about $M_n$ in general.
	
	For a positive integer $k$, we denote the set $\{1,\ldots,k\}$ by $[k]$. We partition the vertex set $V(K_n)$ into $\ceil{\frac{n}{3}}$ sets, and label the vertices of $K_n$ depending on the part they belong to as below.
	\begin{enumerate}
		\item For $n=3m$, $V(K_n) = V_1 \sqcup \ldots \sqcup V_m$, where $V_i = \{v^{(i)}_{1}, v^{(i)}_{2}, v^{(i)}_{3}\}$ for all $i \in [m]$.
		\item For $n=3m+1$, $V(K_n) = V_1 \sqcup \ldots \sqcup V_{m+1}$, where $V_i = \{v^{(i)}_{1}, v^{(i)}_{2}, v^{(i)}_{3}\}$ for all $i \in [m]$, and $V_{m+1} = \{v^{(m+1)}_{1}\}$.
		\item For $n=3m+2$, $V(K_n) = V_1 \sqcup \ldots \sqcup V_{m+1}$, where $V_i = \{v^{(i)}_{1}, v^{(i)}_{2}, v^{(i)}_{3}\}$ for all $i \in [m]$, and $V_{m+1} = \{v^{(m+1)}_{1},v^{(m+1)}_{2}\}$.
	\end{enumerate}
	We call an edge $e$ an \emph{$i$-level edge} (or simply, a \emph{levelled-edge}) if both of its endvertices are in the same $V_i$ for some $i$. Otherwise, if endvertices of $e$ are in $V_i$ and $V_j$ with $i\ne j$, then $e$ is called a \emph{cross-edge} (between $V_i$ and $V_j$) (see Fig.~\ref{v-part}).
	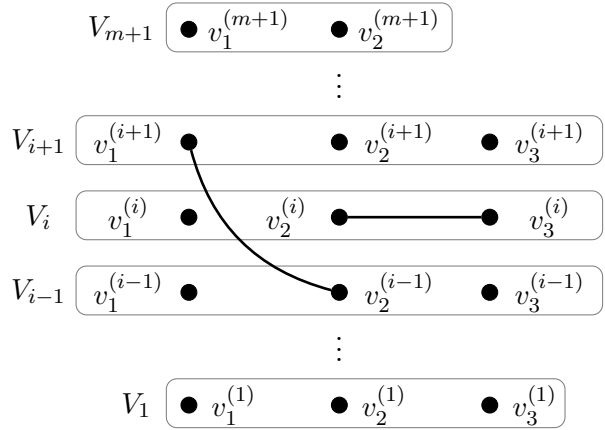
\begin{figure}[h!]
		\centering
		\begin{tikzpicture}
			
			\draw[gray, rounded corners] (1.7,-0.3) rectangle ++(5.3,0.65);\node at (1.3,0) {$V_{1}$};
			\vertex (v11) at (2,0) {};\node at (2.6,0) {$v^{(1)}_{1}$};
			\vertex (v12) at (4,0) {};\node at (4.6,0) {$v^{(1)}_{2}$};
			\vertex (v13) at (6,0) {};\node at (6.6,0) {$v^{(1)}_{3}$};
			
			\node at (4,0.85) {$\vdots$};
			
			\draw[gray, rounded corners] (0.5,1.2) rectangle ++(7,0.65);\node at (0,1.5) {$V_{i-1}$};
			\vertex (v21) at (2,1.5) {};\node at (1.2,1.5) {$v^{(i-1)}_{1}$};
			\vertex (v22) at (4,1.5) {};\node at (4.8,1.5) {$v^{(i-1)}_{2}$};
			\vertex (v23) at (6,1.5) {};\node at (6.8,1.5) {$v^{(i-1)}_{3}$};
			
			\draw[gray, rounded corners] (0.5,2.2) rectangle ++(7,0.65);\node at (0,2.5) {$V_{i}$};
			\vertex (v31) at (2,2.5) {};\node at (1.2,2.5) {$v^{(i)}_{1}$};
			\vertex (v32) at (4,2.5) {};\node at (3.3,2.5) {$v^{(i)}_{2}$};
			\vertex (v33) at (6,2.5) {};\node at (6.8,2.5) {$v^{(i)}_{3}$};
			
			\draw[gray, rounded corners] (0.5,3.2) rectangle ++(7,0.65);\node at (0,3.5) {$V_{i+1}$};
			\vertex (v41) at (2,3.5) {};\node at (1.2,3.5) {$v^{(i+1)}_{1}$};
			\vertex (v42) at (4,3.5) {};\node at (4.8,3.5) {$v^{(i+1)}_{2}$};
			\vertex (v43) at (6,3.5) {};\node at (6.8,3.5) {$v^{(i+1)}_{3}$};
			
			\node at (4,4.35) {$\vdots$};
			
			\draw[gray, rounded corners] (1.7,4.7) rectangle ++(3.8,0.65);\node at (1.1,5) {$V_{m+1}$};
			\vertex (v51) at (2,5) {};\node at (2.8,5) {$v^{(m+1)}_{1}$};
			\vertex (v52) at (4,5) {};\node at (4.8,5) {$v^{(m+1)}_{2}$};
			
			\path [line width=1pt]
			(v32) edge (v33)
			;
			\path [line width=1pt, bend left]
			(v22) edge (v41)
			;
		\end{tikzpicture}
		\caption{The partition of $V(K_{3m+2})$ into $m+1$ levels. Here, $v_2^{(i)} v_3^{(i)}$ is an $i$-level edge and $v_2^{(i-1)}v_1^{(i+1)}$ is a cross-edge between $V_{i-1}$ and $V_{i+1}$.}
		\label{v-part}
	\end{figure}

	First, we define the following discrete vector fields on $M_n$:
	\begin{align*}
		\M'_1 &= \left\{\left(\alpha,\alpha \sqcup \{v^{(1)}_{j}v^{(1)}_{k}\}\right) :  \alpha \in M_n, \alpha \text{ covers only } v^{(1)}_i \text{ of }V_1, \{i,j,k\}=[3]\right\},\\
		\M''_1 &= \left\{\left(\alpha,\alpha \sqcup \{v^{(1)}_{2}v^{(1)}_{3}\}\right) : \alpha \in M_n, \alpha  \text{ leaves entire } V_1 \text{ uncovered}\right\}
	\end{align*}
	(see Fig.~\ref{m1'} and Fig.~\ref{m1''}).
		\begin{figure}[H]
			\centering
		\begin{tikzpicture}
			\node at (4,-1) {$\alpha$};
			
			\draw[gray, rounded corners] (1.7,-0.3) rectangle ++(5.3,0.65);\node at (1.4,0) {$V_{1}$};
			\vertex (v11) at (2,0) {};\node at (2.6,0) {$v^{(1)}_{1}$};
			\vertex (v12) at (4,0) {};\node at (4.6,0) {$v^{(1)}_{2}$};
			\vertex (v13) at (6,0) {};\node at (6.6,0) {$v^{(1)}_{3}$};
			
			\node at (4.5,0.85) {$\vdots$};
			
			\vertex (v) at (3,2) {};\node at (3.4,2) {$v$};
			
			\path [line width=1pt]
			
			(v12) edge (v)
			;    
			
			\draw [to reversed-to](7.2,1) -- (8.2,1);
			
			\node at (11,-1.2) {$\alpha \sqcup \{v^{(1)}_1v^{(1)}_3\}$};
			
			\draw[gray, rounded corners] (8.7,-0.3) rectangle ++(5.3,0.65);\node at (8.4,0) {$V_{1}$};
			\vertex (w11) at (9,0) {};\node at (9.7,0) {$v^{(1)}_{1}$};
			\vertex (w12) at (11,0) {};\node at (11.6,0) {$v^{(1)}_{2}$};
			\vertex (w13) at (13,0) {};\node at (13.6,0) {$v^{(1)}_{3}$};
			
			\node at (11.5,0.85) {$\vdots$};
			
			\vertex (w) at (10,2) {};\node at (10.4,2) {$v$};
			
			\path [line width=1pt]
			(w12) edge (w)
			;
			\path [line width=1pt, bend right]
			(w11) edge (w13)
			;
		\end{tikzpicture}
		\caption{The matching $\alpha$ covers only $v^{(1)}_2$ of $V_1$. Thus, $\alpha$ is paired off with $\alpha \sqcup \{v^{(1)}_1v^{(1)}_3\}$ in $\M'_1$.}
		\label{m1'}
	\end{figure}
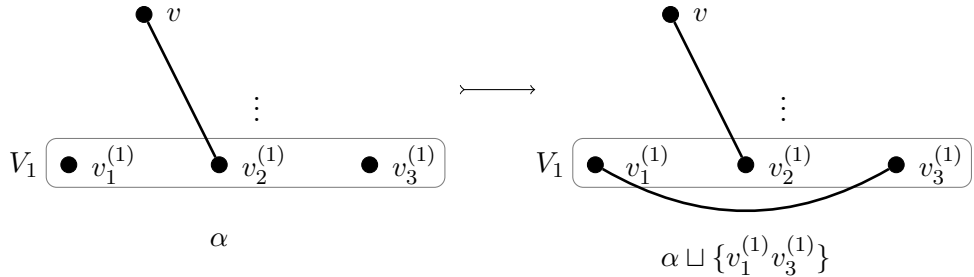

		\begin{figure}[!ht]
			\centering
		\begin{tikzpicture}
			\node at (4,-1) {$\alpha$};
			
			\draw[gray, rounded corners] (1.7,-0.3) rectangle ++(5.3,0.65);\node at (1.4,0) {$V_{1}$};
			\vertex (v11) at (2,0) {};\node at (2.6,0) {$v^{(1)}_{1}$};
			\vertex (v12) at (4,0) {};\node at (4.6,0) {$v^{(1)}_{2}$};
			\vertex (v13) at (6,0) {};\node at (6.6,0) {$v^{(1)}_{3}$};
			
			\node at (4.5,0.85) {$\vdots$};
			
			\draw [to reversed-to](7.2,1) -- (8.2,1);
			
			\node at (11,-1) {$\alpha \sqcup \{v^{(1)}_2v^{(1)}_3\}$};
			
			\draw[gray, rounded corners] (8.7,-0.3) rectangle ++(5.3,0.65);\node at (8.4,0) {$V_{1}$};
			\vertex (w11) at (9,0) {};\node at (9.6,0) {$v^{(1)}_{1}$};
			\vertex (w12) at (11,0) {};\node at (10.5,0) {$v^{(1)}_{2}$};
			\vertex (w13) at (13,0) {};\node at (13.6,0) {$v^{(1)}_{3}$};
			
			\node at (11.5,0.85) {$\vdots$};
			
			\path [line width=1pt]
			
			(w12) edge (w13)
			;
		\end{tikzpicture}
		\caption{The matching $\alpha$ leaves entire $V_1$ uncovered. Thus, $\alpha$ is paired off with $\alpha \sqcup \{v^{(1)}_2v^{(1)}_3\}$ in $\M''_1$.}
		\label{m1''}
	\end{figure}
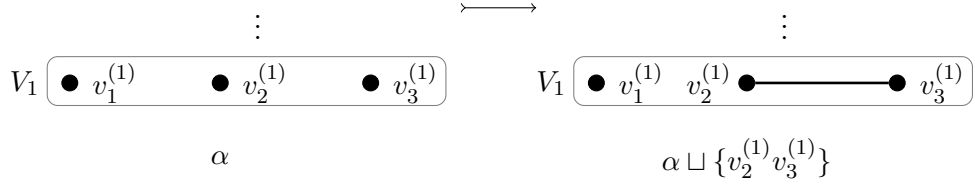
	
	Let $\M_1 = \M'_1 \sqcup \M''_1$. We note that $\M_1$ is also a discrete vector field on $M_n$. Let $\U_1$ denote the set of all matchings in $K_n$ that are not paired off in $\M_1$, i.e., $\U_1 = \{\alpha \in M_n : \alpha$ does not appear in any pair of $\M_1\}$. We observe that $\alpha \in \U_1$ if and only if one of the following holds:
	\begin{enumerate}[(i)]
		\item at least two vertices of $V_1$ are matched by $\alpha$ with vertices outside $V_1$,
		\item $\alpha$ contains the edge $v^{(1)}_1v^{(1)}_i$, where $i \in \{2,3\}$ and $\alpha$ doesn't cover the vertex $v \in V_1 \setminus \{v^{(1)}_1, v^{(1)}_i\}$
	\end{enumerate}
	(see Fig.~\ref{u1}).
	\begin{figure}[!ht]
		\centering
		\begin{tikzpicture}
			\node at (4,-1) {$\alpha_1$};
			
			\draw[gray, rounded corners] (1.7,-0.3) rectangle ++(5.3,0.65);\node at (1.4,0) {$V_{1}$};
			\vertex (v11) at (2,0) {};\node at (2.6,0) {$v^{(1)}_{1}$};
			\vertex (v12) at (4,0) {};\node at (4.6,0) {$v^{(1)}_{2}$};
			\vertex (v13) at (6,0) {};\node at (6.6,0) {$v^{(1)}_{3}$};
			
			\node at (4.5,0.85) {$\vdots$};
			
			\vertex (v) at (3,2) {};\node at (3.4,2) {$v$};
			\vertex (w) at (5,2.5) {};\node at (5.4,2.5) {$w$};
			
			\path [line width=1pt]
			(v12) edge (v)
			(v13) edge (w)
			;

			\node at (11,-1.2) {$\alpha_2$};
			
			\draw[gray, rounded corners] (8.7,-0.3) rectangle ++(5.3,0.65);\node at (8.4,0) {$V_{1}$};
			\vertex (w11) at (9,0) {};\node at (9.7,0) {$v^{(1)}_{1}$};
			\vertex (w12) at (11,0) {};\node at (11.6,0) {$v^{(1)}_{2}$};
			\vertex (w13) at (13,0) {};\node at (13.6,0) {$v^{(1)}_{3}$};
			
			\node at (11.5,0.85) {$\vdots$};
			
			\path [line width=1pt, bend right]
			(w11) edge (w13)
			;
		\end{tikzpicture}
		\caption{The matching $\alpha_1$ matches $v^{(1)}_2$ and $v^{(1)}_3$ of $V_1$ with vertices outside $V_1$, whereas $\alpha_2$ contains the edge $v^{(1)}_1v^{(1)}_3$, but leaves $v^{(1)}_2$ uncovered. Thus, both $\alpha_1, \alpha_2 \in \U_1$.}
		\label{u1}
	\end{figure}
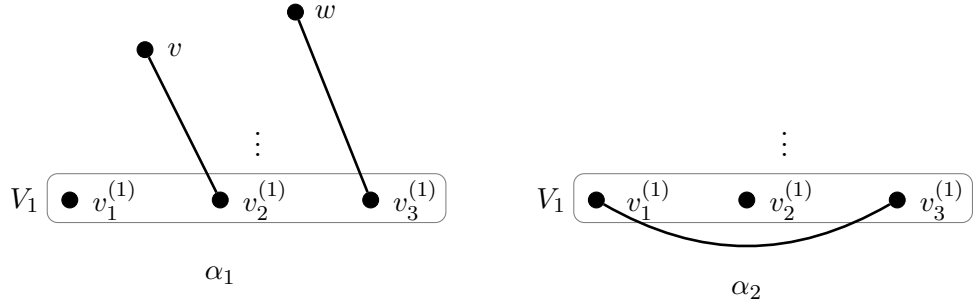

	Next, following the same scheme as before, we define the following discrete vector fields on $M_n$:
	\begin{align*}
		\M'_2 &= \left\{\left(\alpha,\alpha \sqcup \{v^{(2)}_{j}v^{(2)}_{k}\}\right) : \alpha \in \U_1, \alpha \text{ covers only } v^{(2)}_i \text{ of }V_2, \{i,j,k\}=[3]\right\},\\
		\M''_2 &= \left\{\left(\alpha,\alpha \sqcup \{v^{(2)}_{2}v^{(2)}_{3}\}\right) : \alpha \in \U_1, \alpha \text{ leaves entire } V_2 \text{ uncovered}\right\}.
	\end{align*}
	Let $\M_2 = \M'_2 \sqcup \M''_2$ (which is also a discrete vector field on $M_n$) and $\U_2 = \{\alpha \in M_n : \alpha$ does not appear in any pair of $\M_1 \sqcup \M_2\}$. Suppose $\alpha \in \U_1$. Then $\alpha \in \U_2$ if and only if one of the following holds:
	\begin{enumerate}[(i)]
		\item at least two vertices of $V_2$ are matched by $\alpha$ with vertices outside $V_2$,
		\item $\alpha$ contains the edge $v^{(2)}_1v^{(2)}_i$, where $i \in \{2,3\}$ and $\alpha$ doesn't cover the vertex $v \in V_2 \setminus \{v^{(2)}_1, v^{(2)}_i\}$.
	\end{enumerate}
		In general, following the same scheme as above, we get a sequence of discrete vector fields $\M_1,\ldots,\M_m$ on $M_n$ and a family of subsets $\U_1 \supseteq \U_2 \supseteq \ldots \supseteq \U_m$ of $M_n$ such that $\U_k = \{\alpha \in M_n : \alpha$ does not appear in any pair of $\M_1 \sqcup \ldots \sqcup \M_k\}$, for all $k \in [m]$. We note that for any $k \in \{2,\ldots,m\}$, the discrete vector field $\M_k$ may contain only  pairs of matchings of the form $(\alpha, \alpha \sqcup \{e\})$ where both $\alpha$ and $\alpha \sqcup \{e\}$ are in $\U_{k-1}$.
	
	If $n=3m+2$, we define another discrete vector field on $M_n$, viz., 
	\[\M_{m+1} = \left\{\left(\alpha, \alpha \sqcup \{v^{(m+1)}_{1}v^{(m+1)}_{2}\}\right) : \alpha \in \U_m, \alpha \text{ covers neither  } v^{(m+1)}_{1} \text{ nor }v^{(m+1)}_{2}\right\}.\]
	
	For $n=3m$ or $3m+1$, let $\M = \M_1 \sqcup \ldots \sqcup \M_{m}$ and for $n=3m+2$, let $\M = \M_1 \sqcup \ldots \sqcup \M_{m+1}$. We note that $\M$ as defined above is a collection of pairs of matchings in $K_n$, and thus $\M$ also depends on $n$. To make the notation less cumbersome, we avoid adding the parameter $n$ to $\M$.
	
	From the construction of $\M$, we make the following observation.
	\begin{obs}\label{obs1}
		For any $n$, the following hold.
		\begin{enumerate}
			\item $\M$ is a discrete vector field on $M_n$.
			\item If $(\alpha, \alpha \sqcup \{e\}) \in \M$, then $e$ is an $i$-level edge, for some $i \in [\ceil{\frac{n}{3}}]$.
		\end{enumerate}		
	\end{obs}
	We now prove that $\M$ is a gradient vector field on $M_n$.
	\begin{proposition}
		For all $n$, the discrete vector field $\M$ is a gradient vector field on $M_n$.
	\end{proposition}
	\begin{proof}
		Let, if possible, $\alpha_0, \beta_0, \alpha_1, \beta_1, \ldots, \alpha_r, \beta_r, \alpha_{r+1}=\alpha_0$ be a nontrivial closed $\M$-path, and let us denote it by $\gamma$. Let $\beta_0 = \alpha_0 \sqcup \{e_0\}$ and $\alpha_{1} = \beta_0 \setminus \{e'_0\}$ (with $e_0 \ne e'_0$).	
		\begin{center}
			\begin{tikzcd}
			\alpha_0 \ar[r,tail,"\sqcup \{e_0\}"] & \beta_0 \ar[r,"-\{e'_0\}"] & \alpha_1 \ar[r,tail] & \beta_1 \ar[r] & \cdots \ar[r] & \alpha_r \ar[r,tail] & \beta_r \ar[r] & \alpha_{r+1}=\alpha_0
		\end{tikzcd}
		\end{center}
		From Observation~\ref{obs1}, it follows that $e_0$ is a levelled-edge. Now, if $e'_0$ is a cross-edge, then from the construction of $\M$, it follows that $e'_0 \notin \alpha_p$, for all $p \ge 1$. This contradicts the assumption that $\gamma$ is a nontrivial closed $\M$-path as $e'_0 \in \alpha_0$. Thus, $e'_0$ is also a levelled-edge.
		
		Let $e_0$ and $e'_0$ be an $i_0$-level edge and an $i_1$-level edge for some $i_0,i_1 \in [\ceil{\frac{n}{3}}]$, respectively. Thus, it follows that $(\alpha_0,\beta_0) \in \M_{i_0}$ and both $\alpha_0, \beta_0 \in \U_k$ for all $k<i_0$.
		
		We note that $i_0 \ne i_1$ as both $e_0,e'_0 \in \beta_0$. Now, if $i_1>i_0$, then we observe that $\alpha_{1}$ matches the vertices of $V_1\sqcup \ldots \sqcup V_{i_0}$ in \emph{exactly the same manner} as $\beta_0$ does. This implies, just as $\beta_0$, the matching $\alpha_{1} \in \U_k$ for all $k<i_0$. Since $(\alpha_0=\beta_0 \setminus \{e_0\},\beta_0) \in \M$, we have $(\alpha_{1} \setminus \{e_0\}, \alpha_{1}) \in \M$, a contradiction as  $(\alpha_{1}, \beta_{1}) \in \M$. Therefore, we have $i_1<i_0$.
		
		If $e'_0 = v^{(i_1)}_kv^{(i_1)}_\ell$ (where $k, \ell \in [3]$) and the vertex $v^{(i_1)}_m$ (where $m \in [3]\setminus \{k, \ell\}$) is matched by $\alpha_0$ (with a vertex outside $V_{i_1}$), then from the properties of simplices appearing in $\U_{i_1}$, we have $\alpha_{0} \notin \U_{i_1}$, a contradiction. Similarly, if $e'_0 = v^{(i_1)}_2v^{(i_1)}_3$ and the vertex $v^{(i_1)}_1$ is not covered $\alpha_0$, then again  we have $\alpha_{0} \notin \U_{i_1}$, a contradiction.
		
		So suppose $e'_0 = v^{(i_1)}_1v^{(i_1)}_k$ (where $k \in \{2,3\}$) and the vertex $v^{(i_1)}_\ell$ (where $\ell \in [3] \setminus \{1,k\}$) is not covered by $\alpha_0$, and thus not covered by $\alpha_1$ as well. In this case, $\beta_1 = \alpha_1 \sqcup \{v^{(i_1)}_2v^{(i_1)}_3\}$ and $(\alpha_1,\beta_1) \in \M_{i_1}$ (we note that $\beta_{1}$ doesn't contain $e'_0$, and thus neither does $\alpha_2$). By a similar argument as before, $\alpha_2$ is of the form $\beta_1 \setminus \{e'_1\}$ and $\beta_2$ is of the form $\alpha_2 \sqcup \{e_2\}$, where $e'_1$ and $e_2$ are two $i_2$-level edges (distinct from each other) with $i_2<i_1$. Thus, $e_2 \ne e'_0$, and consequently $e'_0 \notin \beta_2, \alpha_3$.	
		\begin{center}
			\begin{tikzcd}
			\alpha_0 \ar[r,tail,"\sqcup \{e_0\}","(i_0)"'] & \beta_0 \ar[r,"-\{e'_0 = v^{(i_1)}_1v^{(i_1)}_k\}","(i_1)"'] & [4em]\alpha_1 \ar[r,tail,"\sqcup \{v^{(i_1)}_2v^{(i_1)}_3\}","(i_1)"'] & [3em]\beta_1 \ar[r,"-\{e'_1\}","(i_2)"'] & \alpha_2 \ar[r,tail,"\sqcup \{e_2\}","(i_2)"'] & \beta_2 \ar[r] & \cdots \ar[r] & \alpha_{r+1}
		\end{tikzcd}
		(with $i_0>i_1>i_2>\cdots$ and so on)
		\end{center}
		By an inductive argument, we conclude that $e'_0 \notin \alpha_p, \beta_p$, for all $p \ge 1$. Thus, $\alpha_{r+1}\ne \alpha_0$ as $e'_0 \in \alpha_0$, which is a contradiction.
	\end{proof}

	We now show that the constructed gradient vector field satisfies the requirements of Theorem~\ref{nolowcrit}.
	
	\begin{proof}[of Theorem~\ref{nolowcrit}]
		We claim that the gradient vector field $\M$ on $M_n$ doesn't admit critical simplices of dimension up to $\nu_n-1$ for any $n$, except one $0$-simplex. From the construction of $\M$ on $M_n$, we observe that if $\alpha \in M_n$ is not paired off in $\M$, then for each $i \in \ceil{\frac{n}{3}}$, the matching $\alpha$ leaves at most one vertex of $V_i$ uncovered. Therefore, $\alpha$ leaves at most $\ceil{\frac{n}{3}}$ vertices of the graph $K_n$ uncovered. This implies 
		\[|\alpha| \ge \frac{1}{2}\left(n - \left\lceil{\frac{n}{3}}\right\rceil\right) > \left\lfloor{\frac{n+1}{3}}\right\rfloor -1 = \nu_n,\]
		i.e.,  the dimension of $\alpha$ is strictly greater than $\nu_n-1$.
		
		Also, the $0$-simplex $\xi \coloneqq \{v^{(1)}_2v^{(1)}_3\}$ is paired off with $\emptyset$ in $\M$, and thus $\xi$ is the only critical $0$-simplex. Therefore, all simplices of $M_n$ of dimension up to $\nu_n-1$, except $\xi$, are not critical with respect to $\M$.
	\end{proof}
	By Theorem~\ref{funda}, it follows that $M_n$ is homotopy equivalent to a CW complex with no cells of dimension up to $\nu_n - 1$, except one $0$-cell. Thus, Theorem~\ref{homcon} follows as a corollary of Theorem~\ref{nolowcrit}. However, we remark that the existence of such a gradient vector field is in fact a stronger notion. Even when a complex is homotopically $k$-connected, it is not guaranteed that there is a gradient vector field on it, with no critical simplices (except one $0$-simplex) of dimension up to $k$. A well-known example is the \emph{dunce hat}. It is a contractible space (and thus simply connected), but no gradient vector field, with a $0$-simplex as the only critical simplex, can be assigned on any triangulation of the dunce hat (\cite{ayala,whitehead,zeeman}).

	\begin{example}[Homotopy type of $M_8$]\label{m8top}
		In order to determine the homotopy type of $M_8$, we consider the gradient vector field $\M$ defined on $M_8$, and extend it to a perfect gradient vector field as follows. Let $\alpha^\circ$ be a $3$-simplex in $M_8$ (i.e., $\alpha^\circ$ is a perfect matching in $K_8$), which is critical with respect to $\M$. From the construction of $\M$, it follows that all the edges in $\alpha^\circ$ are cross-edges. Moreover, for any $e \in \alpha^\circ$, the $2$-simplex $\alpha^\circ \setminus \{e\}$ is also critical with respect to $\M$. Now, since $\alpha^\circ$ is a perfect matching containing only cross-edges, $\alpha^\circ$ matches $v^{(3)}_1$ with a vertex outside $V_3$, say $v^\circ$. So $v^\circ$ is of the form $v^{(i)}_j$, for some $i \in [2]$ and $j \in [3]$. We extend $\M$ to a discrete vector field $\M^\circ$ by adding the new pairs of the form $(\alpha^\circ \setminus \{v^{(3)}_1v^\circ\}, \alpha^\circ)$ to $\M$ (see Fig.~\ref{m8}), for each $3$-simplex $\alpha^\circ$ which is critical with respect to $\M$.
		\begin{figure}[h!]
			\centering
			\begin{tikzpicture}
				\begin{scope}[scale=0.88, transform shape]
				\node at (4,1.7) {$\alpha^\circ$};
				
				\draw[gray, rounded corners] (1,2.2) rectangle ++(6,0.65);\node at (0.7,2.5) {$V_{1}$};
				\vertex (v11) at (2,2.5) {};\node at (1.5,2.5) {$v^{(1)}_{1}$};
				\vertex (v12) at (4,2.5) {};\node at (3.5,2.5) {$v^{(1)}_{2}$};
				\vertex (v13) at (6,2.5) {};\node at (6.6,2.5) {$v^{(1)}_{3}$};
				
				\draw[gray, rounded corners] (1,3.2) rectangle ++(6.8,0.65);\node at (0.7,3.5) {$V_{2}$};
				\vertex (v21) at (2,3.5) {};\node at (1.5,3.5) {$v^{(2)}_{1}$};
				\vertex (v22) at (4,3.5) {};\node at (3.5,3.5) {$v^{(2)}_{2}$};
				\vertex (v23) at (6,3.5) {};\node at (7,3.5) {$v^{(2)}_{3}=v^\circ$};
				
				\draw[gray, rounded corners] (1,4.2) rectangle ++(4,0.65);\node at (0.7,4.5) {$V_{3}$};
				\vertex (v31) at (2,4.5) {};\node at (1.5,4.5) {$v^{(3)}_{1}$};
				\vertex (v32) at (4,4.5) {};\node at (4.6,4.5) {$v^{(3)}_{2}$};
				
				\path [line width=1pt]
				(v11) edge (v32)
				(v12) edge (v22)
				(v13) edge (v21)
				(v23) edge (v31)
				;
				
				\draw [to reversed-to](9.2,3.5) -- (8,3.5);
				
				\node at (13,1.7) {$\alpha^\circ \setminus \{v^{(3)}_1v^\circ\}$};
				
				\draw[gray, rounded corners] (10,2.2) rectangle ++(6,0.65);\node at (9.7,2.5) {$V_{1}$};
				\vertex (w11) at (11,2.5) {};\node at (10.5,2.5) {$v^{(1)}_{1}$};
				\vertex (w12) at (13,2.5) {};\node at (12.5,2.5) {$v^{(1)}_{2}$};
				\vertex (w13) at (15,2.5) {};\node at (15.6,2.5) {$v^{(1)}_{3}$};
				
				\draw[gray, rounded corners] (10,3.2) rectangle ++(6.8,0.65);\node at (9.7,3.5) {$V_{2}$};
				\vertex (w21) at (11,3.5) {};\node at (10.5,3.5) {$v^{(2)}_{1}$};
				\vertex (w22) at (13,3.5) {};\node at (12.5,3.5) {$v^{(2)}_{2}$};
				\vertex (w23) at (15,3.5) {};\node at (16,3.5) {$v^{(2)}_{3}=v^\circ$};
				
				\draw[gray, rounded corners] (10,4.2) rectangle ++(4,0.65);\node at (9.7,4.5) {$V_{3}$};
				\vertex (w31) at (11,4.5) {};\node at (10.5,4.5) {$v^{(3)}_{1}$};
				\vertex (w32) at (13,4.5) {};\node at (13.6,4.5) {$v^{(3)}_{2}$};
				
				\draw[dashed,gray] (w23)--(w31);
				\path [line width=1pt]
				(w11) edge (w32)
				(w12) edge (w22)
				(w13) edge (w21)
				;
			\end{scope}
			\end{tikzpicture}
			\caption{The critical (with respect to $\M$) $3$-simplex $\alpha^\circ$ is paired off with the critical $2$-simplex $\alpha^\circ \setminus \{v^{(3)}_1v^\circ\}$ in $\M^\circ$.}
			\label{m8}
		\end{figure}
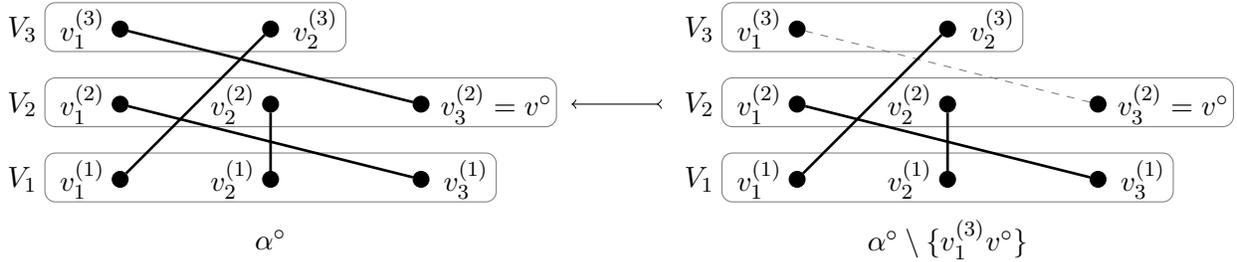
	
		Any $\M^\circ$-path $\gamma$, which is not an $\M$-path, contains a pair $(\alpha_i^{(2)},\beta_i^{(3)})\in \M^\circ \setminus \M$. If $\alpha_{i+1}$ ($\ne \alpha_i$) is any $2$-simplex contained in $\beta_i$, then $\gamma$ does not extend beyond such an $\alpha_{i+1}$, as $\beta_i$ is the only $3$-simplex containing $\alpha_{i+1}$. This implies $\M^\circ$ is also a gradient vector field on $M_8$, and moreover, all the $0$-simplices,  $1$-simplices, and $3$-simplices of $M_8$ are paired off in $\M^\circ$. Thus, only critical simplices of $M_8$ (with respect to $\M^\circ$) are some $2$-simplices and exactly one $0$-simplex, viz., the matching $\{v^{(1)}_2v^{(1)}_3\}$. The number of critical $2$-simplices may be determined from the Euler characteristic of $M_8$ (or, alternatively, by direct counting). If $(f_0,\ldots,f_3)$ is the $f$-vector of $M_8$, then from Equation~(\ref{fv}),
		$f_0 =28$, $f_1 =210$, ${f_2= 420}$, $f_3= 105$, and thus $\chi(M_8) = f_0-f_1+f_2-f_3 = 133$. Therefore, by Theorem~\ref{wedge}, we conclude that $M_8$ is homotopy equivalent to a wedge of 132 spheres of dimension 2.
	\end{example}

\section{Homology groups of $M_7$}\label{hom-m7}

\subsection{Construction of a near-optimal gradient vector field on $M_7$}\label{m7nearopt}
We now consider the matching complex $M_7$ as an example, and determine its homotopy type using techniques developed in this article. Throughout the rest of this article, $\M$ stands for the gradient vector field on $M_7$ in particular, as constructed in Section~\ref{sec-cons}.

We extend the gradient vector field $\M$ to a more useful one as follows. Let $\alpha^*$ be a critical (with respect to $\M$) $1$-simplex containing two cross-edges between $V_1$ and $V_2$. Let $v^*$ be only vertex of $V_1$ that is left uncovered by $\alpha^*$. We observe that $\alpha^*\sqcup \{v^{(3)}_1v^*\}$ is a critical $2$-simplex with respect to $\M$. Thus, we may extend the gradient vector field $\M$ to the discrete vector field $\M^*$ by adding the new pairs of the form $(\alpha^*, \alpha^* \sqcup \{v^{(3)}_1v^*\})$ to $\M$ (see Fig.~\ref{extended}), for each critical (with respect to $\M$) $1$-simplex $\alpha^*$ containing two cross-edges between $V_1$ and $V_2$.

\begin{figure}[h!]
	\centering
	\begin{tikzpicture}
		\begin{scope}
			\draw[gray, rounded corners] (1,2.2) rectangle ++(6,0.65);\node at (0.7,2.5) {$V_{1}$};
			\vertex (v11) at (2,2.5) {};\node at (1.5,2.5) {$v^{(1)}_{1}$};
			\vertex (v12) at (4,2.5) {};\node at (4.9,2.5) {$v^{(1)}_{2}=v^*$};
			\vertex (v13) at (6,2.5) {};\node at (6.6,2.5) {$v^{(1)}_{3}$};
			
			\draw[gray, rounded corners] (1,3.2) rectangle ++(6,0.65);\node at (0.7,3.5) {$V_{2}$};
			\vertex (v21) at (2,3.5) {};\node at (1.5,3.5) {$v^{(2)}_{1}$};
			\vertex (v22) at (4,3.5) {};\node at (4.8,3.5) {$v^{(2)}_{2}$};
			\vertex (v23) at (6,3.5) {};\node at (6.6,3.5) {$v^{(2)}_{3}$};
			
			\draw[gray, rounded corners] (1,4.2) rectangle ++(1.5,0.65);\node at (0.7,4.5) {$V_{3}$};
			\vertex (v31) at (2,4.5) {};\node at (1.5,4.5) {$v^{(3)}_{1}$};
			
			\path [line width=1pt]
			(v11) edge (v21)
			(v13) edge (v22)
			;
		\end{scope}
		
		\draw [to reversed-to](7.2,3.5) -- (8.2,3.5);
		
		\begin{scope}[xshift=8cm]
			\draw[gray, rounded corners] (1,2.2) rectangle ++(6,0.65);\node at (0.7,2.5) {$V_{1}$};
			\vertex (v11) at (2,2.5) {};\node at (1.5,2.5) {$v^{(1)}_{1}$};
			\vertex (v12) at (4,2.5) {};\node at (4.9,2.5) {$v^{(1)}_{2}=v^*$};
			\vertex (v13) at (6,2.5) {};\node at (6.6,2.5) {$v^{(1)}_{3}$};
			
			\draw[gray, rounded corners] (1,3.2) rectangle ++(6,0.65);\node at (0.7,3.5) {$V_{2}$};
			\vertex (v21) at (2,3.5) {};\node at (1.5,3.5) {$v^{(2)}_{1}$};
			\vertex (v22) at (4,3.5) {};\node at (4.8,3.5) {$v^{(2)}_{2}$};
			\vertex (v23) at (6,3.5) {};\node at (6.6,3.5) {$v^{(2)}_{3}$};
			
			\draw[gray, rounded corners] (1,4.2) rectangle ++(1.5,0.65);\node at (0.7,4.5) {$V_{3}$};
			\vertex (v31) at (2,4.5) {};\node at (1.5,4.5) {$v^{(3)}_{1}$};
			
			\path [line width=1pt]
			(v11) edge (v21)
			(v13) edge (v22)
			(v12) edge (v31)
			;
		\end{scope}
	\end{tikzpicture}
	\caption{The critical (with respect to $\M$) $1$-simplex $\alpha^* =$ $\{v^{(1)}_1v^{(2)}_1$, $v^{(1)}_3v^{(2)}_2\}$ is paired off with the critical $2$-simplex $\alpha^* \sqcup \{v^{(3)}_1v^{(1)}_2\}$ in $\M^*$.}
	\label{extended}
\end{figure}
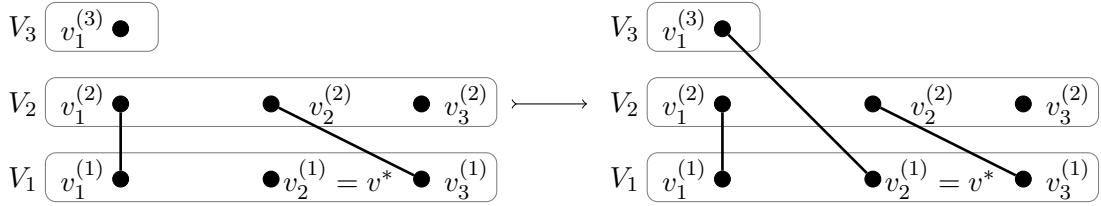

\begin{proposition}
	The discrete vector field $\M^*$ is a gradient vector field on $M_7$.
\end{proposition}
\begin{proof}
	Let, if possible, $\alpha_0^{(p)}, \beta_0^{(p+1)}, \alpha_1^{(p)}, \beta_1^{(p+1)}, \ldots, \alpha_k^{(p)}, \beta_k^{(p+1)}, \alpha_{k+1}^{(p)}=\alpha_0^{(p)}$ be a nontrivial closed $\M^*$-path. If $p=0$, then it leads to a contradiction to the fact that $\M$ is a gradient vector field. So let $p=1$.
	
	If $(\alpha_i,\beta_i)\in \M$ for all $i \in \{0,1,\ldots,k\}$, then again it contradicts the fact that $\M$ is a gradient vector field. So without loss of generality, let $(\alpha_0,\beta_0)\in \M^*\setminus \M$. In this  case, $\alpha_0$ consists of two cross-edges between $V_1$ and $V_2$. Let $\alpha_1=\beta_0 \setminus \{e'\}$, where $e'$ is an edge in $\alpha_0$. We note that if $\beta=\alpha \sqcup \{e\}$ and $(\alpha,\beta) \in \M^*$, then $e$ is either a levelled-edge or a cross-edge between $V_1$ and $V_3$. This implies, for all $i \in [k]$, $e' \notin \beta_i$, and subsequently $e' \notin \alpha_{i+1}$. In particular, $e' \notin \alpha_{k+1}$. This leads to a contradiction, as $\alpha_{k+1}=\alpha_0$.
\end{proof}

\subsection{Computation of (Morse) homology groups of $M_7$}\label{morse-hom}
Throughout this subsection, we always assume the gradient vector field $\M^*$ on the matching complex $M_7$ while discussing the nature of a simplex (i.e., criticality and other related notions). Also, hereafter, while representing a simplex of $M_7$ (i.e., a matching in $K_7$) diagrammatically, we would not explicitly label the vertices of $K_7$ and the sets $V_1$, $V_2$, and $V_3$ for the sake of simplicity, and so it should be understood from the context. For example, we would represent the matching $\{v_1^{(1)}v_1^{(2)}$, $v_2^{(1)}v_1^{(3)}$, $v_3^{(1)}v_2^{(2)}\}$ (see Fig.~\ref{extended}) by the following.
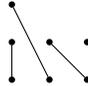
\[\begin{tikzpicture}		
	\smlv (v11) at (0,0) {};
	\smlv (v12) at (0.5,0) {};
	\smlv (v13) at (1,0) {};

	\smlv (v21) at (0,0.5) {};
	\smlv (v22) at (0.5,0.5) {};
	\smlv (v23) at (1,0.5) {};

	\smlv (v31) at (0,1) {};
	
	\path
	(v11) edge (v21)
	(v12) edge (v31)
	(v13) edge (v22)
	;
\end{tikzpicture}\]

\begin{obs}[Characterization of critical simplices]\label{crit-char}
	$\ $
	\begin{enumerate}
		\item  The only critical $0$-simplex is $\xi=\{v^{(1)}_2v^{(1)}_3\}$ (thus $\tilde{C}_0(M_7) \cong \Z$).
		\item Critical $1$-simplices are $\sigma_1 \coloneqq $ $\{v^{(1)}_1v^{(1)}_2$, $v^{(2)}_1v^{(2)}_2\}$, $\sigma_2 \coloneqq $ $\{v^{(1)}_1v^{(1)}_2$, $v^{(2)}_1v^{(2)}_3\}$, $\sigma_3 \coloneqq $ $\{v^{(1)}_1v^{(1)}_3$, $v^{(2)}_1v^{(2)}_2\}$, and $\sigma_4 \coloneqq $ $\{v^{(1)}_1v^{(1)}_3$, $v^{(2)}_1v^{(2)}_3\}$.
		\item Critical $2$-simplices are of the form $\{e_1,e_2,e_3\}$ with one of the following.
		\begin{enumerate}
			\item Each of $e_1$, $e_2$, and $e_3$ is a cross-edge between $V_1$ and $V_2$ (there are 6 such simplices).
			\item Two of $e_1$, $e_2$, and $e_3$ are cross-edges between $V_1$ and $V_2$, and the remaining one is a cross-edge between $V_2$ and $V_3$ (there are 18 such simplices).
		\end{enumerate}
		Therefore, there are 24 critical $2$-simplices.
	\end{enumerate}
\end{obs}

First, we assign a unique label $\ell(e)$ on the edge $e = v^{(i_1)}_{j_1}v^{(i_2)}_{j_2}$ as follows. 
\[\ell(e) = 
\begin{cases}
	i_1j_1i_2j_2,& \text{if } i_1 < i_2\\
	i_2j_2i_1j_1,& \text{if } i_2 < i_1\\
	i_1j_1i_2j_2,& \text{if } i_1 = i_2, j_1 < j_2\\
	i_2j_2i_1j_1,& \text{if } i_1 = i_2, j_2 < j_1
\end{cases}
\]
Next, we introduce a total order $\le$ on the vertex set of the complex $M_7$ (i.e., on $E(K_7)$) by declaring $e_1 \le e_2$ if and only if $\ell(e_1) \le \ell(e_2)$ in the \emph{lexicographic order}. We assign each simplex the orientation induced by this total order on $V(M_7)$, i.e., if $\alpha=\{e_0,e_1,\ldots,e_k\}$ is a matching in $K_7$ with $e_0<e_1<\cdots<e_k$, then we denote the oriented $k$-simplex $[e_0,e_1,\ldots,e_k]$ also by $\alpha$ whenever needed.

\subsubsection{Kernel and image of $\tilde{\partial}_1$}
If $\sigma$ is a critical $1$-simplex and $\alpha^{(0)} \subsetneq \sigma$, then $\alpha$ is one of $\{v^{(1)}_1v^{(1)}_2\}$, $\{v^{(1)}_1v^{(1)}_3\}$, $\{v^{(2)}_1v^{(2)}_2\}$, and $\{v^{(2)}_1v^{(2)}_3\}$. We note that one of the following two cases holds.
\begin{description}
	\item[Case 1:] $\alpha = \{v^{(1)}_1v^{(1)}_i\}$, where $i \in \{2,3\}$\\
	The only possible $\M^*$-path that starts from $\alpha$, and ends at $\xi=\{v^{(1)}_2v^{(1)}_3\}$ is the following.\\
	\[\begin{tikzcd}
		\alpha \ar[r,tail] & \{v^{(1)}_1v^{(1)}_i, v^{(2)}_2v^{(2)}_3\} \ar[r] & \{v^{(2)}_2v^{(2)}_3\} \ar[r,tail] & \{v^{(1)}_2v^{(1)}_3, v^{(2)}_2v^{(2)}_3\} \ar[r] & \xi
	\end{tikzcd}\]
	\item[Case 2:] $\alpha = \{v^{(2)}_1v^{(2)}_i\}$, where $i \in \{2,3\}$\\
	The only possible $\M^*$-path that starts from $\alpha$, and ends at $\xi=\{v^{(1)}_2v^{(1)}_3\}$ is the following.\\
	\[\begin{tikzcd}
		\alpha \ar[r,tail] & \{v^{(1)}_2v^{(1)}_3, v^{(2)}_1v^{(2)}_i\} \ar[r] & \xi
	\end{tikzcd}\]
\end{description}

Let us consider the critical $1$-simplex $\sigma_1 = \{v^{(1)}_1v^{(1)}_2, v^{(2)}_1v^{(2)}_2\}$. Let $\gamma_1$ be the unique $\M^*$-path that starts from $\{v^{(1)}_1v^{(1)}_2\}$ ($\subsetneq \sigma_1$), and ends at $\xi=\{v^{(1)}_2v^{(1)}_3\}$ (from Case~1 above), i.e.,
\[\begin{tikzcd}
	\gamma_1 : \alpha_0=\{v^{(1)}_1v^{(1)}_2\} \ar[r,tail,"(-1)"] & \beta_0 \ar[r,"(+1)"] & \alpha_1 \ar[r,tail,"(+1)"] & \beta_1 \ar[r,"(-1)"] & \alpha_2=\xi,
\end{tikzcd}\]
where $\beta_0=\{v^{(1)}_1v^{(1)}_2, v^{(2)}_2v^{(2)}_3\}$, $\alpha_1=\{v^{(2)}_2v^{(2)}_3\}$, $\beta_1=\{v^{(1)}_2v^{(1)}_3, v^{(2)}_2v^{(2)}_3\}$. Here we have also included $\langle\beta_i,\alpha_j\rangle$, i.e., the incidence number between the oriented simplices $\beta_i$ and $\alpha_j$ (with $i \le j \le i+1$), above the arrow connecting $\beta_i$ and $\alpha_j$. Considering $\alpha_0$, $\beta_0$, $\alpha_1$, $\beta_1$, and $\alpha_2$ as oriented simplices, the multiplicity of $\gamma_1$ (from Equation~(\ref{mult})), 
\begin{align*}
	m(\gamma_1) &= (-\langle\beta_0,\alpha_0\rangle\langle\beta_0,\alpha_1\rangle)(-\langle\beta_1,\alpha_1\rangle\langle\beta_1,\alpha_2\rangle)\\
	&= (-(-1)(+1))(-(+1)(-1)) = +1.
\end{align*}
Let $\gamma_2$ be the unique $\M^*$-path that starts from $\{v^{(2)}_1v^{(2)}_2\}$ ($\subsetneq \sigma_1$), and ends at $\xi=\{v^{(1)}_2v^{(1)}_3\}$ (from Case~2 above), i.e.,
\[\begin{tikzcd}
	\gamma_2 : \alpha'_0=\{v^{(2)}_1v^{(2)}_2\} \ar[r,tail,"(+1)"] & \beta'_0=\{v^{(1)}_2v^{(1)}_3, v^{(2)}_1v^{(2)}_2\} \ar[r,"(-1)"] & \alpha'_1=\xi.
\end{tikzcd}\]
Considering $\alpha'_0$, $\beta'_0$, and $\alpha'_1$ as oriented simplices, the multiplicity of $\gamma_2$ (from Equation~(\ref{mult})),
\[m(\gamma_2) = -\langle\beta'_0,\alpha'_0\rangle\langle\beta'_0,\alpha'_1\rangle = -(+1)(-1) = +1.\]
Therefore, from Equation~(\ref{boundary}) and Equation~(\ref{bound-coeff}),
\begin{align*}
	\tilde{\partial}_1(\sigma_1) &= \left(\langle\sigma_1,\alpha_0\rangle \cdot m(\gamma_1) + \langle\sigma_1,\alpha'_0\rangle \cdot m(\gamma_2)\right)\cdot\xi\\
	&= ((-1)(+1)+(+1)(+1))\cdot\xi = 0.
\end{align*}

Analogous computations yield $\tilde{\partial}_1(\sigma_2)=0$, $\tilde{\partial}_1(\sigma_3)=0$, and $\tilde{\partial}_1(\sigma_4)=0$. Thus, 
\[\ker(\tilde{\partial}_1)=\Bigl\langle \sigma_1, \sigma_2, \sigma_3, \sigma_4 \Bigr\rangle = \tilde{C}_1(M_7) \text{ and } \operatorname{im}(\tilde{\partial}_1)=0.\]
Consequently, the zeroth (Morse) homology group of $M_7$ is $\Z$.
\subsubsection{Image of $\tilde{\partial}_2$}
To begin with, let us consider a critical $2$-simplex and compute its image under the boundary operator $\tilde{\partial}_2$ in the following example.

\begin{example}
	Let us consider the critical $2$-simplex $\eta_1 \coloneqq $ $\{v_2^{(1)}v_2^{(2)}$, $v_3^{(1)}v_3^{(2)}$, $v_1^{(2)}v_1^{(3)}\}$ as shown below.
	\[\begin{tikzpicture}	
		\node at (-0.7,0.5) {$\eta_1 = $};
		
		\smlv (v11) at (0,0) {};
		\smlv (v12) at (0.5,0) {};
		\smlv (v13) at (1,0) {};

		\smlv (v21) at (0,0.5) {};
		\smlv (v22) at (0.5,0.5) {};
		\smlv (v23) at (1,0.5) {};

		\smlv (v31) at (0,1) {};
		
		\path
		(v12) edge (v22)
		(v13) edge (v23)
		(v21) edge (v31)
		;
	\end{tikzpicture}\]
	Let us determine all possible $\M^*$-paths that start from a $1$-simplex contained in $\eta_1$, and end at a critical $1$-simplex. First, a $1$-simplex contained in $\eta_1$ is one of the following three.
	\[\begin{tikzpicture}
		\begin{scope}
			\node at (-0.7,0.5) {$\psi_{11} = $};
			
			\smlv (v11) at (0,0) {};
			\smlv (v12) at (0.5,0) {};
			\smlv (v13) at (1,0) {};

			\smlv (v21) at (0,0.5) {};
			\smlv (v22) at (0.5,0.5) {};
			\smlv (v23) at (1,0.5) {};

			\smlv (v31) at (0,1) {};
			
			\path
			(v13) edge (v23)
			(v21) edge (v31)
			;
		\end{scope}
		
		\begin{scope}[xshift=3cm]
			\node at (-0.7,0.5) {$\psi_{12} = $};
			
			\smlv (v11) at (0,0) {};
			\smlv (v12) at (0.5,0) {};
			\smlv (v13) at (1,0) {};

			\smlv (v21) at (0,0.5) {};
			\smlv (v22) at (0.5,0.5) {};
			\smlv (v23) at (1,0.5) {};

			\smlv (v31) at (0,1) {};
			
			\path
			(v12) edge (v22)
			(v21) edge (v31)
			;
		\end{scope}
		
		\begin{scope}[xshift=6cm]
			\node at (-0.7,0.5) {$\psi_{13} = $};
			
			\smlv (v11) at (0,0) {};
			\smlv (v12) at (0.5,0) {};
			\smlv (v13) at (1,0) {};

			\smlv (v21) at (0,0.5) {};
			\smlv (v22) at (0.5,0.5) {};
			\smlv (v23) at (1,0.5) {};

			\smlv (v31) at (0,1) {};
			
			\path
			(v12) edge (v22)
			(v13) edge (v23)
			;
		\end{scope}
	\end{tikzpicture}\]
	
	In Fig.~\ref{psi1}, we describe all possible $\M^*$-paths that start from $\psi_{11}$, and end either at a critical $1$-simplex or at a $1$-simplex that is paired off with a $0$-simplex in $\M^*$.
	
	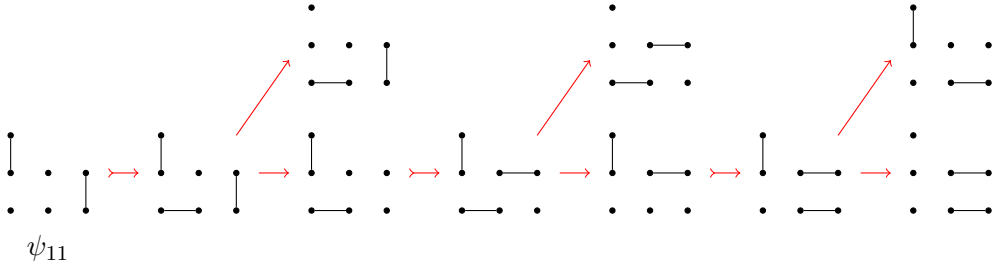
\begin{figure}[!ht]
		\centering
		\begin{tikzpicture}
			\begin{scope}
				\node at (0.5,-0.5) {$\psi_{11}$};
				
				\smlv (v11) at (0,0) {};
				\smlv (v12) at (0.5,0) {};
				\smlv (v13) at (1,0) {};

				\smlv (v21) at (0,0.5) {};
				\smlv (v22) at (0.5,0.5) {};
				\smlv (v23) at (1,0.5) {};

				\smlv (v31) at (0,1) {};
				
				\path
				(v13) edge (v23)
				(v21) edge (v31)
				;
			\end{scope}
			
			\draw [to reversed-to,red](1.3,0.5) -- (1.7,0.5);
			\begin{scope}[xshift=2cm]			
				\smlv (v11) at (0,0) {};
				\smlv (v12) at (0.5,0) {};
				\smlv (v13) at (1,0) {};

				\smlv (v21) at (0,0.5) {};
				\smlv (v22) at (0.5,0.5) {};
				\smlv (v23) at (1,0.5) {};

				\smlv (v31) at (0,1) {};
				
				\path
				(v11) edge (v12)
				(v13) edge (v23)
				(v21) edge (v31)
				;
			\end{scope}
			
			\draw [-to,red](3.3,0.5) -- (3.7,0.5);
			\begin{scope}[xshift=4cm]			
				\smlv (v11) at (0,0) {};
				\smlv (v12) at (0.5,0) {};
				\smlv (v13) at (1,0) {};

				\smlv (v21) at (0,0.5) {};
				\smlv (v22) at (0.5,0.5) {};
				\smlv (v23) at (1,0.5) {};

				\smlv (v31) at (0,1) {};
				
				\path
				(v11) edge (v12)
				(v21) edge (v31)
				;
			\end{scope}
			
			\draw [-to,red](3,1) -- (3.7,2);
			\begin{scope}[xshift=4cm,yshift=1.7cm]			
				\smlv (v11) at (0,0) {};
				\smlv (v12) at (0.5,0) {};
				\smlv (v13) at (1,0) {};

				\smlv (v21) at (0,0.5) {};
				\smlv (v22) at (0.5,0.5) {};
				\smlv (v23) at (1,0.5) {};

				\smlv (v31) at (0,1) {};
				
				\path
				(v11) edge (v12)
				(v13) edge (v23)
				;
			\end{scope}
			
			\draw [to reversed-to,red](5.3,0.5) -- (5.7,0.5);
			\begin{scope}[xshift=6cm]			
				\smlv (v11) at (0,0) {};
				\smlv (v12) at (0.5,0) {};
				\smlv (v13) at (1,0) {};

				\smlv (v21) at (0,0.5) {};
				\smlv (v22) at (0.5,0.5) {};
				\smlv (v23) at (1,0.5) {};

				\smlv (v31) at (0,1) {};
				
				\path
				(v11) edge (v12)
				(v21) edge (v31)
				(v22) edge (v23)
				;
			\end{scope}
			
			\draw [-to,red](7.3,0.5) -- (7.7,0.5);
			\begin{scope}[xshift=8cm]			
				\smlv (v11) at (0,0) {};
				\smlv (v12) at (0.5,0) {};
				\smlv (v13) at (1,0) {};

				\smlv (v21) at (0,0.5) {};
				\smlv (v22) at (0.5,0.5) {};
				\smlv (v23) at (1,0.5) {};

				\smlv (v31) at (0,1) {};
				
				\path
				(v21) edge (v31)
				(v22) edge (v23)
				;
			\end{scope}
			
			\draw [-to,red](7,1) -- (7.7,2);
			\begin{scope}[xshift=8cm,yshift=1.7cm]			
				\smlv (v11) at (0,0) {};
				\smlv (v12) at (0.5,0) {};
				\smlv (v13) at (1,0) {};

				\smlv (v21) at (0,0.5) {};
				\smlv (v22) at (0.5,0.5) {};
				\smlv (v23) at (1,0.5) {};

				\smlv (v31) at (0,1) {};
				
				\path
				(v11) edge (v12)
				(v22) edge (v23)
				;
			\end{scope}
			
			\draw [to reversed-to,red](9.3,0.5) -- (9.7,0.5);
			\begin{scope}[xshift=10cm]			
				\smlv (v11) at (0,0) {};
				\smlv (v12) at (0.5,0) {};
				\smlv (v13) at (1,0) {};

				\smlv (v21) at (0,0.5) {};
				\smlv (v22) at (0.5,0.5) {};
				\smlv (v23) at (1,0.5) {};

				\smlv (v31) at (0,1) {};
				
				\path
				(v12) edge (v13)
				(v21) edge (v31)
				(v22) edge (v23)
				;
			\end{scope}
			
			\draw [-to,red](11.3,0.5) -- (11.7,0.5);
			\begin{scope}[xshift=12cm]			
				\smlv (v11) at (0,0) {};
				\smlv (v12) at (0.5,0) {};
				\smlv (v13) at (1,0) {};

				\smlv (v21) at (0,0.5) {};
				\smlv (v22) at (0.5,0.5) {};
				\smlv (v23) at (1,0.5) {};

				\smlv (v31) at (0,1) {};
				
				\path
				(v12) edge (v13)
				(v22) edge (v23)
				;
			\end{scope}
			
			\draw [-to,red](11,1) -- (11.7,2);
			\begin{scope}[xshift=12cm,yshift=1.7cm]			
				\smlv (v11) at (0,0) {};
				\smlv (v12) at (0.5,0) {};
				\smlv (v13) at (1,0) {};

				\smlv (v21) at (0,0.5) {};
				\smlv (v22) at (0.5,0.5) {};
				\smlv (v23) at (1,0.5) {};

				\smlv (v31) at (0,1) {};
				
				\path
				(v12) edge (v13)
				(v21) edge (v31)
				;
			\end{scope}
		\end{tikzpicture}
		\caption{All possible (maximal) $\M^*$-paths that start from $\psi_{11}$. Note that each of them ends at a $1$-simplex that is paired off with a $0$-simplex in $\M^*$.}\label{psi1}
	\end{figure}
	
	In Fig.~\ref{psi2}, we describe all possible $\M^*$-paths that start from $\psi_{12}$, and end either at a critical $1$-simplex or at a $1$-simplex that is paired off with a $0$-simplex in $\M^*$.
	
	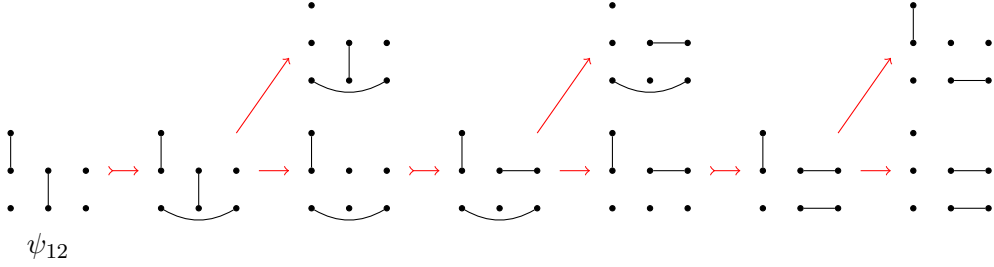
\begin{figure}[!ht]
		\centering
		\begin{tikzpicture}
			\begin{scope}
				\node at (0.5,-0.5) {$\psi_{12}$};
				
				\smlv (v11) at (0,0) {};
				\smlv (v12) at (0.5,0) {};
				\smlv (v13) at (1,0) {};

				\smlv (v21) at (0,0.5) {};
				\smlv (v22) at (0.5,0.5) {};
				\smlv (v23) at (1,0.5) {};

				\smlv (v31) at (0,1) {};
				
				\path
				(v12) edge (v22)
				(v21) edge (v31)
				;
			\end{scope}
			
			\draw [to reversed-to,red](1.3,0.5) -- (1.7,0.5);
			\begin{scope}[xshift=2cm]			
				\smlv (v11) at (0,0) {};
				\smlv (v12) at (0.5,0) {};
				\smlv (v13) at (1,0) {};

				\smlv (v21) at (0,0.5) {};
				\smlv (v22) at (0.5,0.5) {};
				\smlv (v23) at (1,0.5) {};

				\smlv (v31) at (0,1) {};
				
				\path
				(v12) edge (v22)
				(v21) edge (v31)
				;
				\path[bend right]
				(v11) edge (v13)
				;
			\end{scope}
			
			\draw [-to,red](3.3,0.5) -- (3.7,0.5);
			\begin{scope}[xshift=4cm]			
				\smlv (v11) at (0,0) {};
				\smlv (v12) at (0.5,0) {};
				\smlv (v13) at (1,0) {};

				\smlv (v21) at (0,0.5) {};
				\smlv (v22) at (0.5,0.5) {};
				\smlv (v23) at (1,0.5) {};

				\smlv (v31) at (0,1) {};
				
				\path
				(v21) edge (v31)
				;
				\path[bend right]
				(v11) edge (v13)
				;
			\end{scope}
			
			\draw [-to,red](3,1) -- (3.7,2);
			\begin{scope}[xshift=4cm,yshift=1.7cm]			
				\smlv (v11) at (0,0) {};
				\smlv (v12) at (0.5,0) {};
				\smlv (v13) at (1,0) {};

				\smlv (v21) at (0,0.5) {};
				\smlv (v22) at (0.5,0.5) {};
				\smlv (v23) at (1,0.5) {};

				\smlv (v31) at (0,1) {};
				
				\path
				(v12) edge (v22)
				;
				\path[bend right]
				(v11) edge (v13)
				;
			\end{scope}
			
			\draw [to reversed-to,red](5.3,0.5) -- (5.7,0.5);
			\begin{scope}[xshift=6cm]			
				\smlv (v11) at (0,0) {};
				\smlv (v12) at (0.5,0) {};
				\smlv (v13) at (1,0) {};

				\smlv (v21) at (0,0.5) {};
				\smlv (v22) at (0.5,0.5) {};
				\smlv (v23) at (1,0.5) {};

				\smlv (v31) at (0,1) {};
				
				\path
				(v21) edge (v31)
				(v22) edge (v23)
				;
				\path[bend right]
				(v11) edge (v13)
				;
			\end{scope}
			
			\draw [-to,red](7.3,0.5) -- (7.7,0.5);
			\begin{scope}[xshift=8cm]			
				\smlv (v11) at (0,0) {};
				\smlv (v12) at (0.5,0) {};
				\smlv (v13) at (1,0) {};

				\smlv (v21) at (0,0.5) {};
				\smlv (v22) at (0.5,0.5) {};
				\smlv (v23) at (1,0.5) {};

				\smlv (v31) at (0,1) {};
				
				\path
				(v21) edge (v31)
				(v22) edge (v23)
				;
			\end{scope}
			
			\draw [-to,red](7,1) -- (7.7,2);
			\begin{scope}[xshift=8cm,yshift=1.7cm]			
				\smlv (v11) at (0,0) {};
				\smlv (v12) at (0.5,0) {};
				\smlv (v13) at (1,0) {};

				\smlv (v21) at (0,0.5) {};
				\smlv (v22) at (0.5,0.5) {};
				\smlv (v23) at (1,0.5) {};

				\smlv (v31) at (0,1) {};
				
				\path
				(v22) edge (v23)
				;
				\path[bend right]
				(v11) edge (v13)
				;
			\end{scope}
			
			\draw [to reversed-to,red](9.3,0.5) -- (9.7,0.5);
			\begin{scope}[xshift=10cm]			
				\smlv (v11) at (0,0) {};
				\smlv (v12) at (0.5,0) {};
				\smlv (v13) at (1,0) {};

				\smlv (v21) at (0,0.5) {};
				\smlv (v22) at (0.5,0.5) {};
				\smlv (v23) at (1,0.5) {};

				\smlv (v31) at (0,1) {};
				
				\path
				(v12) edge (v13)
				(v21) edge (v31)
				(v22) edge (v23)
				;
			\end{scope}
			
			\draw [-to,red](11.3,0.5) -- (11.7,0.5);
			\begin{scope}[xshift=12cm]			
				\smlv (v11) at (0,0) {};
				\smlv (v12) at (0.5,0) {};
				\smlv (v13) at (1,0) {};

				\smlv (v21) at (0,0.5) {};
				\smlv (v22) at (0.5,0.5) {};
				\smlv (v23) at (1,0.5) {};

				\smlv (v31) at (0,1) {};
				
				\path
				(v12) edge (v13)
				(v22) edge (v23)
				;
			\end{scope}
			
			\draw [-to,red](11,1) -- (11.7,2);
			\begin{scope}[xshift=12cm,yshift=1.7cm]			
				\smlv (v11) at (0,0) {};
				\smlv (v12) at (0.5,0) {};
				\smlv (v13) at (1,0) {};

				\smlv (v21) at (0,0.5) {};
				\smlv (v22) at (0.5,0.5) {};
				\smlv (v23) at (1,0.5) {};

				\smlv (v31) at (0,1) {};
				
				\path
				(v12) edge (v13)
				(v21) edge (v31)
				;
			\end{scope}
		\end{tikzpicture}
		\caption{All possible (maximal) $\M^*$-paths that start from $\psi_{12}$. Note that each of them ends at a $1$-simplex that is paired off with a $0$-simplex in $\M^*$.}\label{psi2}
	\end{figure}
	
	In Fig.~\ref{psi3}, we describe all possible $\M^*$-paths that start from $\psi_{13}$, and end either at a critical $1$-simplex or at a $1$-simplex that is paired off with a $0$-simplex in $\M^*$.
	\begin{figure}[!ht]
		\centering
		\begin{tikzpicture}
			\begin{scope}
				\node at (0.5,-0.5) {$\psi_{13}$};
				
				\smlv (v11) at (0,0) {};
				\smlv (v12) at (0.5,0) {};
				\smlv (v13) at (1,0) {};

				\smlv (v21) at (0,0.5) {};
				\smlv (v22) at (0.5,0.5) {};
				\smlv (v23) at (1,0.5) {};

				\smlv (v31) at (0,1) {};
				
				\path
				(v12) edge (v22)
				(v13) edge (v23)
				;
			\end{scope}
			
			\draw [to reversed-to,red](1.3,0.5) -- (1.7,0.5);
			\begin{scope}[xshift=2cm]
				\smlv (v11) at (0,0) {};
				\smlv (v12) at (0.5,0) {};
				\smlv (v13) at (1,0) {};

				\smlv (v21) at (0,0.5) {};
				\smlv (v22) at (0.5,0.5) {};
				\smlv (v23) at (1,0.5) {};

				\smlv (v31) at (0,1) {};
				
				\path
				(v12) edge (v22)
				(v13) edge (v23)
				;
				\path[bend left]
				(v11) edge (v31)
				;
			\end{scope}
			
			\draw [-to, red](3.3,0.7) -- (3.7,1);
			\begin{scope}[xshift=4cm, yshift=1cm]
				\smlv (v11) at (0,0) {};
				\smlv (v12) at (0.5,0) {};
				\smlv (v13) at (1,0) {};

				\smlv (v21) at (0,0.5) {};
				\smlv (v22) at (0.5,0.5) {};
				\smlv (v23) at (1,0.5) {};

				\smlv (v31) at (0,1) {};
				
				\path
				(v12) edge (v22)
				;
				\path[bend left]
				(v11) edge (v31)
				;
			\end{scope}
			
			\draw [-to, red](3.3,0.4) -- (3.7,0);
			\begin{scope}[xshift=4cm, yshift=-1cm]
				\smlv (v11) at (0,0) {};
				\smlv (v12) at (0.5,0) {};
				\smlv (v13) at (1,0) {};

				\smlv (v21) at (0,0.5) {};
				\smlv (v22) at (0.5,0.5) {};
				\smlv (v23) at (1,0.5) {};

				\smlv (v31) at (0,1) {};
				
				\path
				(v13) edge (v23)
				;
				\path[bend left]
				(v11) edge (v31)
				;
			\end{scope}
			
			\draw [to reversed-to,red](5.3,1.8) -- (5.7,2);
			\begin{scope}[xshift=6cm, yshift=2cm]
				\smlv (v11) at (0,0) {};
				\smlv (v12) at (0.5,0) {};
				\smlv (v13) at (1,0) {};

				\smlv (v21) at (0,0.5) {};
				\smlv (v22) at (0.5,0.5) {};
				\smlv (v23) at (1,0.5) {};

				\smlv (v31) at (0,1) {};
				
				\path
				(v12) edge (v22)
				;
				\path[bend left]
				(v11) edge (v31)
				(v21) edge (v23)
				;
			\end{scope}
			
			\draw [to reversed-to,red](5.3,-0.8) -- (5.7,-1);
			\begin{scope}[xshift=6cm, yshift=-2cm]
				\smlv (v11) at (0,0) {};
				\smlv (v12) at (0.5,0) {};
				\smlv (v13) at (1,0) {};

				\smlv (v21) at (0,0.5) {};
				\smlv (v22) at (0.5,0.5) {};
				\smlv (v23) at (1,0.5) {};

				\smlv (v31) at (0,1) {};
				
				\path
				(v13) edge (v23)
				(v21) edge (v22)
				;
				\path[bend left]
				(v11) edge (v31)
				;
			\end{scope}
			
			\draw [-to,red](7.3,2.8) -- (7.7,3);
			\begin{scope}[xshift=8cm, yshift=3cm]
				\smlv (v11) at (0,0) {};
				\smlv (v12) at (0.5,0) {};
				\smlv (v13) at (1,0) {};

				\smlv (v21) at (0,0.5) {};
				\smlv (v22) at (0.5,0.5) {};
				\smlv (v23) at (1,0.5) {};

				\smlv (v31) at (0,1) {};
				
				\path[bend left]
				(v11) edge (v31)
				(v21) edge (v23)
				;
			\end{scope}
			
			\draw [-to,red](7.3,2.2) -- (7.7,2);
			\begin{scope}[xshift=8cm, yshift=1cm]
				\smlv (v11) at (0,0) {};
				\smlv (v12) at (0.5,0) {};
				\smlv (v13) at (1,0) {};

				\smlv (v21) at (0,0.5) {};
				\smlv (v22) at (0.5,0.5) {};
				\smlv (v23) at (1,0.5) {};

				\smlv (v31) at (0,1) {};
				
				\path
				(v12) edge (v22)
				;
				\path[bend left]
				(v21) edge (v23)
				;
			\end{scope}
			
			\draw [-to,red](7.3,-1.2) -- (7.7,-1);
			\begin{scope}[xshift=8cm, yshift=-1cm]
				\smlv (v11) at (0,0) {};
				\smlv (v12) at (0.5,0) {};
				\smlv (v13) at (1,0) {};

				\smlv (v21) at (0,0.5) {};
				\smlv (v22) at (0.5,0.5) {};
				\smlv (v23) at (1,0.5) {};

				\smlv (v31) at (0,1) {};
				
				\path
				(v13) edge (v23)
				(v21) edge (v22)
				;
			\end{scope}
			
			\draw [-to,red](7.3,-1.8) -- (7.7,-2);
			\begin{scope}[xshift=8cm, yshift=-3cm]
				\smlv (v11) at (0,0) {};
				\smlv (v12) at (0.5,0) {};
				\smlv (v13) at (1,0) {};

				\smlv (v21) at (0,0.5) {};
				\smlv (v22) at (0.5,0.5) {};
				\smlv (v23) at (1,0.5) {};

				\smlv (v31) at (0,1) {};
				
				\path
				(v21) edge (v22)
				;
				\path[bend left]
				(v11) edge (v31)
				;
			\end{scope}
			
			\draw [to reversed-to,red](9,4) -- (9.7,4.7);
			\begin{scope}[xshift=10cm, yshift=5cm]
				\smlv (v11) at (0,0) {};
				\smlv (v12) at (0.5,0) {};
				\smlv (v13) at (1,0) {};

				\smlv (v21) at (0,0.5) {};
				\smlv (v22) at (0.5,0.5) {};
				\smlv (v23) at (1,0.5) {};

				\smlv (v31) at (0,1) {};
				
				\path
				(v12) edge (v13)
				;
				\path[bend left]
				(v11) edge (v31)
				(v21) edge (v23)
				;
			\end{scope}
			
			\draw [to reversed-to,red](9.3,1.7) -- (9.7,2);
			\begin{scope}[xshift=10cm, yshift=2cm]
				\smlv (v11) at (0,0) {};
				\smlv (v12) at (0.5,0) {};
				\smlv (v13) at (1,0) {};

				\smlv (v21) at (0,0.5) {};
				\smlv (v22) at (0.5,0.5) {};
				\smlv (v23) at (1,0.5) {};
				
				\smlv (v31) at (0,1) {};
				
				\path
				(v12) edge (v22)
				;
				\path[bend left]
				(v13) edge (v11)
				(v21) edge (v23)
				;
			\end{scope}
			
			\draw [to reversed-to,red](9.3,-0.7) -- (9.7,-1);
			\begin{scope}[xshift=10cm, yshift=-2cm]
				\smlv (v11) at (0,0) {};
				\smlv (v12) at (0.5,0) {};
				\smlv (v13) at (1,0) {};

				\smlv (v21) at (0,0.5) {};
				\smlv (v22) at (0.5,0.5) {};
				\smlv (v23) at (1,0.5) {};

				\smlv (v31) at (0,1) {};
				
				\path
				(v11) edge (v12)
				(v13) edge (v23)
				(v21) edge (v22)
				;
			\end{scope}

			\draw [to reversed-to,red](9.3,-3) -- (9.7,-3.7);
			\begin{scope}[xshift=10cm, yshift=-5cm]
				\smlv (v11) at (0,0) {};
				\smlv (v12) at (0.5,0) {};
				\smlv (v13) at (1,0) {};

				\smlv (v21) at (0,0.5) {};
				\smlv (v22) at (0.5,0.5) {};
				\smlv (v23) at (1,0.5) {};

				\smlv (v31) at (0,1) {};
				
				\path
				(v12) edge (v13)
				(v21) edge (v22)
				;
				\path[bend left]
				(v11) edge (v31)
				;
			\end{scope}
			
			\draw [-to,red](11.3,6) -- (11.7,6.7);
			\begin{scope}[xshift=12cm, yshift=7cm]
				\smlv (v11) at (0,0) {};
				\smlv (v12) at (0.5,0) {};
				\smlv (v13) at (1,0) {};

				\smlv (v21) at (0,0.5) {};
				\smlv (v22) at (0.5,0.5) {};
				\smlv (v23) at (1,0.5) {};

				\smlv (v31) at (0,1) {};
				
				\path
				(v12) edge (v13)
				;
				\path[bend left]
				(v21) edge (v23)
				;
			\end{scope}
			
			\draw [-to,red](11.3,5.5) -- (11.7,5.5);
			\begin{scope}[xshift=12cm, yshift=5cm]
				\smlv (v11) at (0,0) {};
				\smlv (v12) at (0.5,0) {};
				\smlv (v13) at (1,0) {};

				\smlv (v21) at (0,0.5) {};
				\smlv (v22) at (0.5,0.5) {};
				\smlv (v23) at (1,0.5) {};

				\smlv (v31) at (0,1) {};
				
				\path
				(v12) edge (v13)
				;
				\path[bend left]
				(v11) edge (v31)
				;
			\end{scope}
			
			\draw [-to,red](11.3,2.8) -- (11.7,3);
			\begin{scope}[xshift=12cm, yshift=3cm]
				\smlv (v11) at (0,0) {};
				\smlv (v12) at (0.5,0) {};
				\smlv (v13) at (1,0) {};

				\smlv (v21) at (0,0.5) {};
				\smlv (v22) at (0.5,0.5) {};
				\smlv (v23) at (1,0.5) {};

				\smlv (v31) at (0,1) {};
				
				\path
				(v12) edge (v22)
				;
				\path[bend left]
				(v13) edge (v11)
				;
			\end{scope}
			
			\draw [-to,red](11.3,2.2) -- (11.7,2);
			\begin{scope}[xshift=12cm, yshift=1cm]
				\node at (1.7,0.5) {$=\sigma_4$};
				\node[red] at (0.5,-0.5) {(critical)};
				
				\smlv (v11) at (0,0) {};
				\smlv (v12) at (0.5,0) {};
				\smlv (v13) at (1,0) {};

				\smlv (v21) at (0,0.5) {};
				\smlv (v22) at (0.5,0.5) {};
				\smlv (v23) at (1,0.5) {};

				\smlv (v31) at (0,1) {};
				
				\path[bend left]
				(v13) edge (v11)
				(v21) edge (v23)
				;
			\end{scope}
			
			\draw [-to,red](11.3,-1.2) -- (11.7,-1);
			\begin{scope}[xshift=12cm, yshift=-1cm]
				\node at (1.7,0.5) {$=\sigma_1$};
				\node[red] at (0.5,-0.5) {(critical)};
				
				\smlv (v11) at (0,0) {};
				\smlv (v12) at (0.5,0) {};
				\smlv (v13) at (1,0) {};

				\smlv (v21) at (0,0.5) {};
				\smlv (v22) at (0.5,0.5) {};
				\smlv (v23) at (1,0.5) {};

				\smlv (v31) at (0,1) {};
				
				\path
				(v11) edge (v12)
				(v21) edge (v22)
				;
			\end{scope}
			
			\draw [-to,red](11.3,-1.8) -- (11.7,-2);
			\begin{scope}[xshift=12cm, yshift=-3cm]
				\smlv (v11) at (0,0) {};
				\smlv (v12) at (0.5,0) {};
				\smlv (v13) at (1,0) {};

				\smlv (v21) at (0,0.5) {};
				\smlv (v22) at (0.5,0.5) {};
				\smlv (v23) at (1,0.5) {};

				\smlv (v31) at (0,1) {};
				
				\path
				(v11) edge (v12)
				(v13) edge (v23)
				;
			\end{scope}
			
			\draw [-to,red](11.3,-4.5) -- (11.7,-4.5);
			\begin{scope}[xshift=12cm, yshift=-5cm]
				\smlv (v11) at (0,0) {};
				\smlv (v12) at (0.5,0) {};
				\smlv (v13) at (1,0) {};

				\smlv (v21) at (0,0.5) {};
				\smlv (v22) at (0.5,0.5) {};
				\smlv (v23) at (1,0.5) {};

				\smlv (v31) at (0,1) {};
				
				\path
				(v12) edge (v13)
				;
				\path[bend left]
				(v11) edge (v31)
				;
			\end{scope}
			
			\draw [-to,red](11.3,-5.5) -- (11.7,-6);
			\begin{scope}[xshift=12cm, yshift=-7cm]
				\smlv (v11) at (0,0) {};
				\smlv (v12) at (0.5,0) {};
				\smlv (v13) at (1,0) {};

				\smlv (v21) at (0,0.5) {};
				\smlv (v22) at (0.5,0.5) {};
				\smlv (v23) at (1,0.5) {};

				\smlv (v31) at (0,1) {};
				
				\path
				(v12) edge (v13)
				(v21) edge (v22)
				;
			\end{scope}
		\end{tikzpicture}
		\caption{All possible (maximal) $\M^*$-paths that start from $\psi_{13}$. Note that only two of them end at a critical $1$-simplex.}\label{psi3}
	\end{figure}
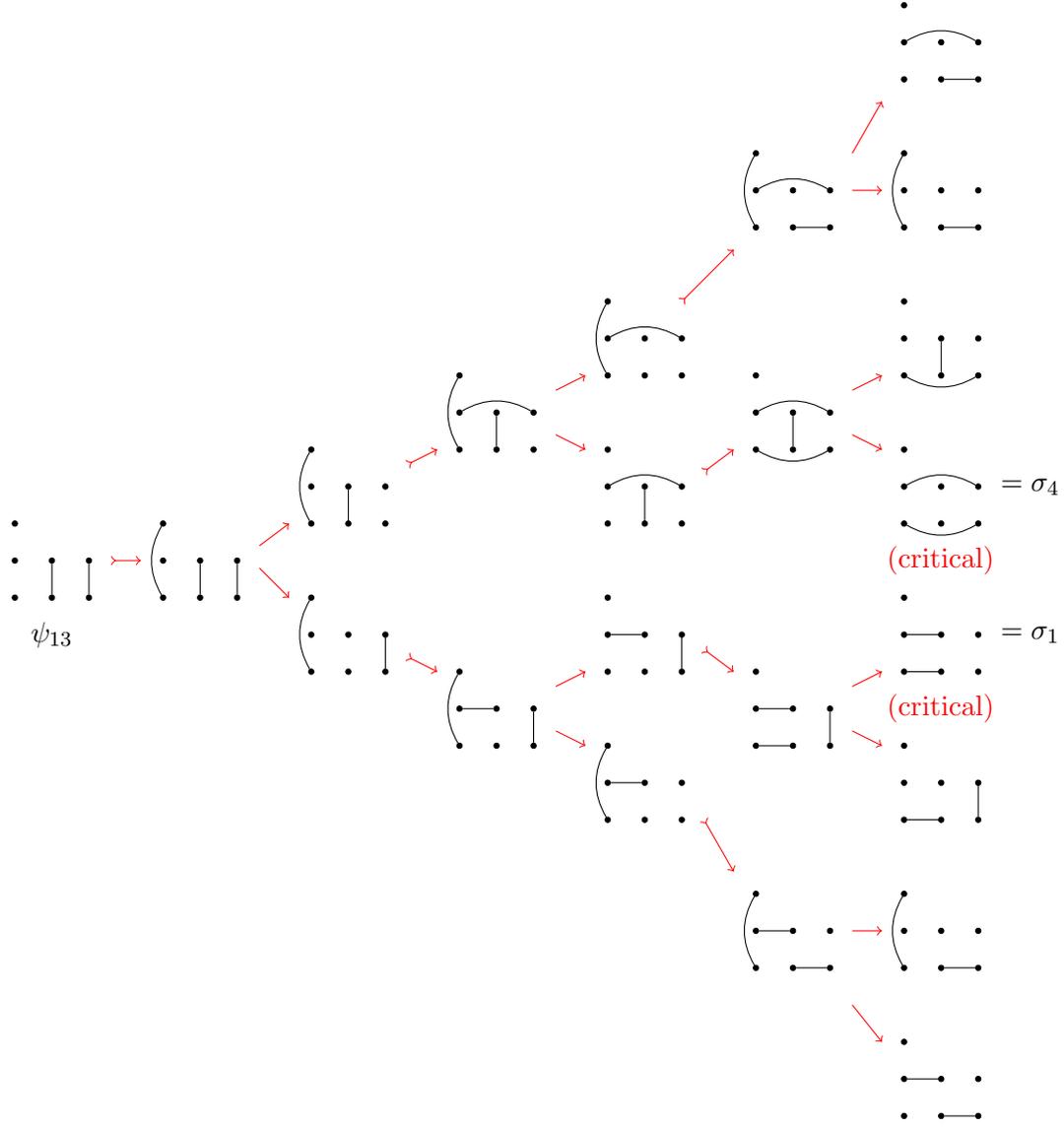
	
	We observe that there are exactly two $\M^*$-paths that start from a $1$-simplex contained in $\eta_1$, and end at a critical $1$-simplex as shown in Fig.~\ref{eta1} (see also Fig.~\ref{tau12} and Fig.~\ref{tau42} in Appendix~\ref{append1}).
	\begin{figure}[!ht]
		\centering
		\begin{tikzpicture}
			\begin{scope}[scale=0.96, transform shape]
			\begin{scope}[xshift=-2cm]
				\node at (0.5,-0.5) {$\eta_1$};
				
				\smlv (v11) at (0,0) {};
				\smlv (v12) at (0.5,0) {};
				\smlv (v13) at (1,0) {};

				\smlv (v21) at (0,0.5) {};
				\smlv (v22) at (0.5,0.5) {};
				\smlv (v23) at (1,0.5) {};

				\smlv (v31) at (0,1) {};
				
				\path
				(v12) edge (v22)
				(v13) edge (v23)
				(v21) edge (v31)
				;
			\end{scope}
			
			\draw [-to,red](-0.7,1) -- (-0.3,1.5);
			\begin{scope}[yshift=1.2cm]
				\begin{scope}
					\node at (0.5,-0.5) {$\psi_{13}$};
					
					\smlv (v11) at (0,0) {};
					\smlv (v12) at (0.5,0) {};
					\smlv (v13) at (1,0) {};

					\smlv (v21) at (0,0.5) {};
					\smlv (v22) at (0.5,0.5) {};
					\smlv (v23) at (1,0.5) {};

					\smlv (v31) at (0,1) {};
					
					\path
					(v12) edge (v22)
					(v13) edge (v23)
					;
				\end{scope}
				
				\draw [to reversed-to,red](1.3,0.5) -- (1.7,0.5);
				\begin{scope}[xshift=2cm]	
					\smlv (v11) at (0,0) {};
					\smlv (v12) at (0.5,0) {};
					\smlv (v13) at (1,0) {};

					\smlv (v21) at (0,0.5) {};
					\smlv (v22) at (0.5,0.5) {};
					\smlv (v23) at (1,0.5) {};

					\smlv (v31) at (0,1) {};
					
					\path
					(v12) edge (v22)
					(v13) edge (v23)
					;
					\path[bend left]
					(v11) edge (v31)
					;
				\end{scope}
				
				\draw [-to,red](3.3,0.5) -- (3.7,0.5);
				\begin{scope}[xshift=4cm]		
					\smlv (v11) at (0,0) {};
					\smlv (v12) at (0.5,0) {};
					\smlv (v13) at (1,0) {};

					\smlv (v21) at (0,0.5) {};
					\smlv (v22) at (0.5,0.5) {};
					\smlv (v23) at (1,0.5) {};

					\smlv (v31) at (0,1) {};
					
					\path
					(v12) edge (v22)
					;
					\path[bend left]
					(v11) edge (v31)
					;
				\end{scope}
				
				\draw [to reversed-to,red](5.3,0.5) -- (5.7,0.5);
				\begin{scope}[xshift=6cm]		
					\smlv (v11) at (0,0) {};
					\smlv (v12) at (0.5,0) {};
					\smlv (v13) at (1,0) {};

					\smlv (v21) at (0,0.5) {};
					\smlv (v22) at (0.5,0.5) {};
					\smlv (v23) at (1,0.5) {};

					\smlv (v31) at (0,1) {};
					
					\path
					(v12) edge (v22)
					;
					\path[bend left]
					(v11) edge (v31)
					(v21) edge (v23)
					;
				\end{scope}
				
				\draw [-to,red](7.3,0.5) -- (7.7,0.5);
				\begin{scope}[xshift=8cm]	
					\smlv (v11) at (0,0) {};
					\smlv (v12) at (0.5,0) {};
					\smlv (v13) at (1,0) {};

					\smlv (v21) at (0,0.5) {};
					\smlv (v22) at (0.5,0.5) {};
					\smlv (v23) at (1,0.5) {};

					\smlv (v31) at (0,1) {};
					
					\path
					(v12) edge (v22)
					;
					\path[bend left]
					(v21) edge (v23)
					;
				\end{scope}
				
				\draw [to reversed-to,red](9.3,0.5) -- (9.7,0.5);
				\begin{scope}[xshift=10cm]			
					\smlv (v11) at (0,0) {};
					\smlv (v12) at (0.5,0) {};
					\smlv (v13) at (1,0) {};

					\smlv (v21) at (0,0.5) {};
					\smlv (v22) at (0.5,0.5) {};
					\smlv (v23) at (1,0.5) {};

					\smlv (v31) at (0,1) {};
					
					\path
					(v12) edge (v22)
					;
					\path[bend left]
					(v13) edge (v11)
					(v21) edge (v23)
					;
				\end{scope}
				
				\draw [-to,red](11.3,0.5) -- (11.7,0.5);
				\begin{scope}[xshift=12cm]
					\node at (0.5,-0.5) {$\sigma_4$};
					
					\smlv (v11) at (0,0) {};
					\smlv (v12) at (0.5,0) {};
					\smlv (v13) at (1,0) {};

					\smlv (v21) at (0,0.5) {};
					\smlv (v22) at (0.5,0.5) {};
					\smlv (v23) at (1,0.5) {};

					\smlv (v31) at (0,1) {};
					
					\path[bend left]
					(v13) edge (v11)
					(v21) edge (v23)
					;
				\end{scope}
			\end{scope}
			
			\draw [-to,red](-0.7,0) -- (-0.3,-0.5);
			\begin{scope}[yshift=-1.2cm]
				\begin{scope}
					\node at (0.5,-0.5) {$\psi_{13}$};
					
					\smlv (v11) at (0,0) {};
					\smlv (v12) at (0.5,0) {};
					\smlv (v13) at (1,0) {};

					\smlv (v21) at (0,0.5) {};
					\smlv (v22) at (0.5,0.5) {};
					\smlv (v23) at (1,0.5) {};

					\smlv (v31) at (0,1) {};
					
					\path
					(v12) edge (v22)
					(v13) edge (v23)
					;
				\end{scope}
				
				\draw [to reversed-to,red](1.3,0.5) -- (1.7,0.5);
				\begin{scope}[xshift=2cm]		
					\smlv (v11) at (0,0) {};
					\smlv (v12) at (0.5,0) {};
					\smlv (v13) at (1,0) {};

					\smlv (v21) at (0,0.5) {};
					\smlv (v22) at (0.5,0.5) {};
					\smlv (v23) at (1,0.5) {};

					\smlv (v31) at (0,1) {};
					
					\path
					(v12) edge (v22)
					(v13) edge (v23)
					;
					\path[bend left]
					(v11) edge (v31)
					;
				\end{scope}
				
				\draw [-to,red](3.3,0.5) -- (3.7,0.5);
				\begin{scope}[xshift=4cm]		
					\smlv (v11) at (0,0) {};
					\smlv (v12) at (0.5,0) {};
					\smlv (v13) at (1,0) {};

					\smlv (v21) at (0,0.5) {};
					\smlv (v22) at (0.5,0.5) {};
					\smlv (v23) at (1,0.5) {};

					\smlv (v31) at (0,1) {};
					
					\path
					(v13) edge (v23)
					;
					\path[bend left]
					(v11) edge (v31)
					;
				\end{scope}
				
				\draw [to reversed-to,red](5.3,0.5) -- (5.7,0.5);
				\begin{scope}[xshift=6cm]		
					\smlv (v11) at (0,0) {};
					\smlv (v12) at (0.5,0) {};
					\smlv (v13) at (1,0) {};

					\smlv (v21) at (0,0.5) {};
					\smlv (v22) at (0.5,0.5) {};
					\smlv (v23) at (1,0.5) {};

					\smlv (v31) at (0,1) {};
					
					\path
					(v13) edge (v23)
					(v21) edge (v22)
					;
					\path[bend left]
					(v11) edge (v31)
					;
				\end{scope}
				
				\draw [-to,red](7.3,0.5) -- (7.7,0.5);
				\begin{scope}[xshift=8cm]	
					\smlv (v11) at (0,0) {};
					\smlv (v12) at (0.5,0) {};
					\smlv (v13) at (1,0) {};

					\smlv (v21) at (0,0.5) {};
					\smlv (v22) at (0.5,0.5) {};
					\smlv (v23) at (1,0.5) {};

					\smlv (v31) at (0,1) {};
					
					\path
					(v13) edge (v23)
					(v21) edge (v22)
					;
				\end{scope}
				
				\draw [to reversed-to,red](9.3,0.5) -- (9.7,0.5);
				\begin{scope}[xshift=10cm]		
					\smlv (v11) at (0,0) {};
					\smlv (v12) at (0.5,0) {};
					\smlv (v13) at (1,0) {};

					\smlv (v21) at (0,0.5) {};
					\smlv (v22) at (0.5,0.5) {};
					\smlv (v23) at (1,0.5) {};

					\smlv (v31) at (0,1) {};
					
					\path
					(v11) edge (v12)
					(v13) edge (v23)
					(v21) edge (v22)
					;
				\end{scope}
				
				\draw [-to,red](11.3,0.5) -- (11.7,0.5);
				\begin{scope}[xshift=12cm]		
					\node at (0.5,-0.5) {$\sigma_1$};
					
					\smlv (v11) at (0,0) {};
					\smlv (v12) at (0.5,0) {};
					\smlv (v13) at (1,0) {};

					\smlv (v21) at (0,0.5) {};
					\smlv (v22) at (0.5,0.5) {};
					\smlv (v23) at (1,0.5) {};

					\smlv (v31) at (0,1) {};
					
					\path
					(v11) edge (v12)
					(v21) edge (v22)
					;
				\end{scope}
			\end{scope}
		\end{scope}
		\end{tikzpicture}
		\caption{Only two possible $\M^*$-paths that start from a $1$-simplex contained in $\eta_1$, and end at a critical $1$-simplex.}\label{eta1}
	\end{figure}
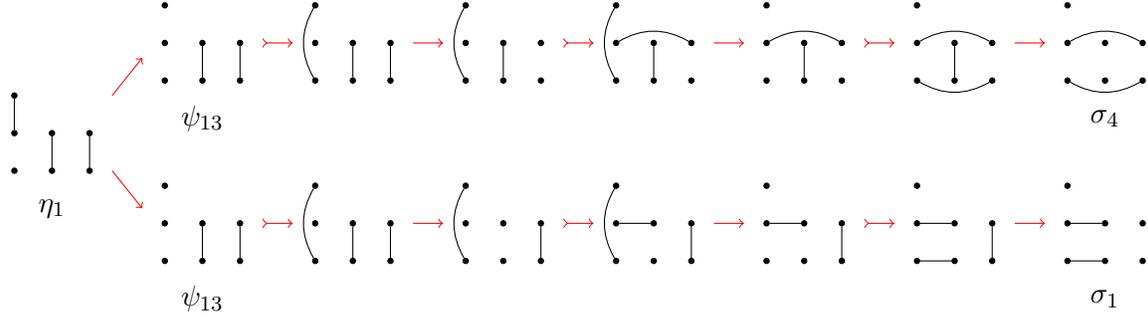
	
	Let $\gamma_1$ be the unique $\M^*$-path that starts from $\psi_{13}$, and ends at $\sigma_4$ (see Fig.~\ref{psi3} and Fig.~\ref{eta1}), i.e.,\\
	\begin{tikzcd}
		\gamma_1 : \psi_{13}=\alpha_0 \ar[r,tail,"(+1)"] & \beta_0 \ar[r,"(+1)"] & \alpha_1 \ar[r,tail,"(+1)"] & \beta_1 \ar[r,"(+1)"] & \alpha_2 \ar[r,tail,"(+1)"] & \beta_2 \ar[r,"(-1)"] & \alpha_3=\sigma_4,
	\end{tikzcd}
	where
	\begin{align*}
		\alpha_0 &= [v_2^{(1)}v_2^{(2)}, v_3^{(1)}v_3^{(2)}]=\psi_{13},  
		&\beta_0 &= [v_1^{(1)}v_1^{(3)}, v_2^{(1)}v_2^{(2)}, v_3^{(1)}v_3^{(2)}]&\\
		\alpha_1 &= [v_1^{(1)}v_1^{(3)}, v_2^{(1)}v_2^{(2)}],
		&\beta_1 &= [v_1^{(1)}v_1^{(3)}, v_2^{(1)}v_2^{(2)}, v_1^{(2)}v_3^{(2)}]&\\
		\alpha_2 &= [v_2^{(1)}v_2^{(2)}, v_1^{(2)}v_3^{(2)}],
		&\beta_2 &= [v_1^{(1)}v_3^{(1)}, v_2^{(1)}v_2^{(2)}, v_1^{(2)}v_3^{(2)}]&\\
		\alpha_3 &= [v_1^{(1)}v_3^{(1)}, v_1^{(2)}v_3^{(2)}]=\sigma_4.
	\end{align*}
	By Equation~(\ref{mult}), the multiplicity of $\gamma_1$,
	\begin{align*}
		m(\gamma_1) &= (-\langle\beta_0,\alpha_0\rangle\langle\beta_0,\alpha_1\rangle)(-\langle\beta_1,\alpha_1\rangle\langle\beta_1,\alpha_2\rangle)(-\langle\beta_2,\alpha_2\rangle\langle\beta_2,\alpha_3\rangle)\\
		&= (-(+1)(+1))(-(+1)(+1))(-(+1)(-1)) = +1.
	\end{align*}		
	Let $\gamma_2$ be the unique $\M^*$-path that starts from $\psi_{13}$, and ends at $\sigma_1$ (see Fig.~\ref{psi3} and Fig.~\ref{eta1}), i.e.,\\
	\begin{tikzcd}
		\gamma_2 : \psi_{13}=\alpha_0 \ar[r,tail,"(+1)"] & \beta_0 \ar[r,"(-1)"] & \alpha_1 \ar[r,tail,"(+1)"] & \beta_1 \ar[r,"(+1)"] & \alpha_2 \ar[r,tail,"(+1)"] & \beta_2 \ar[r,"(-1)"] & \alpha_3=\sigma_1,
	\end{tikzcd}
	where
	\begin{align*}
		\alpha_0 &= [v_2^{(1)}v_2^{(2)}, v_3^{(1)}v_3^{(2)}]=\psi_{13},  
		&\beta_0 &= [v_1^{(1)}v_1^{(3)}, v_2^{(1)}v_2^{(2)}, v_3^{(1)}v_3^{(2)}]&\\
		\alpha_1 &= [v_1^{(1)}v_1^{(3)}, v_3^{(1)}v_3^{(2)}],
		&\beta_1 &= [v_1^{(1)}v_1^{(3)}, v_3^{(1)}v_3^{(2)}, v_1^{(2)}v_2^{(2)}]&\\
		\alpha_2 &= [v_3^{(1)}v_3^{(2)}, v_1^{(2)}v_2^{(2)}],
		&\beta_2 &= [v_1^{(1)}v_2^{(1)}, v_3^{(1)}v_3^{(2)}, v_1^{(2)}v_2^{(2)}]&\\
		\alpha_3 &= [v_1^{(1)}v_2^{(1)}, v_1^{(2)}v_2^{(2)}]=\sigma_1.
	\end{align*}
	By Equation~(\ref{mult}), the multiplicity of $\gamma_2$,
	\begin{align*}
		m(\gamma_2) &= (-\langle\beta_0,\alpha_0\rangle\langle\beta_0,\alpha_1\rangle)(-\langle\beta_1,\alpha_1\rangle\langle\beta_1,\alpha_2\rangle)(-\langle\beta_2,\alpha_2\rangle\langle\beta_2,\alpha_3\rangle)\\
		&= (-(+1)(-1))(-(+1)(+1))(-(+1)(-1)) = -1.
	\end{align*}
	Therefore, from Equation~(\ref{boundary}) and Equation~(\ref{bound-coeff}),
	\begin{align*}
		\tilde{\partial}_2(\eta_1) &= \left(\langle\eta_1,\psi_{13}\rangle \cdot m(\gamma_1)\right)\cdot \sigma_4 + \left(\langle\eta_1,\psi_{13}\rangle \cdot m(\gamma_2)\right)\cdot\sigma_1\\
		&= ((+1)(+1)))\cdot\sigma_4+((+1)(-1))\cdot\sigma_1 = -\sigma_1 + \sigma_4.
	\end{align*}
\end{example}

Considering all possible $\M^*$-paths that start from a $1$-simplex contained in a critical $2$-simplex, and end at a critical $1$-simplex (refer to Appendix~\ref{append1}), we may compute the images of all twenty-four critical $2$-simplices under the boundary operator $\tilde{\partial}_2$. These are all listed in Tab.~\ref{boundary-table}. We describe an algorithmic scheme to compute the images of critical $2$-simplices under $\tilde{\partial}_2$ in Appendix~\ref{append2}, and the computation of boundaries of two more critical $2$-simplices following this scheme is also discussed there.

\begin{table}[htbp]
\input{table}
\caption{Images of all critical $2$-simplices under the boundary operator $\tilde{\partial}_2$.}\label{boundary-table}
\end{table}

\subsubsection{First and second homology groups of $M_7$}

We have $\ker(\tilde{\partial}_1)=\tilde{C}_1(M_7)=\Bigl\langle \sigma_1, \sigma_2, \sigma_3, \sigma_4 \Bigr\rangle$, and from Tab.~\ref{boundary-table}, we conclude that $\operatorname{im}(\tilde{\partial}_2)$ is generated  by $\sigma_1-\sigma_4$, $\sigma_2-\sigma_3$, $\sigma_1-\sigma_2-\sigma_3$, $\sigma_1-\sigma_2+\sigma_4$, $\sigma_1-\sigma_3+\sigma_4$, and $\sigma_2+\sigma_3-\sigma_4$. Thus, 
\begin{align*}
	& H_1(M_7) \\
	\cong& \sfrac{\ker(\tilde{\partial}_1)}{\operatorname{im}(\tilde{\partial}_2)}\\
	\cong& \Bigl\langle \sigma_1, \sigma_2, \sigma_3, \sigma_4 \Bigr\rangle \Big/ \Bigl\langle \sigma_1-\sigma_4, \sigma_2-\sigma_3, \sigma_1-\sigma_2-\sigma_3, \sigma_1-\sigma_2+\sigma_4, \sigma_1-\sigma_3+\sigma_4, \sigma_2+\sigma_3-\sigma_4\Bigr\rangle\\
	\cong& \Bigl\langle \sigma_1, \sigma_2 \Bigr\rangle \Big/ \Bigl\langle \sigma_1-2\sigma_2, 2\sigma_1-\sigma_2 \Bigr\rangle\\
	\cong& \Bigl\langle \sigma_1 \Bigr\rangle \Big/ \Bigl\langle 3\sigma_1 \Bigr\rangle\\
	\cong& \sfrac{\Z}{3\Z} = \Z_3.
\end{align*}

Finally, from Tab.~\ref{boundary-table}, we may verify that out of images of all 24 generators of $\tilde{C}_2(M_7)$ under the boundary operator $\tilde{\partial}_2$, there are exactly four linearly independent ones. Therefore, $\operatorname{rank}(\ker(\tilde{\partial}_2))=24-4 =20$. Hence, $H_2(M_7) \cong \Z^{20}$. This concludes the proof of Theorem~\ref{m7-hom}.

\begin{remark}
	After computing $H_1(M_7)$, one may also compute the second homology group of $M_7$ using the Euler characteristic of $M_7$ as follows. If $f_i$ is the number of $i$-dimensional simplices of $M_7$, then from Equation~(\ref{fv}), $f_0=21$, $f_1=105$, and $f_2=105$ (with $f_i=0$, for all $i\ge 3$). Thus, $\chi(M_7)= \sum_{i \ge 0} (-1)^if_i =21-105+105=21$. If $b_i$ is the $i$-the Betti number of $M_7$, then $b_0=1$, $b_1=0$, and $b_i=0$ for $i \ge 3$ (since $M_7$ is a connected $2$-dimensional complex with $H_1(M_7)=\Z_3$). Thus, $b_2=\chi(M_7)-b_0=20$. Since $H_2(M_7)$ is torsion-free, we get $H_2(M_7)=\Z^{20}$.
\end{remark}

\subsection{Further augmentation of the near-optimal gradient vector field}\label{opt-cons}
We note that $\M^*$ is a highly efficient (near-optimal) gradient vector field on $M_7$ as, from Observation~\ref{crit-char}, it follows that
\begin{enumerate}[(i)]
	\item out of 105 simplices of dimension 2, only 24 are critical (with the second Betti number being 20),
	\item out of 105 simplices of dimension 1, only four are critical (with the first Betti number being 0),
	\item out of 21 simplices of dimension 0, only one is critical (with the zeroth Betti number being 1).
\end{enumerate}
However, we may augment $\M^*$ even further to one that satisfies Theorem~\ref{augmnt}.

\begin{proof}[of Theorem~\ref{augmnt}]
From Fig.~\ref{eta1} and Fig.~\ref{eta2} (Appendix~\ref{append2}), we observe that there is  
\begin{enumerate}[(i)]
	\item a unique $\M^*$-path that starts from a $1$-simplex (viz., $\psi_{13}$) contained in the critical $2$-simplex $\eta_1$, and ends at the critical $1$-simplex $\sigma_4$,
	\item a unique $\M^*$-path that starts from a $1$-simplex (viz., $\psi_{22}$) contained in the critical $2$-simplex $\eta_2$, and ends at the critical $1$-simplex $\sigma_3$,
\end{enumerate}
as shown in Fig.~\ref{fig-canc}.
\begin{figure}[!ht]
	\centering
	\begin{tikzpicture}
		\begin{scope}[scale=0.96, transform shape]
		\begin{scope}[yshift=2.3cm]
			\begin{scope}[xshift=-2cm]
				\node at (0.5,-0.5) {$\eta_1$};
				
				\smlv (v11) at (0,0) {};
				\smlv (v12) at (0.5,0) {};
				\smlv (v13) at (1,0) {};

				\smlv (v21) at (0,0.5) {};
				\smlv (v22) at (0.5,0.5) {};
				\smlv (v23) at (1,0.5) {};

				\smlv (v31) at (0,1) {};
				
				\path
				(v12) edge (v22)
				(v13) edge (v23)
				(v21) edge (v31)
				;
			\end{scope}
			
			\draw [-to,red](-0.7,0.5) -- (-0.3,0.5);
			\begin{scope}
				\node at (0.5,-0.5) {$\psi_{13}$};
				
				\smlv (v11) at (0,0) {};
				\smlv (v12) at (0.5,0) {};
				\smlv (v13) at (1,0) {};

				\smlv (v21) at (0,0.5) {};
				\smlv (v22) at (0.5,0.5) {};
				\smlv (v23) at (1,0.5) {};

				\smlv (v31) at (0,1) {};
				
				\path
				(v12) edge (v22)
				(v13) edge (v23)
				;
			\end{scope}
			
			\draw [to reversed-to,red](1.3,0.5) -- (1.7,0.5);
			\begin{scope}[xshift=2cm]	
				\smlv (v11) at (0,0) {};
				\smlv (v12) at (0.5,0) {};
				\smlv (v13) at (1,0) {};

				\smlv (v21) at (0,0.5) {};
				\smlv (v22) at (0.5,0.5) {};
				\smlv (v23) at (1,0.5) {};

				\smlv (v31) at (0,1) {};
				
				\path
				(v12) edge (v22)
				(v13) edge (v23)
				;
				\path[bend left]
				(v11) edge (v31)
				;
			\end{scope}
			
			\draw [-to,red](3.3,0.5) -- (3.7,0.5);
			\begin{scope}[xshift=4cm]		
				\smlv (v11) at (0,0) {};
				\smlv (v12) at (0.5,0) {};
				\smlv (v13) at (1,0) {};

				\smlv (v21) at (0,0.5) {};
				\smlv (v22) at (0.5,0.5) {};
				\smlv (v23) at (1,0.5) {};

				\smlv (v31) at (0,1) {};
				
				\path
				(v12) edge (v22)
				;
				\path[bend left]
				(v11) edge (v31)
				;
			\end{scope}
			
			\draw [to reversed-to,red](5.3,0.5) -- (5.7,0.5);
			\begin{scope}[xshift=6cm]		
				\smlv (v11) at (0,0) {};
				\smlv (v12) at (0.5,0) {};
				\smlv (v13) at (1,0) {};

				\smlv (v21) at (0,0.5) {};
				\smlv (v22) at (0.5,0.5) {};
				\smlv (v23) at (1,0.5) {};

				\smlv (v31) at (0,1) {};
				
				\path
				(v12) edge (v22)
				;
				\path[bend left]
				(v11) edge (v31)
				(v21) edge (v23)
				;
			\end{scope}
			
			\draw [-to,red](7.3,0.5) -- (7.7,0.5);
			\begin{scope}[xshift=8cm]	
				\smlv (v11) at (0,0) {};
				\smlv (v12) at (0.5,0) {};
				\smlv (v13) at (1,0) {};

				\smlv (v21) at (0,0.5) {};
				\smlv (v22) at (0.5,0.5) {};
				\smlv (v23) at (1,0.5) {};

				\smlv (v31) at (0,1) {};
				
				\path
				(v12) edge (v22)
				;
				\path[bend left]
				(v21) edge (v23)
				;
			\end{scope}
			
			\draw [to reversed-to,red](9.3,0.5) -- (9.7,0.5);
			\begin{scope}[xshift=10cm]			
				\smlv (v11) at (0,0) {};
				\smlv (v12) at (0.5,0) {};
				\smlv (v13) at (1,0) {};

				\smlv (v21) at (0,0.5) {};
				\smlv (v22) at (0.5,0.5) {};
				\smlv (v23) at (1,0.5) {};

				\smlv (v31) at (0,1) {};
				
				\path
				(v12) edge (v22)
				;
				\path[bend left]
				(v13) edge (v11)
				(v21) edge (v23)
				;
			\end{scope}
			
			\draw [-to,red](11.3,0.5) -- (11.7,0.5);
			\begin{scope}[xshift=12cm]
				\node at (0.5,-0.5) {$\sigma_4$};
				
				\smlv (v11) at (0,0) {};
				\smlv (v12) at (0.5,0) {};
				\smlv (v13) at (1,0) {};

				\smlv (v21) at (0,0.5) {};
				\smlv (v22) at (0.5,0.5) {};
				\smlv (v23) at (1,0.5) {};

				\smlv (v31) at (0,1) {};
				
				\path[bend left]
				(v13) edge (v11)
				(v21) edge (v23)
				;
			\end{scope}
		\end{scope}
		
		\begin{scope}
			\begin{scope}[xshift=-2cm]
				\node at (0.5,-0.5) {$\eta_2$};
				
				\smlv (v11) at (0,0) {};
				\smlv (v12) at (0.5,0) {};
				\smlv (v13) at (1,0) {};

				\smlv (v21) at (0,0.5) {};
				\smlv (v22) at (0.5,0.5) {};
				\smlv (v23) at (1,0.5) {};

				\smlv (v31) at (0,1) {};
				
				\path
				(v11) edge (v22)
				(v12) edge (v21)
				(v23) edge (v31)
				;
			\end{scope}
			
			\draw [-to,red](-0.7,0.5) -- (-0.3,0.5);
			\begin{scope}
				\node at (0.5,-0.5) {$\psi_{22}$};
				
				\smlv (v11) at (0,0) {};
				\smlv (v12) at (0.5,0) {};
				\smlv (v13) at (1,0) {};

				\smlv (v21) at (0,0.5) {};
				\smlv (v22) at (0.5,0.5) {};
				\smlv (v23) at (1,0.5) {};

				\smlv (v31) at (0,1) {};
				
				\path
				(v12) edge (v21)
				(v23) edge (v31)
				;
			\end{scope}
			
			\draw [to reversed-to,red](1.3,0.5) -- (1.7,0.5);
			\begin{scope}[xshift=2cm]	
				\smlv (v11) at (0,0) {};
				\smlv (v12) at (0.5,0) {};
				\smlv (v13) at (1,0) {};

				\smlv (v21) at (0,0.5) {};
				\smlv (v22) at (0.5,0.5) {};
				\smlv (v23) at (1,0.5) {};

				\smlv (v31) at (0,1) {};
				
				\path
				(v12) edge (v21)
				(v23) edge (v31)
				;
				\path[bend right]
				(v11) edge (v13)
				;
			\end{scope}
			
			\draw [-to,red](3.3,0.5) -- (3.7,0.5);
			\begin{scope}[xshift=4cm]		
				\smlv (v11) at (0,0) {};
				\smlv (v12) at (0.5,0) {};
				\smlv (v13) at (1,0) {};

				\smlv (v21) at (0,0.5) {};
				\smlv (v22) at (0.5,0.5) {};
				\smlv (v23) at (1,0.5) {};

				\smlv (v31) at (0,1) {};
				
				\path
				(v23) edge (v31)
				;
				\path[bend right]
				(v11) edge (v13)
				;
			\end{scope}
			
			\draw [to reversed-to,red](5.3,0.5) -- (5.7,0.5);
			\begin{scope}[xshift=6cm]		
				\smlv (v11) at (0,0) {};
				\smlv (v12) at (0.5,0) {};
				\smlv (v13) at (1,0) {};

				\smlv (v21) at (0,0.5) {};
				\smlv (v22) at (0.5,0.5) {};
				\smlv (v23) at (1,0.5) {};

				\smlv (v31) at (0,1) {};
				
				\path
				(v21) edge (v22)
				(v23) edge (v31)
				;
				\path[bend right]
				(v11) edge (v13)
				;
			\end{scope}
			
			\draw [-to,red](7.3,0.5) -- (7.7,0.5);
			\begin{scope}[xshift=8cm]	
				\node at (0.5,-0.5) {$\sigma_3$};
				
				\smlv (v11) at (0,0) {};
				\smlv (v12) at (0.5,0) {};
				\smlv (v13) at (1,0) {};

				\smlv (v21) at (0,0.5) {};
				\smlv (v22) at (0.5,0.5) {};
				\smlv (v23) at (1,0.5) {};

				\smlv (v31) at (0,1) {};
				
				\path
				(v21) edge (v22)
				;
				\path[bend right]
				(v11) edge (v13)
				;
			\end{scope}
		\end{scope}
		
		\end{scope}
	\end{tikzpicture}
	\caption{Unique $\M^*$-paths starting from a $1$-simplex contained in $\eta_1$ and $\eta_2$, and ending at $\sigma_4$ and $\sigma_3$, respectively.}\label{fig-canc}
\end{figure}
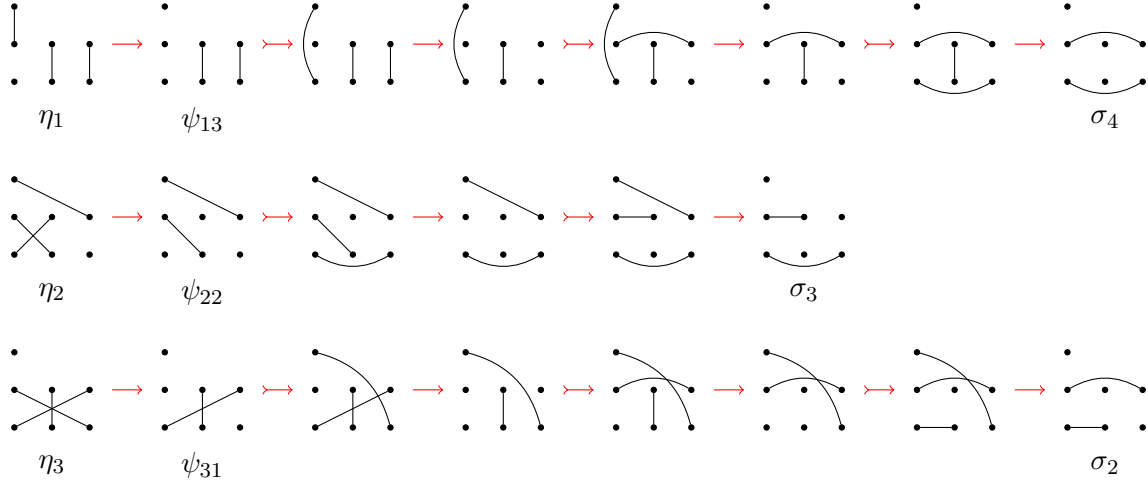
Moreover, we also observe that there is
\begin{enumerate}[(i)]
	\item no $\M^*$-path that starts from a $1$-simplex contained in $\eta_1$, and ends at $\sigma_3$,
	\item no $\M^*$-path that starts from a $1$-simplex contained in $\eta_2$, and ends at $\sigma_4$.
\end{enumerate}
Thus, the requirements of Theorem~\ref{multi-cancel} are satisfied. This allows us to apply the technique of cancellation of a critical pair by reversing paths (Theorem~\ref{cancel}) twice, and end up with a gradient vector field $\M^{**}$ from $\M^*$ such that $\eta_1$, $\sigma_4$, and $\eta_2$, $\sigma_3$ are not critical with respect to $\M^{**}$, while the criticality of all other simplices remains unchanged. Thus, with respect to $\M^{**}$, there are $24-2=22$ critical $2$-simplices, two critical $1$-simplices (viz., $\sigma_1$ and $\sigma_2$), and one critical $0$-simplex (viz., $\xi$).
\end{proof}

By Theorem~\ref{funda}, it follows that $M_7$ is homotopy equivalent to a CW complex with one $0$-cell, two $1$-cells, and twenty-two $2$-cells.

\section{Conclusion}
As $H_1(M_7)$ is nontrivial, with respect to any gradient vector field on $M_7$, there is at least one critical $1$-simplex. As a consequence, from the Morse inequalities (Theorem~\ref{morse-ineq}), it follows that with respect to any gradient vector field on $M_7$, there are at least 21 critical $2$-simplices.

We note that (from Appendix~\ref{append1}) there is an $\M^*$-path from a face of a critical (with respect to $\M^*$) $2$-simplex, say $\eta$, to the critical $1$-simplex $\sigma_1$ if and only if there is an $\M^*$-path from a face of $\eta$ to $\sigma_4$. 
Suppose we first cancel the critical pair $\eta^{(2)}, \sigma_i^{(1)}$, where $i \in \{1,4\}$, by reversing the unique gradient path $\gamma$ from a face of $\eta$ to $\sigma_i$ (by Theorem~\ref{cancel}). Then in addition to an $\M^*$-path (if one exists) from a face of another critical $2$-simplex $\eta'$ to $\sigma_{5-i}$, there is also a \emph{new} gradient path (with respect to the augmented gradient vector field) from a face of $\eta'$ to $\sigma_{5-i}$ that passes through reversed $\gamma$. Thus, the uniqueness of possible gradient paths (with respect to the augmented gradient vector field) from a face of any remaining critical $2$-simplex to $\sigma_{5-i}$, is no longer viable. So, out of $\sigma_1$ and $\sigma_4$, \emph{at most one} can be cancelled. Likewise, at most one of $\sigma_2$ and $\sigma_3$ can be cancelled. Therefore, starting with $\M^*$, the technique of cancellation of critical pairs can be applied \emph{at most} twice. This motivates us to conjecture that the augmented gradient vector field $\M^{**}$, which is obtained after applying the cancellation theorem twice, is indeed an optimal one.
\begin{conjecture}\label{conj}
	A gradient vector field on $M_7$ with respect to which there are 22 critical $2$-simplices, two critical $1$-simplices, and one critical $0$-simplex, is an optimal one.
\end{conjecture}

As mentioned before, in order to compute the (discrete) Morse homology groups of $M_7$ in an efficient manner, we constructed and worked with the near-optimal gradient vector field $\M^*$ that not only admits a low number of critical simplices of each dimension, but also, more importantly, allows an efficient (and human-comprehensible) enumeration of all possible gradient paths of interest. It would be another problem of interest to determine how ``far from random" the constructed gradient vector fields are.
\begin{problem}
	Compare the efficacy of the gradient vector field $\M^{*}$ (and the augmented one, $\M^{**}$) with those found by random methods as described by \cite{adi}, \cite{benedetti2}, \cite{benedetti1}.
\end{problem}

We also expect the techniques developed in this article to be useful for taking up problems of exploring the topology of higher dimensional matching complexes. As discussed in this article, one potential approach consists of the following two steps in order.
\begin{enumerate}
	\item Extend the initial gradient vector field constructed in Section~\ref{sec-cons}, which is optimal up to dimension $\nu_n-1$, to a sufficiently good (optimal or near-optimal) one. Some augmentation tricks used in this article may be adapted for higher dimensions. For example, the augmentation used in Example~\ref{m8top} is applicable to any $M_{2n}$, and the augmentation used in Subsection~\ref{m7nearopt} may be applied to any $M_{3n+1}$.
	\item Using the efficient extended gradient vector field, (perhaps with the aid of computer implementations) compute the homology groups.
\end{enumerate}

\begin{appendices}
\section{$\M^*$-paths ending at a critical $1$-simplex}\label{append1}
\input{append1}

\section{A scheme to compute the boundaries of critical $2$-simplices}\label{append2}
\input{append2}
\end{appendices}

\acknowledgements
\label{sec:ack}
We extend our gratitude to the anonymous referees for valuable insightful comments, corrections, and suggestions which significantly improved the overall exposition of this article. We are also thankful to Prof.\ Goutam Mukherjee for his encouragement and support during this work. We thank Prof.\ Samik Basu for introducing us to the topic of matching and chessboard complexes.

\nocite{*}
\bibliographystyle{abbrvnat}
\bibliography{matching-complex}
\label{sec:biblio}
\end{document}

%% file: table.tex
	\begin{tblr}{|[1pt]c|X|[1pt]c|X|[1pt]c|X|[1pt]c|X|[1pt]}
		\hline[1pt]
		$\eta$ & $\tilde{\partial}_2(\eta)$ & $\eta$ & $\tilde{\partial}_2(\eta)$ & $\eta$ & $\tilde{\partial}_2(\eta)$ & $\eta$ & $\tilde{\partial}_2(\eta)$\\
		\hline[1pt]
		\raisebox{-1\height}{\begin{tikzpicture}		
				\smlv (v11) at (0,0) {};
				\smlv (v12) at (0.5,0) {};
				\smlv (v13) at (1,0) {};

				\smlv (v21) at (0,0.5) {};
				\smlv (v22) at (0.5,0.5) {};
				\smlv (v23) at (1,0.5) {};

				\smlv (v31) at (0,1) {};
				
				\path
				(v11) edge (v21)
				(v12) edge (v22)
				(v13) edge (v23)
				;
		\end{tikzpicture}}
		& $(-\sigma_1+\sigma_1+\sigma_2-\sigma_3+\sigma_4-\sigma_4)=\sigma_2-\sigma_3$&
		\raisebox{-1\height}{\begin{tikzpicture}		
				\smlv (v11) at (0,0) {};
				\smlv (v12) at (0.5,0) {};
				\smlv (v13) at (1,0) {};

				\smlv (v21) at (0,0.5) {};
				\smlv (v22) at (0.5,0.5) {};
				\smlv (v23) at (1,0.5) {};

				\smlv (v31) at (0,1) {};
				
				\path
				(v11) edge (v22)
				(v12) edge (v23)
				(v21) edge (v31)
				;
		\end{tikzpicture}}
		& $\sigma_1-\sigma_2-\sigma_3$ &
		\raisebox{-1\height}{\begin{tikzpicture}		
				\smlv (v11) at (0,0) {};
				\smlv (v12) at (0.5,0) {};
				\smlv (v13) at (1,0) {};

				\smlv (v21) at (0,0.5) {};
				\smlv (v22) at (0.5,0.5) {};
				\smlv (v23) at (1,0.5) {};

				\smlv (v31) at (0,1) {};
				
				\path
				(v11) edge (v21)
				(v12) edge (v23)
				(v22) edge (v31)
				;
		\end{tikzpicture}}
		& $\sigma_1-\sigma_3+\sigma_4$ &
		\raisebox{-1\height}{\begin{tikzpicture}	
				\smlv (v11) at (0,0) {};
				\smlv (v12) at (0.5,0) {};
				\smlv (v13) at (1,0) {};

				\smlv (v21) at (0,0.5) {};
				\smlv (v22) at (0.5,0.5) {};
				\smlv (v23) at (1,0.5) {};

				\smlv (v31) at (0,1) {};
				
				\path
				(v11) edge (v21)
				(v12) edge (v22)
				(v23) edge (v31)
				;
		\end{tikzpicture}}
		& $\sigma_2+\sigma_3-\sigma_4$\\
		\hline
		\raisebox{-1\height}{\begin{tikzpicture}		
				\smlv (v11) at (0,0) {};
				\smlv (v12) at (0.5,0) {};
				\smlv (v13) at (1,0) {};

				\smlv (v21) at (0,0.5) {};
				\smlv (v22) at (0.5,0.5) {};
				\smlv (v23) at (1,0.5) {};

				\smlv (v31) at (0,1) {};
				
				\path
				(v11) edge (v21)
				(v12) edge (v23)
				(v13) edge (v22)
				;
		\end{tikzpicture}}
		& $(\sigma_1-\sigma_2+\sigma_2+\sigma_3-\sigma_3-\sigma_4)=\sigma_1-\sigma_4$&
		\raisebox{-1\height}{\begin{tikzpicture}		
				\smlv (v11) at (0,0) {};
				\smlv (v12) at (0.5,0) {};
				\smlv (v13) at (1,0) {};

				\smlv (v21) at (0,0.5) {};
				\smlv (v22) at (0.5,0.5) {};
				\smlv (v23) at (1,0.5) {};

				\smlv (v31) at (0,1) {};
				
				\path
				(v11) edge (v23)
				(v12) edge (v22)
				(v21) edge (v31)
				;
		\end{tikzpicture}}
		& $-\sigma_1+\sigma_2-\sigma_4$ &
		\raisebox{-1\height}{\begin{tikzpicture}		
				\smlv (v11) at (0,0) {};
				\smlv (v12) at (0.5,0) {};
				\smlv (v13) at (1,0) {};

				\smlv (v21) at (0,0.5) {};
				\smlv (v22) at (0.5,0.5) {};
				\smlv (v23) at (1,0.5) {};

				\smlv (v31) at (0,1) {};
				
				\path
				(v11) edge (v23)
				(v12) edge (v21)
				(v22) edge (v31)
				;
		\end{tikzpicture}}
		& $-\sigma_1+\sigma_4$ &
		\raisebox{-1\height}{\begin{tikzpicture}	
				\smlv (v11) at (0,0) {};
				\smlv (v12) at (0.5,0) {};
				\smlv (v13) at (1,0) {};

				\smlv (v21) at (0,0.5) {};
				\smlv (v22) at (0.5,0.5) {};
				\smlv (v23) at (1,0.5) {};

				\smlv (v31) at (0,1) {};
				
				\path
				(v11) edge (v22)
				(v12) edge (v21)
				(v23) edge (v31)
				;
		\end{tikzpicture}}
		& $-\sigma_2+\sigma_3$\\
		\hline
		\raisebox{-1\height}{\begin{tikzpicture}		
				\smlv (v11) at (0,0) {};
				\smlv (v12) at (0.5,0) {};
				\smlv (v13) at (1,0) {};

				\smlv (v21) at (0,0.5) {};
				\smlv (v22) at (0.5,0.5) {};
				\smlv (v23) at (1,0.5) {};

				\smlv (v31) at (0,1) {};
				
				\path
				(v11) edge (v23)
				(v12) edge (v22)
				(v13) edge (v21)
				;
		\end{tikzpicture}}
		& $(-\sigma_1+\sigma_2+\sigma_3+\sigma_4-\sigma_4)=-\sigma_1+\sigma_2+\sigma_3$ &
		\raisebox{-1\height}{\begin{tikzpicture}		
				\smlv (v11) at (0,0) {};
				\smlv (v12) at (0.5,0) {};
				\smlv (v13) at (1,0) {};

				\smlv (v21) at (0,0.5) {};
				\smlv (v22) at (0.5,0.5) {};
				\smlv (v23) at (1,0.5) {};

				\smlv (v31) at (0,1) {};
				
				\path
				(v11) edge (v22)
				(v13) edge (v23)
				(v21) edge (v31)
				;
		\end{tikzpicture}}
		& $-\sigma_1+\sigma_3-\sigma_4$ &
		\raisebox{-1\height}{\begin{tikzpicture}		
				\smlv (v11) at (0,0) {};
				\smlv (v12) at (0.5,0) {};
				\smlv (v13) at (1,0) {};

				\smlv (v21) at (0,0.5) {};
				\smlv (v22) at (0.5,0.5) {};
				\smlv (v23) at (1,0.5) {};

				\smlv (v31) at (0,1) {};
				
				\path
				(v11) edge (v21)
				(v13) edge (v23)
				(v22) edge (v31)
				;
		\end{tikzpicture}}
		& $-\sigma_1+\sigma_2+\sigma_3$ &
		\raisebox{-1\height}{\begin{tikzpicture}	
				\smlv (v11) at (0,0) {};
				\smlv (v12) at (0.5,0) {};
				\smlv (v13) at (1,0) {};

				\smlv (v21) at (0,0.5) {};
				\smlv (v22) at (0.5,0.5) {};
				\smlv (v23) at (1,0.5) {};

				\smlv (v31) at (0,1) {};
				
				\path
				(v11) edge (v21)
				(v13) edge (v22)
				(v23) edge (v31)
				;
		\end{tikzpicture}}
		& $\sigma_1-\sigma_2+\sigma_4$\\
		\hline
		\raisebox{-1\height}{\begin{tikzpicture}		
				\smlv (v11) at (0,0) {};
				\smlv (v12) at (0.5,0) {};
				\smlv (v13) at (1,0) {};

				\smlv (v21) at (0,0.5) {};
				\smlv (v22) at (0.5,0.5) {};
				\smlv (v23) at (1,0.5) {};

				\smlv (v31) at (0,1) {};
				
				\path
				(v11) edge (v22)
				(v12) edge (v21)
				(v13) edge (v23)
				;
		\end{tikzpicture}}
		& $(-\sigma_1+\sigma_1-\sigma_2-\sigma_3+\sigma_4)=-\sigma_2-\sigma_3+\sigma_4$ &
		\raisebox{-1\height}{\begin{tikzpicture}		
				\smlv (v11) at (0,0) {};
				\smlv (v12) at (0.5,0) {};
				\smlv (v13) at (1,0) {};

				\smlv (v21) at (0,0.5) {};
				\smlv (v22) at (0.5,0.5) {};
				\smlv (v23) at (1,0.5) {};

				\smlv (v31) at (0,1) {};
				
				\path
				(v11) edge (v23)
				(v13) edge (v22)
				(v21) edge (v31)
				;
		\end{tikzpicture}}
		& $-\sigma_2-\sigma_3+\sigma_4$ &
		\raisebox{-1\height}{\begin{tikzpicture}		
				\smlv (v11) at (0,0) {};
				\smlv (v12) at (0.5,0) {};
				\smlv (v13) at (1,0) {};

				\smlv (v21) at (0,0.5) {};
				\smlv (v22) at (0.5,0.5) {};
				\smlv (v23) at (1,0.5) {};

				\smlv (v31) at (0,1) {};
				
				\path
				(v11) edge (v23)
				(v13) edge (v21)
				(v22) edge (v31)
				;
		\end{tikzpicture}}
		& $\sigma_2-\sigma_3$ &
		\raisebox{-1\height}{\begin{tikzpicture}	
				\smlv (v11) at (0,0) {};
				\smlv (v12) at (0.5,0) {};
				\smlv (v13) at (1,0) {};

				\smlv (v21) at (0,0.5) {};
				\smlv (v22) at (0.5,0.5) {};
				\smlv (v23) at (1,0.5) {};

				\smlv (v31) at (0,1) {};
				
				\path
				(v11) edge (v22)
				(v13) edge (v21)
				(v23) edge (v31)
				;
		\end{tikzpicture}}
		& $\sigma_1-\sigma_4$\\
		\hline
		\raisebox{-1\height}{\begin{tikzpicture}		
				\smlv (v11) at (0,0) {};
				\smlv (v12) at (0.5,0) {};
				\smlv (v13) at (1,0) {};

				\smlv (v21) at (0,0.5) {};
				\smlv (v22) at (0.5,0.5) {};
				\smlv (v23) at (1,0.5) {};

				\smlv (v31) at (0,1) {};
				
				\path
				(v11) edge (v22)
				(v12) edge (v23)
				(v13) edge (v21)
				;
		\end{tikzpicture}}
		& $(\sigma_1-\sigma_2+\sigma_3-\sigma_3+\sigma_4)=\sigma_1-\sigma_2+\sigma_4$ &
		\raisebox{-1\height}{\begin{tikzpicture}		
				\smlv (v11) at (0,0) {};
				\smlv (v12) at (0.5,0) {};
				\smlv (v13) at (1,0) {};

				\smlv (v21) at (0,0.5) {};
				\smlv (v22) at (0.5,0.5) {};
				\smlv (v23) at (1,0.5) {};

				\smlv (v31) at (0,1) {};
				
				\path
				(v12) edge (v22)
				(v13) edge (v23)
				(v21) edge (v31)
				;
		\end{tikzpicture}}
		& $-\sigma_1+\sigma_4$ &
		\raisebox{-1\height}{\begin{tikzpicture}		
				\smlv (v11) at (0,0) {};
				\smlv (v12) at (0.5,0) {};
				\smlv (v13) at (1,0) {};

				\smlv (v21) at (0,0.5) {};
				\smlv (v22) at (0.5,0.5) {};
				\smlv (v23) at (1,0.5) {};

				\smlv (v31) at (0,1) {};
				
				\path
				(v12) edge (v21)
				(v13) edge (v23)
				(v22) edge (v31)
				;
		\end{tikzpicture}}
		& $-\sigma_1+\sigma_2-\sigma_4$ &
		\raisebox{-1\height}{\begin{tikzpicture}	
				\smlv (v11) at (0,0) {};
				\smlv (v12) at (0.5,0) {};
				\smlv (v13) at (1,0) {};

				\smlv (v21) at (0,0.5) {};
				\smlv (v22) at (0.5,0.5) {};
				\smlv (v23) at (1,0.5) {};

				\smlv (v31) at (0,1) {};
				
				\path
				(v12) edge (v21)
				(v13) edge (v22)
				(v23) edge (v31)
				;
		\end{tikzpicture}}
		& $\sigma_1-\sigma_2-\sigma_3$\\
		\hline
		\raisebox{-1\height}{\begin{tikzpicture}		
				\smlv (v11) at (0,0) {};
				\smlv (v12) at (0.5,0) {};
				\smlv (v13) at (1,0) {};

				\smlv (v21) at (0,0.5) {};
				\smlv (v22) at (0.5,0.5) {};
				\smlv (v23) at (1,0.5) {};

				\smlv (v31) at (0,1) {};
				
				\path
				(v11) edge (v23)
				(v12) edge (v21)
				(v13) edge (v22)
				;
		\end{tikzpicture}}
		& $(-\sigma_1-\sigma_2+\sigma_2+\sigma_3-\sigma_4)=-\sigma_1+\sigma_3-\sigma_4$ &
		\raisebox{-1\height}{\begin{tikzpicture}		
				\smlv (v11) at (0,0) {};
				\smlv (v12) at (0.5,0) {};
				\smlv (v13) at (1,0) {};

				\smlv (v21) at (0,0.5) {};
				\smlv (v22) at (0.5,0.5) {};
				\smlv (v23) at (1,0.5) {};

				\smlv (v31) at (0,1) {};
				
				\path
				(v12) edge (v23)
				(v13) edge (v22)
				(v21) edge (v31)
				;
		\end{tikzpicture}}
		& $-\sigma_2+\sigma_3$ &
		\raisebox{-1\height}{\begin{tikzpicture}		
				\smlv (v11) at (0,0) {};
				\smlv (v12) at (0.5,0) {};
				\smlv (v13) at (1,0) {};

				\smlv (v21) at (0,0.5) {};
				\smlv (v22) at (0.5,0.5) {};
				\smlv (v23) at (1,0.5) {};

				\smlv (v31) at (0,1) {};
				
				\path
				(v12) edge (v23)
				(v13) edge (v21)
				(v22) edge (v31)
				;
		\end{tikzpicture}}
		& $\sigma_2+\sigma_3-\sigma_4$ &
		\raisebox{-1\height}{\begin{tikzpicture}	
				\smlv (v11) at (0,0) {};
				\smlv (v12) at (0.5,0) {};
				\smlv (v13) at (1,0) {};

				\smlv (v21) at (0,0.5) {};
				\smlv (v22) at (0.5,0.5) {};
				\smlv (v23) at (1,0.5) {};

				\smlv (v31) at (0,1) {};
				
				\path
				(v12) edge (v22)
				(v13) edge (v21)
				(v23) edge (v31)
				;
		\end{tikzpicture}}
		& $\sigma_1-\sigma_3+\sigma_4$\\
		\hline[1pt]
	\end{tblr}

%% file: append2.tex
First, in each figure in Appendix~\ref{append1}, above each arrow connecting $\beta^{(2)}_i$ and $\alpha^{(1)}_j$ (with $0 \le i \le j \le i+1$), we write the incidence number $\langle\beta_i,\alpha_j\rangle$. Also, above each arrow from the critical $2$-simplex $\eta$ (leftmost column) to the $1$-simplex $\alpha_0$, we write the incidence number $\langle\eta,\alpha_0\rangle$, as shown below. 
\begin{center}
	\begin{tikzcd}
		\eta \ar[r, "\langle\eta{,}\alpha_0\rangle"] 
		& \alpha_0 \ar[r,tail,"\langle\beta_0{,}\alpha_0\rangle"] 
		& \beta_0 \ar[r,"\langle\beta_0{,}\alpha_1\rangle"] 
		& \alpha_1 \ar[r,tail,"\langle\beta_1{,}\alpha_1\rangle"] 
		& \beta_1 \ar[r,"\langle\beta_1{,}\alpha_2\rangle"] 
		& \cdots \ar[r,"\langle\beta_{k-1}{,}\alpha_k\rangle"] 
		& [1.5em]\alpha_k \ar[r,tail,"\langle\beta_k{,}\alpha_k\rangle"] 
		& [1.1em]\beta_k \ar[r,"\langle\beta_k{,}\alpha_{k+1}\rangle"] 
		& [1.3em]\alpha_{k+1}
	\end{tikzcd}
\end{center}

Next, let us choose a particular critical $2$-simplex $\eta_0$. We describe a sequence of algorithmic steps to compute $\tilde{\partial}_2(\eta_0)$ below.
\begin{description}
	\item[Step 1:] Pick a path from $\eta_0$ to a critical $1$-simplex, say $\sigma_i$, in a figure of Appendix~\ref{append1}.
	\item[Step 2:] Count the number of $-1$'s appearing above the arrows in the path. Let the count be $r$.
	\item[Step 3:] Count the number of non-critical $2$-simplices in the path (i.e., the number of $\beta^{(2)}_i$s). Let the count be $s$. The contribution of this specific path to  $\tilde{\partial}_2(\eta_0)$ is $(-1)^{r+s}\sigma_i$.
	\item[Step 4:] Compute the contribution of each possible path from $\eta_0$ to a critical $1$-simplex in all the figures of Appendix~\ref{append1}.
	\item[Step 5:] Sum of the contributions of all possible paths, from $\eta_0$ to a critical $1$-simplex, is $\tilde{\partial}_2(\eta_0)$.
\end{description}

Let us consider a couple of critical $2$-simplices and compute their images under the boundary operator $\tilde{\partial}_2$ following the scheme above in the examples below.

\begin{example}
	Let us consider the critical $2$-simplex $\eta_2 \coloneqq $ $\{v_1^{(1)}v_2^{(2)}$, $v_2^{(1)}v_1^{(2)}$, $v_3^{(2)}v_1^{(3)}\}$ as shown below.
	\[\begin{tikzpicture}		
		\node at (-0.7,0.5) {$\eta_2 = $};
		
		\smlv (v11) at (0,0) {};
		\smlv (v12) at (0.5,0) {};
		\smlv (v13) at (1,0) {};

		\smlv (v21) at (0,0.5) {};
		\smlv (v22) at (0.5,0.5) {};
		\smlv (v23) at (1,0.5) {};

		\smlv (v31) at (0,1) {};
		
		\path
		(v11) edge (v22)
		(v12) edge (v21)
		(v23) edge (v31)
		;
	\end{tikzpicture}\]
	
	We observe that there are exactly two $\M^*$-paths that start from a $1$-simplex contained in $\eta_2$, and end at a critical $1$-simplex (Fig.~\ref{tau23} and Fig.~\ref{tau31} in Appendix~\ref{append1}) as shown in Fig.~\ref{eta2} (along with the incidence numbers).
	\begin{figure}[!ht]
		\centering
		\begin{tikzpicture}
			\begin{scope}[scale=0.96, transform shape]
			\begin{scope}[xshift=-2cm]
				\node at (0.5,-0.5) {$\eta_2$};
				
				\smlv (v11) at (0,0) {};
				\smlv (v12) at (0.5,0) {};
				\smlv (v13) at (1,0) {};

				\smlv (v21) at (0,0.5) {};
				\smlv (v22) at (0.5,0.5) {};
				\smlv (v23) at (1,0.5) {};

				\smlv (v31) at (0,1) {};
				
				\path
				(v11) edge (v22)
				(v12) edge (v21)
				(v23) edge (v31)
				;
			\end{scope}
			
			\draw [-to,red](-0.7,1) -- (-0.3,1.5) node[midway,above left] {$+1$};
			\begin{scope}[yshift=1.2cm]
				\begin{scope}
					\node at (0.5,-0.5) {$\psi_{21}$};
					
					\smlv (v11) at (0,0) {};
					\smlv (v12) at (0.5,0) {};
					\smlv (v13) at (1,0) {};

					\smlv (v21) at (0,0.5) {};
					\smlv (v22) at (0.5,0.5) {};
					\smlv (v23) at (1,0.5) {};

					\smlv (v31) at (0,1) {};
					
					\path
					(v11) edge (v22)
					(v12) edge (v21)
					;
				\end{scope}
				
				\draw [to reversed-to,red](1.3,0.5) -- (1.7,0.5) node[midway,above] {$+1$};
				\begin{scope}[xshift=2cm]			
					\smlv (v11) at (0,0) {};
					\smlv (v12) at (0.5,0) {};
					\smlv (v13) at (1,0) {};

					\smlv (v21) at (0,0.5) {};
					\smlv (v22) at (0.5,0.5) {};
					\smlv (v23) at (1,0.5) {};

					\smlv (v31) at (0,1) {};
					
					\path
					(v11) edge (v22)
					(v12) edge (v21)
					;
					\path[bend right]
					(v13) edge (v31)
					;
				\end{scope}
				
				\draw [-to,red](3.3,0.5) -- (3.7,0.5) node[midway,above] {$-1$};
				\begin{scope}[xshift=4cm]			
					\smlv (v11) at (0,0) {};
					\smlv (v12) at (0.5,0) {};
					\smlv (v13) at (1,0) {};

					\smlv (v21) at (0,0.5) {};
					\smlv (v22) at (0.5,0.5) {};
					\smlv (v23) at (1,0.5) {};

					\smlv (v31) at (0,1) {};
					
					\path
					(v11) edge (v22)
					;
					\path[bend right]
					(v13) edge (v31)
					;
				\end{scope}
				
				\draw [to reversed-to,red](5.3,0.5) -- (5.7,0.5) node[midway,above] {$+1$};
				\begin{scope}[xshift=6cm]			
					\smlv (v11) at (0,0) {};
					\smlv (v12) at (0.5,0) {};
					\smlv (v13) at (1,0) {};

					\smlv (v21) at (0,0.5) {};
					\smlv (v22) at (0.5,0.5) {};
					\smlv (v23) at (1,0.5) {};

					\smlv (v31) at (0,1) {};
					
					\path
					(v11) edge (v22)
					;
					\path[bend right]
					(v13) edge (v31)
					(v23) edge (v21)
					;
				\end{scope}
				
				\draw [-to,red](7.3,0.5) -- (7.7,0.5) node[midway,above] {$+1$};
				\begin{scope}[xshift=8cm]			
					\smlv (v11) at (0,0) {};
					\smlv (v12) at (0.5,0) {};
					\smlv (v13) at (1,0) {};

					\smlv (v21) at (0,0.5) {};
					\smlv (v22) at (0.5,0.5) {};
					\smlv (v23) at (1,0.5) {};

					\smlv (v31) at (0,1) {};
					
					\path[bend right]
					(v13) edge (v31)
					(v23) edge (v21)
					;
				\end{scope}
				
				\draw [to reversed-to,red](9.3,0.5) -- (9.7,0.5) node[midway,above] {$+1$};
				\begin{scope}[xshift=10cm]			
					\smlv (v11) at (0,0) {};
					\smlv (v12) at (0.5,0) {};
					\smlv (v13) at (1,0) {};

					\smlv (v21) at (0,0.5) {};
					\smlv (v22) at (0.5,0.5) {};
					\smlv (v23) at (1,0.5) {};

					\smlv (v31) at (0,1) {};
					
					\path
					(v11) edge (v12)
					;
					\path[bend right]
					(v13) edge (v31)
					(v23) edge (v21)
					;
				\end{scope}
				
				\draw [-to,red](11.3,0.5) -- (11.7,0.5) node[midway,above] {$-1$};
				\begin{scope}[xshift=12cm]
					\node at (0.5,-0.5) {$\sigma_2$};
					
					\smlv (v11) at (0,0) {};
					\smlv (v12) at (0.5,0) {};
					\smlv (v13) at (1,0) {};

					\smlv (v21) at (0,0.5) {};
					\smlv (v22) at (0.5,0.5) {};
					\smlv (v23) at (1,0.5) {};

					\smlv (v31) at (0,1) {};
					
					\path
					(v11) edge (v12)
					;
					\path[bend left]
					(v21) edge (v23)
					;
				\end{scope}
			\end{scope}
			
			\draw [-to,red](-0.7,0) -- (-0.3,-0.5) node[midway,below left] {$+1$};
			\begin{scope}[yshift=-1.2cm]
				\begin{scope}
					\node at (0.5,-0.5) {$\psi_{22}$};
					
					\smlv (v11) at (0,0) {};
					\smlv (v12) at (0.5,0) {};
					\smlv (v13) at (1,0) {};

					\smlv (v21) at (0,0.5) {};
					\smlv (v22) at (0.5,0.5) {};
					\smlv (v23) at (1,0.5) {};

					\smlv (v31) at (0,1) {};
					
					\path
					(v12) edge (v21)
					(v23) edge (v31)
					;
				\end{scope}
				
				\draw [to reversed-to,red](1.3,0.5) -- (1.7,0.5) node[midway,above] {$+1$};
				\begin{scope}[xshift=2cm]			
					\smlv (v11) at (0,0) {};
					\smlv (v12) at (0.5,0) {};
					\smlv (v13) at (1,0) {};

					\smlv (v21) at (0,0.5) {};
					\smlv (v22) at (0.5,0.5) {};
					\smlv (v23) at (1,0.5) {};

					\smlv (v31) at (0,1) {};
					
					\path
					(v12) edge (v21)
					(v23) edge (v31)
					;
					\path[bend right]
					(v11) edge (v13)
					;
				\end{scope}
				
				\draw [-to,red](3.3,0.5) -- (3.7,0.5) node[midway,above] {$-1$};
				\begin{scope}[xshift=4cm]			
					\smlv (v11) at (0,0) {};
					\smlv (v12) at (0.5,0) {};
					\smlv (v13) at (1,0) {};

					\smlv (v21) at (0,0.5) {};
					\smlv (v22) at (0.5,0.5) {};
					\smlv (v23) at (1,0.5) {};

					\smlv (v31) at (0,1) {};
					
					\path
					(v23) edge (v31)
					;
					\path[bend right]
					(v11) edge (v13)
					;
				\end{scope}
				
				\draw [to reversed-to,red](5.3,0.5) -- (5.7,0.5) node[midway,above] {$-1$};
				\begin{scope}[xshift=6cm]			
					\smlv (v11) at (0,0) {};
					\smlv (v12) at (0.5,0) {};
					\smlv (v13) at (1,0) {};

					\smlv (v21) at (0,0.5) {};
					\smlv (v22) at (0.5,0.5) {};
					\smlv (v23) at (1,0.5) {};

					\smlv (v31) at (0,1) {};
					
					\path
					(v21) edge (v22)
					(v23) edge (v31)
					;
					\path[bend right]
					(v11) edge (v13)
					;
				\end{scope}
				
				\draw [-to,red](7.3,0.5) -- (7.7,0.5) node[midway,above] {$+1$};
				\begin{scope}[xshift=8cm]	
					\node at (0.5,-0.5) {$\sigma_3$};
					
					\smlv (v11) at (0,0) {};
					\smlv (v12) at (0.5,0) {};
					\smlv (v13) at (1,0) {};

					\smlv (v21) at (0,0.5) {};
					\smlv (v22) at (0.5,0.5) {};
					\smlv (v23) at (1,0.5) {};

					\smlv (v31) at (0,1) {};
					
					\path
					(v21) edge (v22)
					;
					\path[bend right]
					(v11) edge (v13)
					;
				\end{scope}
			\end{scope}
			\end{scope}
		\end{tikzpicture}
		\caption{Only two possible $\M^*$-paths that start from a $1$-simplex contained in $\eta_2$, and end at a critical $1$-simplex.}\label{eta2}
	\end{figure}
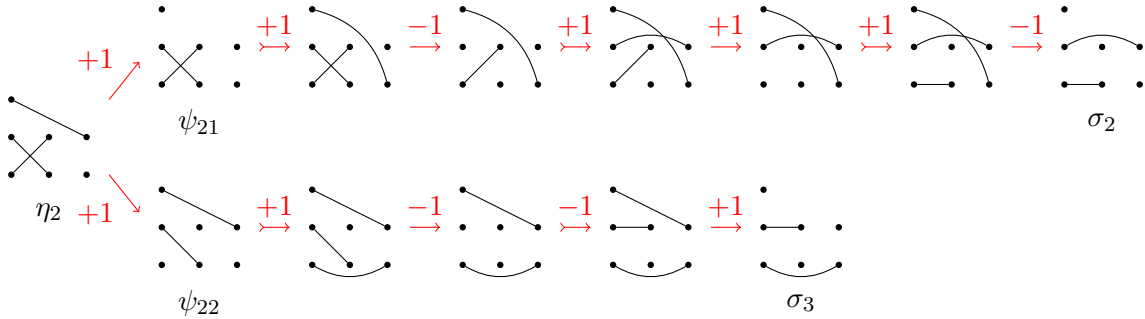
	\begin{itemize}
		\item Let $P_1$ be the path from $\eta_2$ to $\sigma_2$ via $\psi_{21}$ in Fig.~\ref{eta2}. There are two $-1$'s appearing in $P_1$ and there are three non-critical $2$-simplices in $P_1$. Thus, $P_1$ contributes $(-1)^{2+3}\sigma_2=-\sigma_2$ to $\tilde{\partial}_2(\eta_2)$.
		\item Let $P_2$ be the path from $\eta_2$ to $\sigma_3$ via $\psi_{22}$ in Fig.~\ref{eta2}. There are two $-1$'s appearing in $P_2$ and there are two non-critical $2$-simplices in $P_2$. Thus, $P_2$ contributes $(-1)^{2+2}\sigma_3=\sigma_3$ to $\tilde{\partial}_2(\eta_2)$.
	\end{itemize}
	Therefore, $\tilde{\partial}_2(\eta_2) = -\sigma_2 + \sigma_3$.
\end{example}

\begin{example}
	Let us consider the critical $2$-simplex $\eta_3 \coloneqq $ $\{v_1^{(1)}v_3^{(2)}$, $v_2^{(1)}v_2^{(2)}$, $v_3^{(1)}v_1^{(2)}\}$ as shown below.
	\[\begin{tikzpicture}		
		\node at (-0.7,0.5) {$\eta_3 = $};
		
		\smlv (v11) at (0,0) {};
		\smlv (v12) at (0.5,0) {};
		\smlv (v13) at (1,0) {};

		\smlv (v21) at (0,0.5) {};
		\smlv (v22) at (0.5,0.5) {};
		\smlv (v23) at (1,0.5) {};

		\smlv (v31) at (0,1) {};
		
		\path
		(v11) edge (v23)
		(v12) edge (v22)
		(v13) edge (v21)
		;
	\end{tikzpicture}\]
	
	We observe that there are five $\M^*$-paths that start from a $1$-simplex contained in $\eta_3$, and end at a critical $1$-simplex (Fig.~\ref{tau13}, Fig.~\ref{tau23}, Fig.~\ref{tau33}, and Fig.~\ref{tau42} in Appendix~\ref{append1}) as shown in Fig.~\ref{eta3} (along with the incidence numbers).
	\begin{figure}[!ht]
		\centering
		\begin{tikzpicture}
			\begin{scope}[scale=0.96, transform shape]
			\begin{scope}[xshift=-2cm]
				\node at (0.5,-0.5) {$\eta_3$};
				
				\smlv (v11) at (0,0) {};
				\smlv (v12) at (0.5,0) {};
				\smlv (v13) at (1,0) {};

				\smlv (v21) at (0,0.5) {};
				\smlv (v22) at (0.5,0.5) {};
				\smlv (v23) at (1,0.5) {};

				\smlv (v31) at (0,1) {};
				
				\path
				(v11) edge (v23)
				(v12) edge (v22)
				(v13) edge (v21)
				;
			\end{scope}
			
			\draw [-to,red](-1.5,1.2) -- (-0.3,4) node[midway,above left] {$+1$};
			\begin{scope}[yshift=4cm]
				\begin{scope}
					\node at (0.5,-0.5) {$\psi_{31}$};
					
					\smlv (v11) at (0,0) {};
					\smlv (v12) at (0.5,0) {};
					\smlv (v13) at (1,0) {};

					\smlv (v21) at (0,0.5) {};
					\smlv (v22) at (0.5,0.5) {};
					\smlv (v23) at (1,0.5) {};

					\smlv (v31) at (0,1) {};
					
					\path
					(v11) edge (v23)
					(v12) edge (v22)
					;
				\end{scope}
				
				\draw [to reversed-to,red](1.3,0.5) -- (1.7,0.5) node[midway,above] {$+1$};
				\begin{scope}[xshift=2cm]			
					\smlv (v11) at (0,0) {};
					\smlv (v12) at (0.5,0) {};
					\smlv (v13) at (1,0) {};

					\smlv (v21) at (0,0.5) {};
					\smlv (v22) at (0.5,0.5) {};
					\smlv (v23) at (1,0.5) {};

					\smlv (v31) at (0,1) {};
					
					\path
					(v11) edge (v23)
					(v12) edge (v22)
					;
					\path[bend right]
					(v13) edge (v31)
					;
				\end{scope}
				
				\draw [-to,red](3.3,0.5) -- (3.7,0.5) node[midway,above] {$-1$};
				\begin{scope}[xshift=4cm]			
					\smlv (v11) at (0,0) {};
					\smlv (v12) at (0.5,0) {};
					\smlv (v13) at (1,0) {};

					\smlv (v21) at (0,0.5) {};
					\smlv (v22) at (0.5,0.5) {};
					\smlv (v23) at (1,0.5) {};

					\smlv (v31) at (0,1) {};
					
					\path
					(v11) edge (v23)
					;
					\path[bend right]
					(v13) edge (v31)
					;
				\end{scope}
				
				\draw [to reversed-to,red](5.3,0.5) -- (5.7,0.5) node[midway,above] {$+1$};
				\begin{scope}[xshift=6cm]			
					\smlv (v11) at (0,0) {};
					\smlv (v12) at (0.5,0) {};
					\smlv (v13) at (1,0) {};

					\smlv (v21) at (0,0.5) {};
					\smlv (v22) at (0.5,0.5) {};
					\smlv (v23) at (1,0.5) {};

					\smlv (v31) at (0,1) {};
					
					\path
					(v11) edge (v23)
					(v21) edge (v22)
					;
					\path[bend right]
					(v13) edge (v31)
					;
				\end{scope}
				
				\draw [-to,red](7.3,0.5) -- (7.7,0.5) node[midway,above] {$+1$};
				\begin{scope}[xshift=8cm]			
					\smlv (v11) at (0,0) {};
					\smlv (v12) at (0.5,0) {};
					\smlv (v13) at (1,0) {};

					\smlv (v21) at (0,0.5) {};
					\smlv (v22) at (0.5,0.5) {};
					\smlv (v23) at (1,0.5) {};

					\smlv (v31) at (0,1) {};
					
					\path
					(v21) edge (v22)
					;
					\path[bend right]
					(v13) edge (v31)
					;
				\end{scope}
				
				\draw [to reversed-to,red](9.3,0.5) -- (9.7,0.5) node[midway,above] {$+1$};
				\begin{scope}[xshift=10cm]			
					\smlv (v11) at (0,0) {};
					\smlv (v12) at (0.5,0) {};
					\smlv (v13) at (1,0) {};

					\smlv (v21) at (0,0.5) {};
					\smlv (v22) at (0.5,0.5) {};
					\smlv (v23) at (1,0.5) {};

					\smlv (v31) at (0,1) {};
					
					\path				
					(v11) edge (v12)
					(v21) edge (v22)
					;
					\path[bend right]
					(v13) edge (v31)
					;
				\end{scope}
				
				\draw [-to,red](11.3,0.5) -- (11.7,0.5) node[midway,above] {$-1$};
				\begin{scope}[xshift=12cm]
					\node at (0.5,-0.5) {$\sigma_1$};
					
					\smlv (v11) at (0,0) {};
					\smlv (v12) at (0.5,0) {};
					\smlv (v13) at (1,0) {};

					\smlv (v21) at (0,0.5) {};
					\smlv (v22) at (0.5,0.5) {};
					\smlv (v23) at (1,0.5) {};

					\smlv (v31) at (0,1) {};
					
					\path
					(v11) edge (v12)
					(v21) edge (v22)
					;
				\end{scope}
			\end{scope}

			\draw [-to,red](-1.2,1) -- (-0.3,2) node[midway, right] {$+1$};
			\begin{scope}[yshift=2cm]
				\begin{scope}
					\node at (0.5,-0.5) {$\psi_{31}$};
					
					\smlv (v11) at (0,0) {};
					\smlv (v12) at (0.5,0) {};
					\smlv (v13) at (1,0) {};

					\smlv (v21) at (0,0.5) {};
					\smlv (v22) at (0.5,0.5) {};
					\smlv (v23) at (1,0.5) {};

					\smlv (v31) at (0,1) {};
					
					\path
					(v11) edge (v23)
					(v12) edge (v22)
					;
				\end{scope}
				
				\draw [to reversed-to,red](1.3,0.5) -- (1.7,0.5) node[midway,above] {$+1$};
				\begin{scope}[xshift=2cm]			
					\smlv (v11) at (0,0) {};
					\smlv (v12) at (0.5,0) {};
					\smlv (v13) at (1,0) {};

					\smlv (v21) at (0,0.5) {};
					\smlv (v22) at (0.5,0.5) {};
					\smlv (v23) at (1,0.5) {};

					\smlv (v31) at (0,1) {};
					
					\path
					(v11) edge (v23)
					(v12) edge (v22)
					;
					\path[bend right]
					(v13) edge (v31)
					;
				\end{scope}
				
				\draw [-to,red](3.3,0.5) -- (3.7,0.5) node[midway,above] {$+1$};
				\begin{scope}[xshift=4cm]			
					\smlv (v11) at (0,0) {};
					\smlv (v12) at (0.5,0) {};
					\smlv (v13) at (1,0) {};

					\smlv (v21) at (0,0.5) {};
					\smlv (v22) at (0.5,0.5) {};
					\smlv (v23) at (1,0.5) {};

					\smlv (v31) at (0,1) {};
					
					\path
					(v12) edge (v22)
					;
					\path[bend right]
					(v13) edge (v31)
					;
				\end{scope}
				
				\draw [to reversed-to,red](5.3,0.5) -- (5.7,0.5) node[midway,above] {$+1$};
				\begin{scope}[xshift=6cm]			
					\smlv (v11) at (0,0) {};
					\smlv (v12) at (0.5,0) {};
					\smlv (v13) at (1,0) {};

					\smlv (v21) at (0,0.5) {};
					\smlv (v22) at (0.5,0.5) {};
					\smlv (v23) at (1,0.5) {};

					\smlv (v31) at (0,1) {};
					
					\path
					(v12) edge (v22)
					;
					\path[bend right]
					(v23) edge (v21)
					(v13) edge (v31)
					;
				\end{scope}
				
				\draw [-to,red](7.3,0.5) -- (7.7,0.5) node[midway,above] {$+1$};
				\begin{scope}[xshift=8cm]			
					\smlv (v11) at (0,0) {};
					\smlv (v12) at (0.5,0) {};
					\smlv (v13) at (1,0) {};

					\smlv (v21) at (0,0.5) {};
					\smlv (v22) at (0.5,0.5) {};
					\smlv (v23) at (1,0.5) {};

					\smlv (v31) at (0,1) {};
					
					\path[bend right]
					(v23) edge (v21)
					(v13) edge (v31)
					;
				\end{scope}
				
				\draw [to reversed-to,red](9.3,0.5) -- (9.7,0.5) node[midway,above] {$+1$};
				\begin{scope}[xshift=10cm]			
					\smlv (v11) at (0,0) {};
					\smlv (v12) at (0.5,0) {};
					\smlv (v13) at (1,0) {};

					\smlv (v21) at (0,0.5) {};
					\smlv (v22) at (0.5,0.5) {};
					\smlv (v23) at (1,0.5) {};

					\smlv (v31) at (0,1) {};
					
					\path
					(v11) edge (v12)
					;
					\path[bend right]
					(v23) edge (v21)
					(v13) edge (v31)
					;
				\end{scope}
				
				\draw [-to,red](11.3,0.5) -- (11.7,0.5) node[midway,above] {$-1$};
				\begin{scope}[xshift=12cm]
					\node at (0.5,-0.5) {$\sigma_2$};
					
					\smlv (v11) at (0,0) {};
					\smlv (v12) at (0.5,0) {};
					\smlv (v13) at (1,0) {};

					\smlv (v21) at (0,0.5) {};
					\smlv (v22) at (0.5,0.5) {};
					\smlv (v23) at (1,0.5) {};

					\smlv (v31) at (0,1) {};
					
					\path
					(v11) edge (v12)
					;
					\path[bend right]
					(v23) edge (v21)
					;
				\end{scope}
			\end{scope}

			\draw [-to,red](-0.7,0.5) -- (-0.3,0.5) node[midway,above] {$-1$};
			\begin{scope}
				\begin{scope}
					\node at (0.5,-0.5) {$\psi_{32}$};
					
					\smlv (v11) at (0,0) {};
					\smlv (v12) at (0.5,0) {};
					\smlv (v13) at (1,0) {};

					\smlv (v21) at (0,0.5) {};
					\smlv (v22) at (0.5,0.5) {};
					\smlv (v23) at (1,0.5) {};

					\smlv (v31) at (0,1) {};
					
					\path
					(v11) edge (v23)
					(v13) edge (v21)
					;
				\end{scope}
				
				\draw [to reversed-to,red](1.3,0.5) -- (1.7,0.5) node[midway,above] {$-1$};
				\begin{scope}[xshift=2cm]			
					\smlv (v11) at (0,0) {};
					\smlv (v12) at (0.5,0) {};
					\smlv (v13) at (1,0) {};

					\smlv (v21) at (0,0.5) {};
					\smlv (v22) at (0.5,0.5) {};
					\smlv (v23) at (1,0.5) {};

					\smlv (v31) at (0,1) {};
					
					\path
					(v11) edge (v23)
					(v12) edge (v31)
					(v13) edge (v21)
					;
				\end{scope}
				
				\draw [-to,red](3.3,0.5) -- (3.7,0.5) node[midway,above] {$+1$};
				\begin{scope}[xshift=4cm]			
					\smlv (v11) at (0,0) {};
					\smlv (v12) at (0.5,0) {};
					\smlv (v13) at (1,0) {};

					\smlv (v21) at (0,0.5) {};
					\smlv (v22) at (0.5,0.5) {};
					\smlv (v23) at (1,0.5) {};

					\smlv (v31) at (0,1) {};
					
					\path
					(v11) edge (v23)
					(v12) edge (v31)
					;
				\end{scope}
				
				\draw [to reversed-to,red](5.3,0.5) -- (5.7,0.5) node[midway,above] {$+1$};
				\begin{scope}[xshift=6cm]			
					\smlv (v11) at (0,0) {};
					\smlv (v12) at (0.5,0) {};
					\smlv (v13) at (1,0) {};

					\smlv (v21) at (0,0.5) {};
					\smlv (v22) at (0.5,0.5) {};
					\smlv (v23) at (1,0.5) {};

					\smlv (v31) at (0,1) {};
					
					\path
					(v11) edge (v23)
					(v12) edge (v31)
					(v21) edge (v22)
					;
				\end{scope}
				
				\draw [-to,red](7.3,0.5) -- (7.7,0.5) node[midway,above] {$+1$};
				\begin{scope}[xshift=8cm]			
					\smlv (v11) at (0,0) {};
					\smlv (v12) at (0.5,0) {};
					\smlv (v13) at (1,0) {};

					\smlv (v21) at (0,0.5) {};
					\smlv (v22) at (0.5,0.5) {};
					\smlv (v23) at (1,0.5) {};

					\smlv (v31) at (0,1) {};
					
					\path
					(v12) edge (v31)
					(v21) edge (v22)
					;
				\end{scope}
				
				\draw [to reversed-to,red](9.3,0.5) -- (9.7,0.5) node[midway,above] {$+1$};
				\begin{scope}[xshift=10cm]			
					\smlv (v11) at (0,0) {};
					\smlv (v12) at (0.5,0) {};
					\smlv (v13) at (1,0) {};

					\smlv (v21) at (0,0.5) {};
					\smlv (v22) at (0.5,0.5) {};
					\smlv (v23) at (1,0.5) {};

					\smlv (v31) at (0,1) {};
					
					\path
					(v12) edge (v31)
					(v21) edge (v22)
					;
					\path[bend right]
					(v11) edge (v13)
					;
				\end{scope}
				
				\draw [-to,red](11.3,0.5) -- (11.7,0.5) node[midway,above] {$-1$};
				\begin{scope}[xshift=12cm]
					\node at (0.5,-0.5) {$\sigma_3$};
					
					\smlv (v11) at (0,0) {};
					\smlv (v12) at (0.5,0) {};
					\smlv (v13) at (1,0) {};

					\smlv (v21) at (0,0.5) {};
					\smlv (v22) at (0.5,0.5) {};
					\smlv (v23) at (1,0.5) {};

					\smlv (v31) at (0,1) {};
					
					\path
					(v21) edge (v22)
					;
					\path[bend right]
					(v11) edge (v13)
					;
				\end{scope}
			\end{scope}
			
			\draw [-to,red](-1.2,-0.7) -- (-0.3,-1.5) node[midway,above right] {$+1$};
			\begin{scope}[yshift=-2cm]
				\begin{scope}
					\node at (0.5,-0.5) {$\psi_{33}$};
					
					\smlv (v11) at (0,0) {};
					\smlv (v12) at (0.5,0) {};
					\smlv (v13) at (1,0) {};

					\smlv (v21) at (0,0.5) {};
					\smlv (v22) at (0.5,0.5) {};
					\smlv (v23) at (1,0.5) {};

					\smlv (v31) at (0,1) {};
					
					\path
					(v12) edge (v22)
					(v13) edge (v21)
					;
				\end{scope}
				
				\draw [to reversed-to,red](1.3,0.5) -- (1.7,0.5) node[midway,above] {$+1$};
				\begin{scope}[xshift=2cm]			
					\smlv (v11) at (0,0) {};
					\smlv (v12) at (0.5,0) {};
					\smlv (v13) at (1,0) {};

					\smlv (v21) at (0,0.5) {};
					\smlv (v22) at (0.5,0.5) {};
					\smlv (v23) at (1,0.5) {};

					\smlv (v31) at (0,1) {};
					
					\path
					(v12) edge (v22)
					(v13) edge (v21)
					;
					\path[bend left]
					(v11) edge (v31)
					;
				\end{scope}
				
				\draw [-to,red](3.3,0.5) -- (3.7,0.5) node[midway,above] {$+1$};
				\begin{scope}[xshift=4cm]			
					\smlv (v11) at (0,0) {};
					\smlv (v12) at (0.5,0) {};
					\smlv (v13) at (1,0) {};

					\smlv (v21) at (0,0.5) {};
					\smlv (v22) at (0.5,0.5) {};
					\smlv (v23) at (1,0.5) {};

					\smlv (v31) at (0,1) {};
					
					\path
					(v12) edge (v22)
					;
					\path[bend left]
					(v11) edge (v31)
					;
				\end{scope}
				
				\draw [to reversed-to,red](5.3,0.5) -- (5.7,0.5) node[midway,above] {$+1$};
				\begin{scope}[xshift=6cm]			
					\smlv (v11) at (0,0) {};
					\smlv (v12) at (0.5,0) {};
					\smlv (v13) at (1,0) {};

					\smlv (v21) at (0,0.5) {};
					\smlv (v22) at (0.5,0.5) {};
					\smlv (v23) at (1,0.5) {};

					\smlv (v31) at (0,1) {};
					
					\path
					(v12) edge (v22)
					;
					\path[bend left]
					(v11) edge (v31)
					(v21) edge (v23)
					;
				\end{scope}
				
				\draw [-to,red](7.3,0.5) -- (7.7,0.5) node[midway,above] {$+1$};
				\begin{scope}[xshift=8cm]			
					\smlv (v11) at (0,0) {};
					\smlv (v12) at (0.5,0) {};
					\smlv (v13) at (1,0) {};

					\smlv (v21) at (0,0.5) {};
					\smlv (v22) at (0.5,0.5) {};
					\smlv (v23) at (1,0.5) {};

					\smlv (v31) at (0,1) {};
					
					\path
					(v12) edge (v22)
					;
					\path[bend left]
					(v21) edge (v23)
					;
				\end{scope}
				
				\draw [to reversed-to,red](9.3,0.5) -- (9.7,0.5) node[midway,above] {$+1$};
				\begin{scope}[xshift=10cm]			
					\smlv (v11) at (0,0) {};
					\smlv (v12) at (0.5,0) {};
					\smlv (v13) at (1,0) {};

					\smlv (v21) at (0,0.5) {};
					\smlv (v22) at (0.5,0.5) {};
					\smlv (v23) at (1,0.5) {};

					\smlv (v31) at (0,1) {};
					
					\path
					(v12) edge (v22)
					;
					\path[bend left]
					(v13) edge (v11)
					(v21) edge (v23)
					;
				\end{scope}
				
				\draw [-to,red](11.3,0.5) -- (11.7,0.5) node[midway,above] {$-1$};
				\begin{scope}[xshift=12cm]		
					\node at (0.5,-0.5) {$\sigma_4$};
					
					\smlv (v11) at (0,0) {};
					\smlv (v12) at (0.5,0) {};
					\smlv (v13) at (1,0) {};

					\smlv (v21) at (0,0.5) {};
					\smlv (v22) at (0.5,0.5) {};
					\smlv (v23) at (1,0.5) {};

					\smlv (v31) at (0,1) {};
					
					\path[bend left]
					(v13) edge (v11)
					(v21) edge (v23)
					;
				\end{scope}
				
			\end{scope}
			
			\draw [-to,red](-1.5,-1) -- (-0.3,-3.5) node[midway,below left] {$+1$};
			\begin{scope}[yshift=-4cm]
				\begin{scope}
					\node at (0.5,-0.5) {$\psi_{31}$};
					
					\smlv (v11) at (0,0) {};
					\smlv (v12) at (0.5,0) {};
					\smlv (v13) at (1,0) {};

					\smlv (v21) at (0,0.5) {};
					\smlv (v22) at (0.5,0.5) {};
					\smlv (v23) at (1,0.5) {};

					\smlv (v31) at (0,1) {};
					
					\path
					(v11) edge (v23)
					(v12) edge (v22)
					;
				\end{scope}
				
				\draw [to reversed-to,red](1.3,0.5) -- (1.7,0.5) node[midway,above] {$+1$};
				\begin{scope}[xshift=2cm]			
					\smlv (v11) at (0,0) {};
					\smlv (v12) at (0.5,0) {};
					\smlv (v13) at (1,0) {};

					\smlv (v21) at (0,0.5) {};
					\smlv (v22) at (0.5,0.5) {};
					\smlv (v23) at (1,0.5) {};

					\smlv (v31) at (0,1) {};
					
					\path
					(v11) edge (v23)
					(v12) edge (v22)
					;
					\path[bend right]
					(v13) edge (v31)
					;
				\end{scope}
				
				\draw [-to,red](3.3,0.5) -- (3.7,0.5) node[midway,above] {$+1$};
				\begin{scope}[xshift=4cm]			
					\smlv (v11) at (0,0) {};
					\smlv (v12) at (0.5,0) {};
					\smlv (v13) at (1,0) {};

					\smlv (v21) at (0,0.5) {};
					\smlv (v22) at (0.5,0.5) {};
					\smlv (v23) at (1,0.5) {};

					\smlv (v31) at (0,1) {};
					
					\path
					(v12) edge (v22)
					;
					\path[bend right]
					(v13) edge (v31)
					;
				\end{scope}
				
				\draw [to reversed-to,red](5.3,0.5) -- (5.7,0.5) node[midway,above] {$+1$};
				\begin{scope}[xshift=6cm]			
					\smlv (v11) at (0,0) {};
					\smlv (v12) at (0.5,0) {};
					\smlv (v13) at (1,0) {};

					\smlv (v21) at (0,0.5) {};
					\smlv (v22) at (0.5,0.5) {};
					\smlv (v23) at (1,0.5) {};

					\smlv (v31) at (0,1) {};
					
					\path
					(v12) edge (v22)
					;
					\path[bend right]
					(v13) edge (v31)
					(v23) edge (v21)
					;
				\end{scope}
				
				\draw [-to,red](7.3,0.5) -- (7.7,0.5) node[midway,above] {$-1$};
				\begin{scope}[xshift=8cm]			
					\smlv (v11) at (0,0) {};
					\smlv (v12) at (0.5,0) {};
					\smlv (v13) at (1,0) {};

					\smlv (v21) at (0,0.5) {};
					\smlv (v22) at (0.5,0.5) {};
					\smlv (v23) at (1,0.5) {};

					\smlv (v31) at (0,1) {};
					
					\path
					(v12) edge (v22)
					;
					\path[bend left]
					(v21) edge (v23)
					;
				\end{scope}
				
				\draw [to reversed-to,red](9.3,0.5) -- (9.7,0.5) node[midway,above] {$+1$};
				\begin{scope}[xshift=10cm]			
					\smlv (v11) at (0,0) {};
					\smlv (v12) at (0.5,0) {};
					\smlv (v13) at (1,0) {};

					\smlv (v21) at (0,0.5) {};
					\smlv (v22) at (0.5,0.5) {};
					\smlv (v23) at (1,0.5) {};

					\smlv (v31) at (0,1) {};
					
					\path
					(v12) edge (v22)
					;
					\path[bend left]
					(v13) edge (v11)
					(v21) edge (v23)
					;
				\end{scope}
				
				\draw [-to,red](11.3,0.5) -- (11.7,0.5) node[midway,above] {$-1$};
				\begin{scope}[xshift=12cm]		
					\node at (0.5,-0.5) {$\sigma_4$};
					
					\smlv (v11) at (0,0) {};
					\smlv (v12) at (0.5,0) {};
					\smlv (v13) at (1,0) {};

					\smlv (v21) at (0,0.5) {};
					\smlv (v22) at (0.5,0.5) {};
					\smlv (v23) at (1,0.5) {};

					\smlv (v31) at (0,1) {};
					
					\path[bend left]
					(v13) edge (v11)
					(v21) edge (v23)
					;
				\end{scope}
				
			\end{scope}
			
		\end{scope}
		\end{tikzpicture}
		\caption{All five $\M^*$-paths that start from a $1$-simplex contained in $\eta_3$, and end at a critical $1$-simplex.}\label{eta3}
	\end{figure}
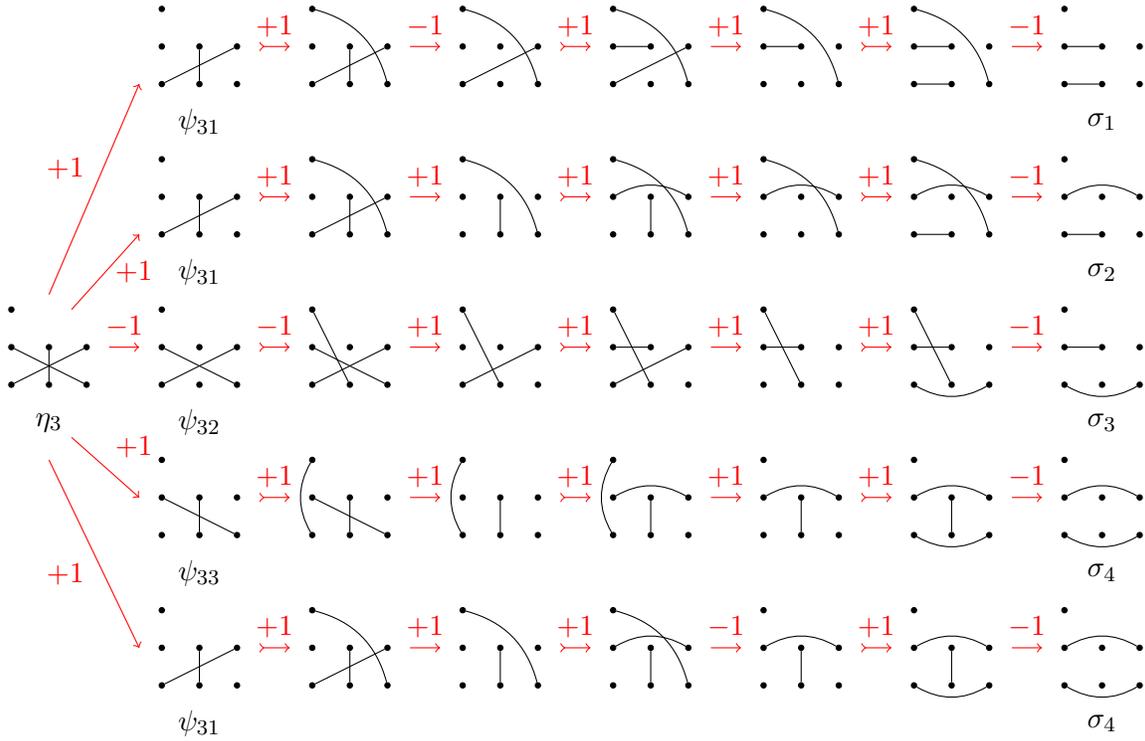
	\begin{itemize}
		\item Let $P_1$ be the path from $\eta_3$ to $\sigma_1$ via $\psi_{31}$ in Fig.~\ref{eta3}. There are two $-1$'s appearing in $P_1$ and there are three non-critical $2$-simplices in $P_1$. Thus, $P_1$ contributes $(-1)^{2+3}\sigma_1=-\sigma_1$ to $\tilde{\partial}_2(\eta_3)$.
		\item Let $P_2$ be the path from $\eta_3$ to $\sigma_2$ via $\psi_{31}$ in Fig.~\ref{eta3}. There is only one $-1$ appearing in $P_2$ and there are three non-critical $2$-simplices in $P_2$. Thus, $P_2$ contributes $(-1)^{1+3}\sigma_2=\sigma_2$ to $\tilde{\partial}_2(\eta_3)$.
		\item Let $P_3$ be the path from $\eta_3$ to $\sigma_3$ via $\psi_{32}$ in Fig.~\ref{eta3}. There are three $-1$'s appearing in $P_3$ and there are three non-critical $2$-simplices in $P_3$. Thus, $P_3$ contributes $(-1)^{3+3}\sigma_3=\sigma_3$ to $\tilde{\partial}_2(\eta_3)$.
		\item Let $P_4$ be the path from $\eta_3$ to $\sigma_4$ via $\psi_{33}$ in Fig.~\ref{eta3}. There is only one $-1$ appearing in $P_4$ and there are three non-critical $2$-simplices in $P_4$. Thus, $P_4$ contributes $(-1)^{1+3}\sigma_4=\sigma_4$ to $\tilde{\partial}_2(\eta_3)$.
		\item Let $P_5$ be the path from $\eta_3$ to $\sigma_4$ via $\psi_{31}$ in Fig.~\ref{eta3}. There are two $-1$'s appearing in $P_5$ and there are three non-critical $2$-simplices in $P_5$. Thus, $P_5$ contributes $(-1)^{2+3}\sigma_4=-\sigma_4$ to $\tilde{\partial}_2(\eta_3)$.
	\end{itemize}
	Therefore, $\tilde{\partial}_2(\eta_3) = -\sigma_1 + \sigma_2 + \sigma_3 + \sigma_4 - \sigma_4 = -\sigma_1 + \sigma_2 + \sigma_3$.
\end{example}